\documentclass{article}
\usepackage{amsmath}
\usepackage{amsthm}
\usepackage{graphics}
\usepackage{epsfig}
\usepackage{stmaryrd}

\newtheorem{theorem}{Theorem}[section]

\newtheorem{lem}[theorem]{Lemma}
\newtheorem{prop}[theorem]{Proposition}

\theoremstyle{definition}

\newtheorem{rem}[theorem]{Remark}

\numberwithin{equation}{section}

\usepackage{amssymb}
\usepackage{amsfonts}
\usepackage{psfrag}
\usepackage{url}
\usepackage{mathtools}
\usepackage{color}
\usepackage{tikz-cd}

\usepackage[makeroom]{cancel}

\usepackage{stmaryrd}
\DeclareMathOperator*{\bigtimes}{\vartimes}

\newcommand{\K}{\mathbb{K}}
\newcommand{\N}{\mathbb{N}}
\newcommand{\R}{\mathbb{R}}

\newcommand{\Z}{\mathbb{Z}}
\newcommand{\Q}{\mathbb{Q}}
\newcommand{\C}{\mathbb{C}}

\newcommand{\cC}{\mathcal{C}}
\newcommand{\cH}{\mathcal{H}}
\newcommand{\cK}{\mathcal{K}}
\newcommand{\cL}{\mathcal{L}}

\newcommand{\cR}{\mathcal{R}}

\newcommand{\cG}{\mathcal{G}}

\begin{document}

\title{On Lipschitz equivalence of finite-dimensional linear flows}
\author{Arno Berger and Anthony Wynne}

\date{\today}
\maketitle

\begin{abstract}
\noindent
Two flows on a finite-dimensional normed space $X$ are Lipschitz
equivalent if some homeomorphism $h$ of $X$ that is bi-Lipschitz near the
origin preserves all orbits, i.e., $h$ maps each orbit onto an
orbit. A complete classification by Lipschitz equivalence is
established for all linear flows on $X$, in terms of basic linear
algebra properties of their generators. Utilizing equivalence instead
of the much more restrictive conjugacy, the classification theorem
significantly extends known results. The analysis is entirely
elementary though somewhat intricate. It highlights, more clearly than
does the existing literature, the fundamental roles played by
linearity and finite-dimensionality.   
\end{abstract}
\hspace*{6.6mm}{\small {\bf Keywords.} Equivalence between flows,
  linear flow, Lipschitz similarity, distortion points.}

\noindent
\hspace*{6.6mm}{\small {\bf MSC2020.} 34A30, 34C41, 34D08, 37C15.}


\section{Introduction}\label{sec1}

Let $X\ne \{0\}$ be a finite-dimensional normed space over $\R$ and $\varphi$
a flow on $X$, i.e., $\varphi : \R \times X \to X$ is
continuous with $\varphi(t+s , x) = \varphi \bigl( t, \varphi (s,x)
\bigr)$ and $\varphi(0,x) = x$ for all $t,s\in \R$, $x\in X$. A
fundamental question throughout dynamics is that of classification:
When, precisely, are two flows $\varphi$, $\psi$ on $X$ {\em the
  same\/}? Taking a geometrically motivated approach to this question,
say that $\varphi$, $\psi$ are {\bf 
  equivalent}, in symbols $\varphi \stackrel{0}{\thicksim} \psi$, if there exists a
homeomorphism $h:X\to X$ with $h(0)=0$ that maps each $\varphi$-orbit
$\varphi_{\R}(x):= \{\varphi (t, x):t\in \R\}$ onto
a $\psi$-orbit, i.e., 
\begin{equation}\label{eq1_1}
  h \bigl( \varphi_{\R} (x) \bigr) = \psi_{\R} \bigl( h(x)\bigr)  \qquad
  \forall x \in X \, .
\end{equation}
If $h$, $h^{-1}$ both are H\"{o}lder continuous (or Lipschitz
continuous, differentiable, linear) then $\varphi$, $\psi$ are
said to be {\bf H\"{o}lder} (or {\bf Lipschitz}, {\bf differentiably},
{\bf linearly}) {\bf equivalent}, in symbols $\varphi 
\stackrel{1^-}{\thicksim} \psi$ (or $\varphi
\stackrel{1}{\thicksim} \psi$, $\varphi
\stackrel{{\sf diff}}{\thicksim} \psi$, $\varphi
\stackrel{{\sf lin}}{\thicksim} \psi$). Plainly $\stackrel{\bigstar}{\thicksim}$
yields an equivalence relation 
for each $\bigstar\in \{0,1^-,1, {\sf diff}, {\sf lin}\}$, thereby establishing
five natural classifications of all flows on $X$. In addition,
consider the following, much more
restrictive form of equivalence: Say that $\varphi$, $\psi$ are {\bf conjugate}, in symbols
$\varphi \stackrel{0}{\cong} \psi$, if 
\begin{equation}\label{eq1_2}
  h \bigl( \varphi  (t,x) \bigr) = \psi \bigl( t,h(x)\bigr)  \qquad
  \forall t\in \R , x \in X \, .
\end{equation}
Notice that (\ref{eq1_2}) implies (\ref{eq1_1}) but most definitely
not vice versa, and analogously defined
$\stackrel{\bigstar}{\cong}$ again yields an equivalence relation for
each $\bigstar \in \{0, 1^-, 1, {\sf diff}, {\sf lin}\}$. Simple examples
show that no two of these equivalences, or the
classification established by them, coincide, not even when $\mbox{\rm
  dim}\, X=1$.

Building on the classical literature briefly reviewed below, the
present article concludes the authors' earlier work \cite{BW, BW2} 
by carrying out a comprehensive analysis of Lipschitz equivalence and conjugacy
for {\em linear\/} flows. Recall that a flow $\varphi$ on $X$ is {\bf linear} if the time-$t$
map $\varphi_t = \varphi (t,\cdot):X\to X$ is linear, or equivalently
if $\varphi_t = e^{t A^{\varphi}}$,
for every $t\in \R$, with a (unique) linear operator $A^{\varphi}$ on
$X$ called the {\bf generator} of $\varphi$. Henceforth, upper case Greek letters
$\Phi$, $\Psi$ are used exclusively to denote linear flows. The challenge, then, is to
characterize $\Phi \stackrel{1}{\thicksim} \Psi$ and $\Phi
\stackrel{1}{\cong} \Psi$ in terms of basic linear algebra properties
of $A^{\Phi}$, $A^{\Psi}$. Correspondingly the main result of this
article, Theorem \ref{thm1x} below, can be
viewed as a {\bf Lipschitz classification theorem}. To preview the result,
note that every linear flow $\Phi$ on $X$ determines unique
decompositions $\Phi \stackrel{{\sf lin}}{\cong}  \Phi_{\sf D} \times \Phi_{\sf
  AD} 
\stackrel{{\sf lin}}{\cong} \Phi_{\sf S} \times \Phi_{\sf C} \times
\Phi_{\sf U}$ into a {\em diagonal\/}
and {\em anti-diagonal\/} part, as well as into a
{\em stable}, {\em central}, and {\em unstable\/} part; see Section \ref{sec2a} below
for formal details. For convenience denote by
$\Phi_{*\alpha}$ the linear flow generated by $\alpha A^{\Phi}$, for
any $\alpha \in \R \setminus \{0\}$. Thus, $\Phi_{*\alpha}$ simply is
$\Phi$ with its first variable (``time'') rescaled by $\alpha$. Finally, the statement involves
the concepts of {\em Lyapunov\/} and {\em Lipschitz
similarity}, introduced rigorously in Section
\ref{sec2a} also. For now, simply say that two linear operators are {\bf
  Lyapunov similar} if they (more precisely, the flows they generate) have the same
Lyapunov exponents, with matching multiplicities, and say that they
are  {\bf Lipschitz similar} if they are Lyapunov similar and their
anti-diagonal parts are similar.

\begin{theorem}\label{thm1x}
Let $\Phi$, $\Psi$ be linear flows on $X$. Then each of the following
four statements implies the other three:
\begin{enumerate}
\item $\Phi \stackrel{1}{\thicksim} \Psi$, i.e., $\Phi$, $\Psi$ are
  Lipschitz equivalent;
  \item there exists $\alpha
    \in \R \setminus \{0\}$ so that $\Phi \stackrel{1}{\cong}
    \Psi_{*\alpha}$, i.e., $\Phi$, $\Psi_{*\alpha}$ are Lipschitz conjugate;
\item there exists $\beta
    \in \R \setminus \{0\}$ so that $A^{\Phi}$, $\beta A^{\Psi}$ are
    Lyapunov similar while $A^{\Phi_{\sf  AD}}$, $\beta A^{\Psi_{\sf
        AD}}$ as well as $A^{\Phi_{\sf  C}}$, $\beta A^{\Psi_{\sf
        C}}$ are similar;
\item there exists $\gamma
    \in \R \setminus \{0\}$ so that $A^{\Phi}$, $\gamma A^{\Psi}$ are
    Lipschitz similar while $A^{\Phi_{\sf  C}}$, $\gamma A^{\Psi_{\sf
        C}}$ are similar.
  \end{enumerate}
  Moreover, $\Phi \stackrel{1}{\cong}\Psi$ if and
only if $A^{\Phi}$, $A^{\Psi}$ are Lipschitz similar while $A^{\Phi_{\sf  C}}$, $ A^{\Psi_{\sf
        C}}$ are similar.
\end{theorem}

To put Theorem \ref{thm1x} in context, it is instructive to compare it
to its differentiable (hence slightly more restrictive) and H\"{o}lder
(hence slightly less restrictive) counterparts; stated here without
proofs, these results have been presented by the authors in detail
(though sometimes couched in slightly different terminology and
notation) in \cite{BW} and \cite{BW2} respectively.

\begin{prop}\label{prop1zb}
Let $\Phi$, $\Psi$ be linear flows on $X$. Then each of the following
five statements implies the other four:
\begin{enumerate}
\item $\Phi \stackrel{\sf lin}{\thicksim} \Psi$, i.e., $\Phi$, $\Psi$ are
  linearly equivalent;
 \item $\Phi \stackrel{\sf diff}{\thicksim} \Psi$, i.e., $\Phi$, $\Psi$ are
  differentiably equivalent;
  \item there exists $\alpha
    \in \R \setminus \{0\}$ so that $\Phi \stackrel{\sf lin}{\cong}
    \Psi_{*\alpha}$, i.e., $\Phi$, $\Psi_{*\alpha}$ are linearly
    conjugate;
     \item there exists $\beta
    \in \R \setminus \{0\}$ so that $\Phi \stackrel{\sf diff}{\cong}
    \Psi_{*\beta}$, i.e., $\Phi$, $\Psi_{*\beta}$ are differentiably conjugate;
\item there exists $\gamma
    \in \R \setminus \{0\}$ so that $A^{\Phi}$, $\gamma A^{\Psi}$ are similar.
  \end{enumerate}
  Moreover, $\Phi \stackrel{\sf lin}{\cong} \Psi$ if and only if $\Phi
  \stackrel{\sf diff}{\cong} \Psi$ if and only if $A^{\Phi}$,
  $A^{\Psi}$ are similar. 
\end{prop}

\begin{prop}\label{prop1za}
Let $\Phi$, $\Psi$ be linear flows on $X$. Then each of the following
three statements implies the other two:
\begin{enumerate}
\item $\Phi \stackrel{1^-}{\thicksim} \Psi$, i.e., $\Phi$, $\Psi$ are
  H\"{o}lder equivalent;
  \item there exists $\alpha
    \in \R \setminus \{0\}$ so that $\Phi \stackrel{1^-}{\cong}
    \Psi_{*\alpha}$, i.e., $\Phi$, $\Psi_{*\alpha}$ are H\"{o}lder conjugate;
\item there exists $\beta
    \in \R \setminus \{0\}$ so that $A^{\Phi}$, $\beta A^{\Psi}$ are
    Lyapunov similar while $A^{\Phi_{\sf C}}$, $\beta A^{\Psi_{\sf
        C}}$ are similar.
  \end{enumerate}
  Moreover, $\Phi \stackrel{1^-}{\cong} \Psi$ if and only if
  $A^{\Phi}$, $A^{\Psi}$ are Lyapunov similar while $A^{\Phi_{\sf C}}$,
  $ A^{\Psi_{\sf C}}$ are similar.
\end{prop}

Classifications of linear flows have long been studied in the
literature, notably for {\em hyperbolic\/} flows, that is, for
$\Phi_{\sf C}$, $\Psi_{\sf C}$ being trivial; see, e.g., \cite{Amann, BW,
  Irwin, R} for broad context, as well as \cite{ACK1, AK, DSS, He, LZ, Willems} for
specific studies.
One striking aspect of Theorem
\ref{thm1x} is the fact that (i)$\Rightarrow$(ii). Thus, for {\em linear\/}
flows $\varphi$, $\psi$ validity of (\ref{eq1_1}) always entails
validity of (\ref{eq1_2}), up to a linear, orbit-independent rescaling
of time. This remarkable property, which does not 
seem to be shared by any wider class of flows on $X$, is indicative of
the extraordinary coherence between individual orbits of linear
flows. As far as the authors have been able to ascertain, the
property has not been stated, let alone proved rigorously before,
though it appears to have been part of linear systems folklore for
quite some time; see, e.g., \cite[Rem.\ 7.4]{Willems} as well as
\cite{ACK1, CK, KS}.

With Lipschitz continuity being a fundamental concept throughout
analysis \cite{Hei}, another key aspect of Theorem \ref{thm1x} is its
relation to the differentiable and H\"{o}lder counterparts,
Propositions \ref{prop1zb} and \ref{prop1za} respectively. While
\cite[Rem.\ 2.5]{ACK1} seems to suggest, somewhat misleadingly, that the equivalence
relation $\stackrel{1}{\cong}$ for linear flows simply coincides with $\stackrel{{\sf
    diff}}{\cong}$, the characterization of Lipschitz {\em
  conjugacy\/} contained in Theorem \ref{thm1x} has, in essence, been
established in \cite{KS}; see also \cite{MM} for a related (albeit quite
informal) discussion. Specifically, \cite{KS} argues that
$\stackrel{1}{\cong}$ is ``very close'' to $\stackrel{{\sf
    diff}}{\cong}$, and a crucial role in the argument is subsequently
played by a well-known theorem due to Rademacher which asserts that
every Lipschitz map $h:X\to X$ is differentiable almost
everywhere. However, in
the general setting of Theorem \ref{thm1x}, that is, for mere Lipschitz {\em equivalence}, no
assumptions whatsoever are made regarding the (orientation and
regularity of) re-parametrizations of individual orbits, and hence it appears
doubtful whether Rademacher's theorem can be applied
fruitfully (or at all). Instead, the present
article utilizes a simple geometric idea gleaned from \cite{MM} which
it develops into the basic but consequential concept of {\em
  distortion\/} ({\em points\/}) for stable flows. Aided by this
concept, the article then proceeds to prove, in an entirely elementary
manner, that (i)$\Rightarrow$(iii) in Theorem \ref{thm1x}. From this,
validity of the entire theorem follows rather straightforwardly, as
(iii)$\Leftrightarrow$(iv)$\Rightarrow$(ii), and obviously
(ii)$\Rightarrow$(i).  As indicated where appropriate below, the
elementary approach developed here also helps to address other gaps and
inaccuracies in the existing literature.

\medskip

\begin{rem}
An analogous classification problem presents itself in discrete time,
i.e., for linear operators $A, B:X\to X$ which are {\bf conjugate} (or
{\bf nonlinearly similar} \cite{CappSh, CSSW}), in symbols
$A\stackrel{0}{\cong} B$, if $h(Ax) = Bh(x)$ for some homeomorphism
$h:X\to X$ and all $x\in X$. As in continuous time, it is
natural to consider $\stackrel{\bigstar}{\cong}$ for $\bigstar \in
\{0,1^-,1,{\sf diff}, {\sf lin}\}$, each of which yields an
equivalence relation on all linear operators on $X$. In
straightforward analogy to Proposition \ref{prop1zb}, notice that
$A\stackrel{{\sf lin}}{\cong} B$ if and only if $A\stackrel{{\sf
    diff}}{\cong} B$ if and only if $A$, $B$ are similar. By contrast,
the problem of characterizing $A\stackrel{\bigstar}{\cong} B$ for
$\bigstar \in \{0,1^-, 1\}$, in terms of similarity invariants akin to
Theorem \ref{thm1x}, has turned out to be significantly more challenging than in 
continuous time; see, e.g., \cite{CappSh, CSSW, Cruz, HP, KR} for the
long history of the problem and its many ramifications, with
\cite{Cruz} specifically addressing the Lipschitz case.
\end{rem}

The remainder of this article is organized as follows: Section
\ref{sec2} properly introduces various forms of equivalence and
conjugacy to be studied in subsequent sections, together with some
tailor-made analytical notation. Section \ref{sec2a} briefly reviews a
few basic concepts pertinent to linear flows, notably irreducibility
and Lyapunov exponents, and then discusses the novel concept of
Lipschitz similarity. Section \ref{sec3} defines distortion points and
refined Lyapunov spaces for stable flows which in turn yield crucial Lipschitz
invariants for such flows. Section
\ref{sec4} presents a proof of the main result, Theorem \ref{thm1x},
via a series of preparatory lemmas addressing important special cases
thereof. It also outlines the straightforward extension of Theorem
\ref{thm1x} to {\em complex\/} spaces. Finally, Section \ref{sec6}
highlights several subtle aspects of Lipschitz equivalence by briefly
discussing a novel variant of the concept, referred to as {\em
  pointwise\/} Lipschitz equivalence at $0$.

\section{Analytic preliminaries}\label{sec2}

The familiar symbols $\N$, $\N_0$, $\Q^+$, $\Q$, $\R^+$,
$\R$, $\R \cup \{-\infty, \infty\}$, and $\C$ denote the sets of all
positive whole, non-negative 
whole, positive rational, rational, positive real, real, extended
real, and complex numbers respectively, each with their usual
arithmetic, order (except for $\C$), and
topology. The empty set is $\varnothing$. Every $z\in \C$ can be
written uniquely as $z=a+ib$ where 
$a= {\sf Re}\, z$, $b={\sf Im}\, z$ are real numbers, with complex conjugate
$\overline{z} = a-ib$ and modulus $|z|=\sqrt{a^2+b^2}$. Let $\C^+ =
\{z\in \C: {\sf Re}\, z>0\}$, and given any $v,w \in
\C$ and $Z\subset \C$, let $v+wZ=\{v+w z :z\in Z\}$ for convenience;
thus for example $\C^-:= - \C^+ = \{z\in \C: {\sf Re}\, z<0\}$ and $i\R =
\{ia: a\in \R\}$. As usual, the cardinality (or number of elements) of
any finite set $Z$ is denoted $\# Z$, with $\# \varnothing := 0$.

Throughout, let $X = \R^d$, where the actual value of $d\in \N$ is
either clear from the context or irrelevant. Denote by $|\cdot|$ the
Euclidean norm on $X$, as well as the norm on linear operators on $X$
induced by it; this is solely for convenience, as all concepts and
results discussed herein are independent of any particular
norm. Similarly to \cite{BW, BW2}, the case of a (finite-dimensional)
normed space over $\C$ does not pose any additional challenge and is
only commented on briefly at the end of Section \ref{sec4}. Denote by
$e_1, \ldots , e_d$ the canonical basis of $X$, by $O_X=O_d$ and
$I_X=I_d$ the zero and identity operator (or $d\times d$-matrix) on
$X$ respectively, and let $B_r(x) = \{y\in X : |y-x|<r\}$ for every
$r\in \R^+$, $x\in X$. 

Given two flows $\varphi, \psi$ on $X$ and a homeomorphism $h:X\to X$, say
that $\varphi$ is $h$-{\bf related} to $\psi$, in symbols $\varphi
\stackrel{h}{\thicksim} \psi$, if (\ref{eq1_1}) holds,
or equivalently if $h$, $h^{-1}$ both map orbits into orbits. An
orbit-wise characterization of $\varphi
\stackrel{h}{\thicksim} \psi$ is readily established; see, e.g.,
\cite[Sec.\ 2]{BW2}.

\begin{prop}\label{prop12}
Let $\varphi$, $\psi$ be flows on $X$. For every homeomorphism $h:X\to
X$ the following statements are equivalent:
\begin{enumerate}
\item $\varphi \stackrel{h}{\thicksim} \psi$;
\item for every $x\in X$ there exists a continuous bijection
  $\tau_x:\R\to \R$ with $\tau_x (0)=0$ so that
  $$h \bigl( \varphi_t (x)\bigr) = \psi_{\tau_x(t)} \bigl( h(x) \bigr)
  \qquad \forall t \in \R \, .
  $$ 
\end{enumerate}
\end{prop}

Proposition \ref{prop12} motivates the simplest, most
fundamental form of equivalence between flows: Say
that $\varphi$, $\psi$ are {\bf equivalent}, in symbols $\varphi
\thicksim \psi$, if $\varphi \stackrel{h}{\thicksim} \psi$ for some
homeomorphism $h$. Informally put, $\varphi \thicksim \psi$ means that every
$\varphi$-orbit is, up to a change of spatial coordinates (via $h$)
and a (possibly orbit-dependent) re-parametrization of time (via
$\tau_x$), also a $\psi$-orbit and vice versa.

A natural way of refining $ \varphi  \stackrel{h}{\thicksim}
\psi  $ is to require additional regularity of $h$. Note that if $ \varphi
\stackrel{h}{\thicksim} \psi  $ then also $ \varphi
\stackrel{\widetilde{h}}{\thicksim} \widetilde{\psi} $, where
$\widetilde{h} = h - h(0)$ and $\widetilde{\psi}_t = \psi_t \bigl( \cdot
+ h(0)\bigr) - h(0)$ for all $t\in \R$. Thus no generality is lost by
assuming that $h(0)=0$. Bearing this in mind, denote by $\cH= \cH (X)$ the set of all
homeomorphisms $h:X\to X$ with $h(0)=0$, and let $\cH_{\beta}=
\cH_{\beta} (X)$ with $0\le \beta \le 1$ be the set of all $h\in \cH$ for which $h$,
$h^{-1}$ both satisfy a $\beta$-H\"{o}lder condition (a.k.a.\
Lipschitz condition in case $\beta=1$) near $0$, i.e.,
\begin{equation}\label{eq2_0}
\cH_{\beta} = \left\{ h \in \cH : \exists r \in \R^+ \: \mbox{\rm
    s.t.} \: \sup\nolimits_{x,y\in B_r(0), x\ne y}  \frac{ |h(x) -
    h(y)| +  |h^{-1}(x) - h^{-1}(y)| }{|x-y|^{\beta}} < \infty \right\}  \, .
\end{equation}
Thus $\cH_0 = \cH$. Since $\beta \mapsto \cH_{\beta}$ is decreasing,
also consider $\cH_{1^-}:= \bigcap_{ 0\le \beta<1} \cH_{\beta}$. Furthermore, let
$$
\cH_{\sf diff} = \bigl\{ h\in \cH : h, h^{-1} \enspace \mbox{\rm are
  differentiable at } 0\, \bigr\} \, , \qquad
\cH_{\sf lin} = \bigl\{h\in \cH : h \enspace \mbox{\rm is
 linear}\,   \bigr\} \, .
$$
This yields five sets $\cH_{\bigstar}$, where $\bigstar \in \{0,1^-,1,
{\sf diff}, {\sf lin}\}$, and plainly
$\cH_0 \supset \cH_{1^-}\supset \cH_1 \supset \cH_{\sf lin} $ as well
as $\cH_0 \supset \cH_{\sf diff} \supset \cH_{\sf  lin}$, 
with each inclusion being strict, 
whereas $\cH_{1^-}\not \supset \cH_{\sf diff}$ and $\cH_{\sf diff}\not
\supset \cH_{1}$. Understand $\varphi
\stackrel{\bigstar }{\thicksim} \psi$ to mean that $\varphi
\stackrel{h}{\thicksim} \psi$ for some $h\in \cH_{\bigstar}$. Clearly,
$\stackrel{\bigstar}{\thicksim}$ yields an equivalence relation
between flows on $X$. Say that
$\varphi$, $\psi$ are {\bf topologically}, {\bf H\"{older}}, {\bf Lipschitz}, {\bf
  differentiably}, and {\bf linearly equivalent} if $\varphi
\stackrel{0}{\thicksim} \psi$, $\varphi
\stackrel{1^-}{\thicksim} \psi$, $\varphi
\stackrel{1}{\thicksim} \psi$, $\varphi
\stackrel{{\sf diff}}{\thicksim} \psi$, and
$\varphi \stackrel{{\sf lin}}{\thicksim} \psi$ respectively. With this, 
\begin{equation}\label{eq13}
\varphi \stackrel{{\sf lin}}{\thicksim} \psi \enspace \Longrightarrow
\enspace
\varphi
\stackrel{1}{\thicksim} \psi
\enspace \Longrightarrow
\enspace \varphi
\stackrel{1^-}{\thicksim} \psi
 \enspace \Longrightarrow
\enspace \varphi
\stackrel{0}{\thicksim} \psi \, ,
\qquad 
\varphi \stackrel{{\sf lin}}{\thicksim} \psi \enspace \Longrightarrow
\enspace
\varphi
\stackrel{{\sf diff}}{\thicksim} \psi \enspace \Longrightarrow
\enspace \varphi
\stackrel{0}{\thicksim} \psi \, ,
\end{equation}
and simple examples show that none of the implications in
(\ref{eq13}) can be reversed in general, not even for
$d=1$; also, $\varphi \stackrel{{\sf diff}}{\thicksim} \psi \not
\Rightarrow \varphi \stackrel{1^-}{\thicksim} \psi$ and
$\varphi \stackrel{1}{\thicksim} \psi\not \Rightarrow \varphi
\stackrel{{\sf diff}}{\thicksim} \psi $ in general. 

As alluded to in the Introduction, more restrictive forms of
equivalence have often been considered in the literature. To put these
in context, observe that $\tau_x$ in Proposition \ref{prop12} is uniquely
determined unless $\varphi_{\R}(x) = \{x\}$, i.e., unless $x$ is a
fixed point of $\varphi$, in symbols $x\in \mbox{\rm Fix}\, \varphi$,
in which case the continuous bijection $\tau_x$ is
arbitrary. Correspondingly, imposing additional requirements on the
family $(\tau_x)_{x\in X}$ yields more restrictive forms of
equivalence. For instance, understand
$\varphi\stackrel{\bigstar}{\thickapprox}\psi$ to mean that $\varphi
\stackrel{h}{\thicksim}\psi$ for some $h\in \cH_{\bigstar}$ so that
either $\tau_x$ is increasing for every $x\in X\setminus 
\mbox{\rm Fix} \, \varphi$ or else $\tau_x$ is
decreasing for every $x$. With this, clearly
\begin{equation}\label{eq2_2}
\varphi\stackrel{\bigstar}{\thickapprox}\psi \enspace \Longrightarrow
\enspace \varphi
\stackrel{\bigstar}{\thicksim}\psi\qquad \forall \bigstar \in
\{0,1^-,1, {\sf diff}, {\sf lin}\} \, ,
\end{equation}
and simple examples again show that no implication in (\ref{eq2_2})
can be reversed in general. It is a readily established, however, that for {\em linear\/} flows
(\ref{eq2_2}) can in fact be reversed; see \cite[Sec.\ 2]{BW2}.

\begin{prop}\label{lemH1}
Let $\Phi$, $\Psi$ be linear flows on $X$. Then
$$
\Phi
\stackrel{\bigstar}{\thickapprox} \Psi \enspace \Longleftrightarrow
\enspace \Phi
\stackrel{\bigstar}{\thicksim} \Psi \qquad \forall \bigstar\in \{0 , 1^-, 1,
\mbox{\sf diff}, \mbox{\sf lin}\} \, .
$$
\end{prop}

Still more restrictively than
$\varphi\stackrel{\bigstar}{\thickapprox}\psi $, understand
$\varphi\stackrel{\bigstar}{\cong}\psi $ to mean that (\ref{eq1_2})
holds for some $h\in \cH_{\bigstar}$, and say that $\varphi$, $\psi$
are  {\bf topologically}, {\bf H\"{older}}, {\bf Lipschitz}, {\bf
  differentiably}, and {\bf linearly conjugate} if
$\varphi\stackrel{0}{\cong}\psi$, $\varphi\stackrel{1^-}{\cong}\psi$,
$\varphi\stackrel{1}{\cong}\psi$, $\varphi\stackrel{{\sf
    diff}}{\cong}\psi$, and $\varphi\stackrel{{\sf
    lin}}{\cong}\psi$ respectively. Obviously
$\varphi\stackrel{\bigstar}{\cong}\psi $ implies
$\varphi\stackrel{\bigstar}{\thickapprox}\psi $, and (\ref{eq13})
remains valid with all symbols $\thicksim$ replaced by either
$\thickapprox$ or $\cong$, where again 
no implication can be reversed in general. Notice however that for
linear flows, by Theorem \ref{thm1x} 
together with \cite{BW, BW2}, validity of $\Phi\stackrel{\bigstar}{\thicksim}\Psi $
does, rather amazingly, imply that
$\Phi\stackrel{\bigstar}{\cong}\Psi_{*\alpha} $ for some non-zero
$\alpha$; also, $\Phi\stackrel{{\sf diff}}{\thicksim}\Psi$ implies
$\Phi\stackrel{{\sf lin}}{\thicksim}\Psi$, by Proposition
\ref{prop1zb}, and similarly with $\thicksim$ replaced by $\cong$.

Finally, given a (not necessarily linear) flow $\varphi$ on $X$ and any point $x\in X$,
denote the minimal $\varphi$-period of $x$ by $T_x^{\varphi} := \inf
\{t\in \R^+: \varphi_t(x) = x\}$, with the usual convention that $\inf
\varnothing = \infty$. Thus $x\in \mbox{\rm Fix}\, \varphi$ if and only if $T_x^{\varphi} =
0$. If $T_x^{\varphi} \in \R^+$ then $x$ is $T$-{\bf periodic} with $T\in
\R^+$, i.e., $\varphi_T(x) = x$, precisely if $T/T_x^{\varphi}\in
\N$. For convenience, let $\mbox{\rm Per}_T\varphi = \{x\in
X: \varphi_T(x) = x\}$ for every $T\in \R^+$, and let $\mbox{\rm
  Per}\, \varphi = \bigcup_{T\in \R^+} \mbox{\rm Per}_T \varphi$.
For linear flows, one basic, greatly simplifying feature is as
follows: If a point $x$ is (backward and forward) asymptotic to {\em some\/}
point, then $x$ itself is fixed or periodic. For
non-linear flows this implication may fail (though the
converse always holds).

\begin{prop}\label{propxy}
Let $\Phi$ be a linear flow on $X$. Then for every $x\in X$,
$T\in \R^+$, and convergent sequence $(a_n)$ in $\R$:
\begin{enumerate}
\item $\lim_{|t|\to \infty} \Phi_t x $ exists $\enspace
  \Longleftrightarrow \enspace  x\in   \mbox{\rm Fix} \, \Phi$;
\item $\lim_{k\in \Z, |k|\to \infty} \Phi_{kT + a_{|k|}} x$ exists  $\enspace  \Longleftrightarrow \enspace $
  $x\in \mbox{\rm Per}_T \Phi$.
\end{enumerate}
Furthermore, for every $x\in X_{\sf C}^{\Phi}$:
\begin{enumerate}
\item[{\rm (iii)}] $\lim_{t\to \infty} \Phi_t x$ exists $\enspace  \Longleftrightarrow \enspace $
  $\lim_{t\to -\infty} \Phi_t x$ exists $\enspace
  \Longleftrightarrow \enspace  x\in   \mbox{\rm Fix} \, \Phi$;
\item[{\rm (iv)}] $\lim_{n\to \infty} \Phi_{nT + a_n} x$ exists $\enspace  \Longleftrightarrow \enspace $
  $\lim_{n\to \infty} \Phi_{-nT + a_n}x$ exists  $\enspace  \Longleftrightarrow \enspace $
  $x\in \mbox{\rm Per}_T \Phi$.
\end{enumerate}
\end{prop}

\section{Linear flows and Lipschitz similarity}\label{sec2a}

Let $\Phi$ be a linear flow on $X$. A set $Y\subset X$ is $\Phi$-{\bf
  invariant} if $\Phi_t Y = Y$ for every $t\in \R$, or equivalently if
$\Phi_{\R} y \subset Y$ for every $y\in Y$. Notice that a sub{\em
  space\/} $Y$ of $X$ is $\Phi$-invariant if and only if
$A^{\Phi}Y\subset Y$. A linear flow $\Phi$ is {\bf irreducible} if $X=
Y \oplus  \widetilde{Y}$ with $\Phi$-invariant subspaces $Y$,
$\widetilde{Y}$ implies that either
$Y=\{0\}$ or $\widetilde{Y} = \{0\}$, and it is {\bf diagonal}({\bf izable}) if
$\Phi_t$ is diagonalizable (over $\C$) for some and hence every $t\in
\R \setminus \{0\}$. Thus, $\Phi$ is irreducible or diagonal if and only
if, relative to the appropriate basis, the generator $A^{\Phi}$ is a single real
Jordan block or a diagonal matrix (over $\C$) respectively. In
particular, for an irreducible $\Phi$ the spectrum $\sigma (\Phi):=
\sigma (A^{\Phi})$, i.e., the set of all eigenvalues of $A^{\Phi}$, is
either a real singleton or a non-real complex conjugate 
pair, so $\sigma(\Phi) = \{z, \overline{z}\}$ for some $z\in
\C$.
There exists a unique decomposition $X =
\bigoplus_{\ell = 1}^{\ell_0} X_{\ell}^{\Phi}$ where $\ell_0 \in \N$ and $X_{\ell}^{\Phi}$
is, for every $\ell \in \{1, \ldots , \ell_0\}$, a $\Phi$-invariant
subspace so that $\Phi|_{\R \times X_{\ell}^{\Phi}}$ is irreducible. With
this, $\Phi$ is irreducible precisely if
$\ell_0 = 1$ and is diagonal precisely if $\Phi|_{\R \times
  X_{\ell}^{\Phi}}$ is diagonal for every $\ell$. Moreover, letting
$$
X_{{\sf D}}^{\Phi} = \bigoplus \bigl\{X_{\ell}^{\Phi} : \Phi|_{\R  \times X_{\ell}^{\Phi}} \enspace
\mbox{\rm is diagonal} \bigr\} \, , \quad
X_{{\sf AD}}^{\Phi} = \bigoplus \bigl\{X_{\ell}^{\Phi} :
\Phi|_{\R\times X_{\ell}^{\Phi}} \enspace
\mbox{\rm is not diagonal} \bigr\} \, ,
$$
clearly the decomposition $X= X_{\sf D}^{\Phi} \oplus X_{\sf
  AD}^{\Phi}$ is $\Phi$-invariant, and $\Phi$ is diagonal precisely
if $X = X_{\sf D}^{\Phi}$. Refer to $X_{\sf D}^{\Phi}$ and $X_{\sf
  AD}^{\Phi}$ as the {\bf diagonal} and {\bf anti-diagonal}
space of $\Phi$ respectively, and say that $\Phi$ is {\bf
  anti-diagonal} whenever $X= X_{\sf AD}^{\Phi}$. 

For another slate of $\Phi$-invariant subspaces,
recall that $X = X_{\sf
  S}^{\Phi} \oplus X_{\sf C}^{\Phi} \oplus X_{\sf
  U}^{\Phi} = X_{\sf H}^{\Phi} \oplus X_{\sf C}^{\Phi}$, where
\begin{align*}
& X_{\sf S}^{\Phi} = \Bigl\{ x\in X : \lim\nolimits_{t\to \infty}
  \Phi_t x = 0 \Bigr\} = \bigoplus \bigl\{X_{\ell}^{\Phi} : \sigma (
                 \Phi|_{\R\times X_{\ell}^{\Phi}} )  \subset \C^- \bigr\}  \, , \\[1mm]
& X_{\sf C}^{\Phi} = \Bigl\{ x\in X : \lim\nolimits_{|t|\to \infty}
  e^{-\varepsilon |t|}\Phi_t x = 0 \enspace \forall \varepsilon > 0
                                                                                     \Bigr\} = \bigoplus \bigl\{X_{\ell}^{\Phi} : \sigma (
                 \Phi|_{\R\times X_{\ell}^{\Phi}} )  \subset i\R 
                                                                                     \bigr\}  \, ,   \\[1mm]
& X_{\sf U}^{\Phi} = \Bigl\{ x\in X : \lim\nolimits_{t\to - \infty}
                                                                                      \Phi_t  x = 0 \Bigr\} = \bigoplus \bigl\{X_{\ell}^{\Phi} : \sigma (
                 \Phi|_{\R\times X_{\ell}^{\Phi}} )  \subset \C^+ \bigr\}  \, ,  \\[1mm]
& X_{\sf H}^{\Phi}= X_{\sf S}^{\Phi} \oplus X_{\sf U}^{\Phi} =  \bigoplus \bigl\{X_{\ell}^{\Phi} : \sigma (
                 \Phi|_{\R\times X_{\ell}^{\Phi}} )  \subset \C \setminus  (   i\R      )\bigr\}     \, , 
\end{align*}
are referred to as the {\bf stable}, {\bf central}, {\bf unstable}, and
{\bf hyperbolic} space of $\Phi$ respectively. Say that $\Phi$ is {\bf stable}, {\bf central}, {\bf
  unstable}, and {\bf hyperbolic} if $X$ equals $X_{\sf S}^{\Phi}$, $X_{\sf
  C}^{\Phi}$, $X_{\sf U}^{\Phi} $, and $X_{\sf H}^{\Phi}$
respectively; moreover, $\Phi$ is {\bf bounded} if $\sup_{t\in
  \R}|\Phi_t|<\infty$, or equivalently if $\Phi$ is diagonal and central. 
For convenience throughout, usage of the word {\em flow\/} in conjunction with
any of these adjectives, as well as {\em irreducible}, ({\em
  anti-}){\em diagonal}, or {\em generated by}, automatically 
implies that the flow under consideration is linear. 
For every $\bullet\in \{{\sf D}, {\sf AD}, {\sf S}, {\sf C}, {\sf U}, {\sf H}\}$,
let $d_{\bullet}^{\Phi} = \dim X_{\bullet}^{\Phi}$,
write $\Phi|_{\R \times X_{\bullet}^{\Phi}}$ simply as
$\Phi_{\bullet}$, and denote by $P_{\bullet}^{\Phi}$ the linear
projection of $X$ onto $X_{\bullet}^{\Phi}$, along $X_{\circ}^{\Phi}$
with $\circ \in \{{\sf D}, {\sf AD}\}\setminus \{\bullet\}$ if $\bullet \in \{{\sf D},
{\sf AD}\}$, along $\bigoplus_{\circ
  \in \{{\sf S}, {\sf C}, {\sf U} \} \setminus \{  \bullet \} }
X_{\circ}^{\Phi}$ if $\bullet \in
\{{\sf S}, {\sf C}, {\sf U}\}$, and along $X_{\sf C}^{\Phi}$ if
$\bullet = {\sf H}$. Clearly, $\Phi  \stackrel{{\sf
    lin}}{\cong} \bigtimes_{\bullet\in \{{\sf D}, {\sf AD}\}}
\Phi_{\bullet} \stackrel{{\sf
    lin}}{\cong} \bigtimes_{\bullet\in \{{\sf S}, {\sf C}, {\sf U}\}} \Phi_{\bullet}$ via the linear
isomorphisms $\bigtimes_{\bullet\in \{{\sf D}, {\sf AD}\}}
P_{\bullet}^{\Phi} $, $ \bigtimes_{\bullet\in \{{\sf S}, {\sf C}, {\sf U}\}}
P_{\bullet}^{\Phi}$ respectively; also $d = d_{\sf D}^{\Phi} + d_{\sf
AD}^{\Phi} = d_{\sf S}^{\Phi} + d_{\sf C}^{\Phi} + d_{\sf
U}^{\Phi} = d_{\sf H}^{\Phi} + d_{\sf C}^{\Phi} $.
For convenience, the {\bf time-reversal} $\Phi_{\ast (-1)}$ of $\Phi$
is denoted $\Phi^*$, in other words 
$\Phi^*$ is generated by $-A^{\Phi}$. Notice that
$\Phi^* \stackrel{\sf lin}{\thicksim} \Phi$ and also $X_{\bullet}^{\Phi^*} =
X_{\bullet}^{\Phi} $ for every $\bullet \in \{{\sf D}, {\sf AD}, {\sf C},
{\sf H}\}$ whereas $X_{\sf S}^{\Phi^*} =
X_{\sf U}^{\Phi} $ and $X_{\sf U}^{\Phi^*} =
X_{\sf S}^{\Phi}$.

Several arguments in this article require a modicum of explicit
calculations. To keep these calculations transparent and consistent, the
following notations are used throughout:
Let $J_1 =[0]\in \R^{1\times 1}$, and for $m\in \N\setminus
\{1\}$ denote by $J_m$ the standard nilpotent $m\times m$-Jordan
block, 
$$
J_m = \left[ 
\begin{array}{ccccc}
0 & 1 & 0 & \cdots & 0 \\
\vdots & \ddots & \ddots &   &  \vdots \\
\vdots &  &  & \ddots & 0 \\
\vdots  & & & \ddots  & 1\\
0 & \cdots & \cdots & \cdots & 0
\end{array}
\right]\in \R^{m\times m} ;
$$
moreover, for every $m\in \N$ let 
$$
J_m (a) = aI_m + J_m \, , \quad J_m(a+ib) = a I_{2m} +  \left[
\begin{array}{c|c}
J_m & - b I_m \\ \hline
b  I_m & J_m
\end{array}
\right]  \qquad \forall  a\in \R , b\in \R  \setminus \{0\} \, .
$$
For every $z\in \C$, therefore, $J_m(z)$ simply is a {\em real\/} Jordan block with $\sigma \bigl(
J_m(z)\bigr) = \{z, \overline{z}\}$. Note that 
$J_m(z)\in \R^{m\times m}$ if $z\in \R$, whereas $J_m(z)\in \R^{2
  m\times 2m}$ if $z\in \C\setminus \R$. To efficiently keep track of
the size of $J_m(z)$, define ${\sf d}(b)\in \{1,2\}$ as
$$
{\sf d}(b) = \left\{
  \begin{array}{cl}
    1 & \mbox{\rm if } b = 0 \, , \\
    2 & \mbox{\rm if } b\in \R \setminus \{0\} \, ;
  \end{array}
\right.
$$
with this $J_m(a+ib) = a I_{m \, {\sf d}(b)} + J_m(ib)\in \R^{m\, {\sf
    d}(b)\times m\, {\sf d}(b)}$ for every $a,b\in \R$. Moreover, let
$$
K_m(b) = \left\{
  \begin{array}{ll}
    J_m & \mbox{\rm if } b = 0 \, , \\
   \mbox{\rm diag}\, [J_m,J_m] & \mbox{\rm if } b\in \R \setminus \{0\} \, ,
  \end{array}
\right. 
$$
as well as
$$
R_m(b) =e^{J_m(ib) - K_m(b) } = \left\{
  \begin{array}{ll}
    I_m & \mbox{\rm if } b = 0 \, , \\[1mm]
    \left[
\begin{array}{c|r}
\cos b \, I_m & - \sin b \,  I_m \\ \hline
\sin b \,  I_m & \cos b \, I_m
\end{array}
\right]  & \mbox{\rm if } b\in \R \setminus \{0\} \, ,
  \end{array}
\right. 
$$
so $K_m(b), R_m(b)\in \R^{m \, {\sf d}(b) \times m \, {\sf d}(b)}$
also. Note that $|K_m(b)^j|=1$ for all $b\in \R$, $j\in \{0, \ldots 
, m-1\}$, whereas $K_m(b)^{m} = O_{m \, {\sf d}(b)}$; here $K_m(b)^0:=
I_{m\, {\sf d}(b)}$. Also, $R_m(b)^{-1} = R_m (-b)=
R_m(b)^{\top}$; in particular, $R_m(b)$
preserves $|\cdot|$ for every $b\in \R$. Additionally, $R_m(b+c) = R_m (b)
R_m(c)$ provided that ${\sf d}(b+c)={\sf d}(b) = {\sf d}(c)$. Notice
that $K_m(b) R_m(c)= R_m(c)
K_m(b)$ for all $b,c\in \R$ with ${\sf d}(b) = {\sf d}(c)$, and
similarly $K_m(b)^{\top} R_m(c)= R_m(c)
K_m(b)^{\top}$. Valid for all $m\in \N$, $a,b\in\R$, the universal
formula 
\begin{equation}\label{eq2z1}
e^{tJ_m(a+ib)} = e^{at} R_m(bt) e^{tK_m(b)}  = e^{at} R_m(bt)
\sum\nolimits_{j=0}^{m-1} \frac{t^j}{j!} K_m(b)^j \qquad \forall t \in
\R\setminus \{0\} \, ,
\end{equation}
is going to be most useful on several occasions throughout this
article. (If $b\ne 0$ then (\ref{eq2z1}) fails for $t=0$ due to a
dimension mismatch between $R_m(0)$ and $K_m(b)$. This, however, is of
no concern, as (\ref{eq2z1}) is going to be utilized principally for large $|t|$.)

For the analysis in subsequent sections, it is helpful to
recall one further classical concept: Given any linear flow
$\Phi$ on $X$, the (forward) {\bf Lyapunov exponent} of $\Phi$ at $x$,
$$
\lambda_+^{\Phi} (x) = \lim\nolimits_{t\to \infty} \frac{\log |\Phi_t
  x|}{t}\, , 
$$
exists for every $x\in X\setminus \{0\}$, by virtue of either an
explicit computation or the multiplicative ergodic theorem \cite{ARDS,
CK}; moreover, the range of $x\mapsto
\lambda_+^{\Phi}(x)$ equals $\{{\sf Re}\, z : z \in \sigma(\Phi)\}$. With
$\lambda_+^{\Phi}(0):= -\infty$ for convenience, the set $L^{\Phi}
(s):= \{x\in X : \lambda_+^{\Phi}
(x) \le s\}$ is a $\Phi$-invariant subspace for every $s\in \R$,
referred to as the {\bf Lyapunov space} of $\Phi$ at $s$. It is
readily seen that
$$
L^{\Phi} (s)= \bigoplus \bigl\{X_{\ell}^{\Phi} : \sigma (
                 \Phi|_{\R\times X_{\ell}^{\Phi}} ) =
                 \{z,\overline{z}\} \enspace \mbox{\rm with}\enspace
              {\sf Re}\, z \le s \bigr\}  
\qquad \forall s\in \R \, .
$$
Letting $\ell^{\Phi}(s)= \dim L^{\Phi}(s)$, clearly the
integer-valued function $\ell^{\Phi}$ is non-decreasing and
right-con\-tin\-uous, with $\lim_{s\to -\infty}\ell^{\Phi}(s) = 0$ and
$\lim_{s\to \infty}\ell^{\Phi}(s) = d$. Observe that $\lambda_+^{\Phi} (x) = a$ for
some $x\in X$, $a\in \R$ precisely if $\ell^{\Phi} (a^-) < \ell^{\Phi}
(a)$, and refer to the non-negative integer $\ell^{\Phi} (a) - \ell^{\Phi}
(a^-)$ as the {\bf multiplicity} of $a$. Let $\lambda_1^{\Phi}\le \lambda_2^{\Phi} \le \ldots \le
\lambda_d^{\Phi}$ be the (not necessarily different) Lyapunov
exponents of the linear flow $\Phi$, that is, $\bigl\{ \lambda_j^{\Phi} : j\in \{1,
\ldots , d\} \bigr\} = \bigl\{ \lambda_+^{\Phi} (x): x\in X \setminus
\{0\} \bigr\}$, with each exponent repeated according to its 
multiplicity. For convenience, let
$$
\Lambda^{\Phi} := \Lambda^{A^{\Phi}} := \mbox{\rm diag} \, 
[\lambda_1^{\Phi}, \ldots , \lambda_d^{\Phi}] \, .
$$
Note that if
$\Phi$ is irreducible with $\sigma (\Phi) = \{z,\overline{z}\}$ for
some $z\in \C$, then simply $\Lambda^{\Phi} = {\sf Re}\,  z \, I_d$. Also,
$\Phi$ is stable, unstable, central, and hyperbolic precisely if
$\lambda_j^{\Phi}<0$, $\lambda_j^{\Phi}>0$, $\lambda_j^{\Phi}=0$, and
$\lambda_j^{\Phi}\ne 0$ for every $j\in \{1, \ldots , d\}$
respectively. Moreover, $\ell^{\Phi^*}(-s) = d - \ell^{\Phi}(s^-)$ for
all $s\in \R$, and consequently $\lambda_j^{\Phi^*} = -
\lambda_{d+1-j}^{\Phi}$ for every $j\in \{1, \ldots , d\}$, that is,
$\Lambda^{\Phi^*} = -\mbox{\rm diag}\, [\lambda_d^{\Phi}, \ldots ,
\lambda_1^{\Phi}]$.

Recall that $A^{\Phi}$, $A^{\Psi}$ are {\bf
  similar} if $A^{\Phi} = P^{-1} A^{\Psi} P$ for some invertible $P\in
\R^{d\times d}$. Clearly, $A^{\Phi}$, $A^{\Psi}$ are similar if and
only if $A^{\Phi_{\bullet}}$, $A^{\Psi_{\bullet}}$ are similar for every $\bullet
\in \{{\sf D}, {\sf AD}, {\sf S}, {\sf C}, {\sf U}, {\sf H}\}$. 
Say that the generators of two linear flows $\Phi, \Psi$ are {\bf Lyapunov
  similar} if $\Lambda^{\Phi}$, $\Lambda^{\Psi}$ are similar, in which
case automatically $\Lambda^{\Phi} = \Lambda^{\Psi}$. Thus $A^{\Phi}$, $A^{\Psi}$ are Lyapunov similar
precisely if they have the same Lyapunov exponents, with matching
multiplicities, or equivalently if $\ell^{\Phi} =
\ell^{\Psi}$. Observe that if $A^{\Phi}$, $A^{\Psi}$
are similar then they are Lyapunov similar, as are $A^{\Phi_{\bullet}}$, $A^{\Psi_{\bullet}}$ for every $\bullet
\in \{{\sf D}, {\sf AD}, {\sf S}, {\sf C}, {\sf U}, {\sf H}\}$. 

In order to introduce one intermediate notion of
similarity, for every $m\in \N$, $a,b\in \R$ define
\begin{equation}\label{eq3_1a}
\cL J_m(a+ib) = \left\{
  \begin{array}{ll}
    a I_{{\sf d}(b)} & \mbox{\rm if } m = 1 \, ,\\
    J_m(a+ib) & \mbox{\rm if } m \ge 2 \, .
    \end{array}
\right.
\end{equation}
More generally, given any $A\in \R^{d\times d}$, choose an invertible $P\in \R^{d\times
  d}$ so that
\begin{equation}\label{eq3_1b}
A  = P^{-1} \mbox{\rm diag}\, \bigl[
J_{m_1} (z_1) , \ldots , J_{m_k} (z_k)
\bigr] P \, ,
\end{equation}
with appropriate $k\in \N$, $m_1, \ldots , m_k\in \N$, $z_1, \ldots ,
z_k\in \C$, and define
$$
\cL A = P^{-1} \mbox{\rm diag}\, \bigl[
\cL J_{m_1} (z_1) , \ldots , \cL J_{m_k} (z_k)
\bigr]  P \, .
$$
Say that $A^{\Phi}$, $A^{\Psi}$ are {\bf Lipschitz similar} if $\cL A^{\Phi}$, $\cL A^{\Psi}$ are similar. Notice that 
$\cL A$ is determined by $A$ up to similarity, and so the notion of Lipschitz
similarity is well-defined. To fully appreciate it, observe
that $A^{\Phi}$, $A^{\Psi}$ are Lipschitz similar if and only if
$A^{\Phi_{\bullet}}$, $A^{\Psi_{\bullet}}$ are Lipschitz similar for every $\bullet
\in \{{\sf D}, {\sf AD}, {\sf S}, {\sf C}, {\sf U}, {\sf
  H}\}$. Moreover, it is readily seen that $A^{\Phi}$, $A^{\Psi}$ are
Lipschitz similar precisely if $A^{\Phi}$, $A^{\Psi}$ are Lyapunov
similar while $A^{\Phi_{\sf AD}}$, $A^{\Psi_{\sf AD}}$ are similar.
This underscores the intermediate nature of Lipschitz similarity,
in that 
\begin{equation}\label{eq2z2}
A^{\Phi}, A^{\Psi} \enspace \mbox{\rm similar} \quad
\Longrightarrow \quad
A^{\Phi}, A^{\Psi} \enspace \mbox{\rm Lipschitz similar} \quad
\Longrightarrow \quad
A^{\Phi}, A^{\Psi} \enspace \mbox{\rm Lyapunov similar}  \, ;
\end{equation}
neither implication in (\ref{eq2z2}) can be reversed in general
for $d\ge 2$, though the left and right implication is reversible
whenever $\Phi$, $\Psi$ are anti-diagonal and diagonal, respectively.

For a simple illustration of (\ref{eq2z2}) for $d=2$, fix any $A\in
\R^{2\times 2}$. Then, up to
multiplication by $\alpha \in \R \setminus \{0\}$, the matrix $A$ is
{\em similar\/} either to $J_1(i)$, $J_2$, $J_2(1)$, or to precisely one of
$J_1(1+ib)$ with $b\in \R^+$ or 
\begin{equation}\label{eq2z3}
O_2, \, \mbox{\rm diag}[a,1] \:\mbox{\rm with} \: a\in [-1,1] \, ;
\end{equation}
$A$ is {\em Lipschitz similar\/} precisely to one of $J_2, J_2(1)$ or (\ref{eq2z3});
and $A$ is {\em Lyapunov similar\/} precisely to one of (\ref{eq2z3}); see
also Figure \ref{fig21}.

\begin{figure}[ht] 
  \psfrag{tl1a}[]{$J_1(i)$}
  \psfrag{tl1b}[]{$O_2$}
  \psfrag{tl1c}[]{$ J_2$}
  \psfrag{tl1e}[]{$\mbox{\rm diag}\, [a,1]$ with $a\in [-1,1]$}
  \psfrag{tl1d}[]{$J_2(1)$}
  \psfrag{tl1f}[]{$J_1(1+ib)$ with $b\in \R^+$}
  \psfrag{tl1g}[]{ $J_2(1)$}
   \psfrag{tarr}[]{$\Longrightarrow$}
  \psfrag{tsmoo}[]{{\em similar}}
  \psfrag{tlip}[]{{\em Lipschitz similar}}
  \psfrag{thoel}[]{{\em Lyapunov similar}}
%
%
%
\vspace*{2mm}
\begin{center}
\includegraphics{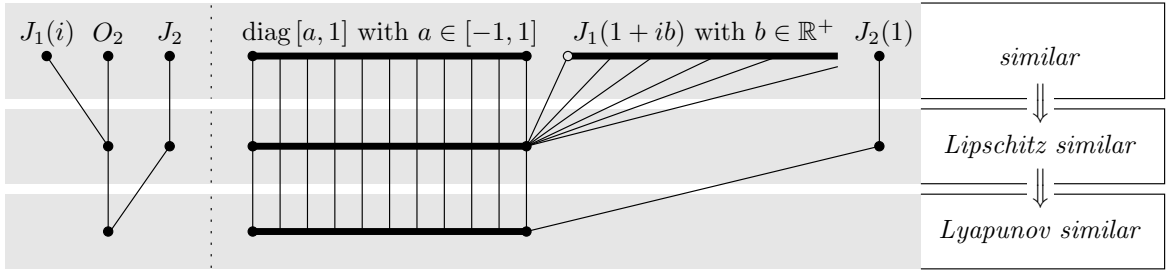}
\end{center}
\vspace*{-4mm}
\caption{Given $A\in \R^{2\times 2}$, the matrix $\alpha A$ is similar, Lipschitz
  similar, or Lyapunov similar for some $\alpha \in \R
  \setminus \{0\}$ to precisely one of the
  matrices shown in the respective row; the scalar $\alpha$ is
  uniquely determined if and only if $\sigma (A)\not \subset i\R$,
  visualized here as the cases to the right of the dotted line.}\label{fig21}
\end{figure}

This section concludes with two preparatory observations
regarding linear flows that are going to be instrumental later on. 
The first observation asserts that the Lyapunov spaces of {\em stable\/}
flows, and hence also their Lyapunov exponents, are well-behaved under
H\"{o}lder equivalence. The result plays a crucial role in
the H\"{o}lder classification of linear flows \cite{BW2, MM}.

\begin{prop}{{\rm \cite[Lem.\ 6.6]{BW2}}}\label{prop2bb}
Let $\Phi$, $\Psi$ be stable flows on $X$. Assume that $\Phi
\stackrel{h}{\thicksim} \Psi$ for some $h\in \cH_{1^-}(X)$. Then there
exists a unique $\alpha \in \R^+$ so that $\Lambda^{\Phi} = \alpha \Lambda^{\Psi}$,
\begin{equation}\label{eq3_4A}
h \bigl( L^{\Phi} (\alpha s)\bigr) = L^{\Psi} (s) \qquad \forall s \in
\R \, ,
\end{equation}
as well as
\begin{equation}\label{eq3_4B}
\lim\nolimits_{t\to \infty} \frac{\tau_x (t)}{t} = \alpha \qquad
\forall x \in X \setminus \{0\} \, .
\end{equation}
\end{prop}

\begin{rem}\label{rem3_1A}
(i) The first conclusion in Proposition \ref{prop2bb} carries over to
{\em arbitrary\/} linear flows $\Phi$, $\Psi$ on $X$ as follows: If
$\Phi \stackrel{1^-}{\thicksim} \Psi$ then $\Lambda^{\Phi} = \alpha
\Lambda^{\Psi}$ for some $\alpha \in \R \setminus \{0\}$; see
\cite[Thm.\ 6.4]{BW2}. By contrast, simple examples show that
(\ref{eq3_4A}) and (\ref{eq3_4B}) do not even carry over to {\em hyperbolic\/}
flows in general, and convergence in (\ref{eq3_4B}) can be arbitrarily
slow. However, under the stronger assumption that $h\in \cH_1(X)$,
both conclusions are going to be strengthened significantly by Theorem
\ref{lem3aol1} below.

(ii) In \cite{CK,KS} the term {\em Lyapunov space\/} refers not to
$L^{\Phi}(s)$ but instead to the $\Phi$-invariant subspace
$$
\widetilde{L^{\Phi}}(s):= \bigoplus \bigl\{X_{\ell}^{\Phi} : \sigma ( \Phi|_{\R\times
  X_{\ell}^{\Phi}} ) \subset s +i\R \bigr\}  \qquad \forall s\in \R \, . 
$$
Based on $\widetilde{L^{\Phi}}$, the argument proving \cite[Thm.\ 5.6]{KS}
proceeds on the assumption that a Lipschitz conjugacy $h$ ``surely
preserves the Lyapunov spaces''; in other words, if $\Phi
\stackrel{h}{\cong} \Psi$ with $h\in \cH_1(X)$ then $h\bigl(
\widetilde{L^{\Phi}}(s)\bigr) = \widetilde{L^{\Psi}}(s)$ for all
$s\in \R$. (A similar assumption seems to underpin \cite{MM}.) For a simple demonstration that this assumption is 
unwarranted in general even for stable flows, take $d=2$ and let
$\Phi$ be generated by $\mbox{\rm diag}\, [-2,-1]$, in which case
$$
L^{\Phi}(s) = \left\{
  \begin{array}{ll}
    \{0\} & \mbox{\rm if } s< -2 \, , \\
    \mbox{\rm span}\{e_1\} & \mbox{\rm if } -2\le s< -1 \, , \\
      \R^2  & \mbox{\rm if } s\ge - 1 \, , \\
    \end{array}
  \right.
  \quad
  \widetilde{L^{\Phi}}(s) = \left\{
  \begin{array}{ll}
    \mbox{\rm span}\{e_1\} & \mbox{\rm if } s=-2  \, , \\
    \mbox{\rm span}\{e_2\} & \mbox{\rm if } s=-1 \, , \\
      \{0\} & \mbox{\rm otherwise }. \\
    \end{array}
  \right.
  $$
Now, fix any diffeomorphism $f:\R \to \R$ with the property that
\begin{equation}\label{eq3_rem32A}
\sup\nolimits_{t\in \R} \bigl( |f(t) - t| + |f(t) - tf'(t)| \bigr) <
\infty \, ,
\end{equation}
and define $h_f:\R^2 \to \R^2$ as
$$
h_f(x) = \left\{ \begin{array}{cl}
                   x & \mbox{\rm if } x\in \mbox{\rm span}\{e_1\} \, , \\[2mm]
                   \left[
                   \begin{array}{c}
x_2^2 f(x_1/x_2^2) \\[-1mm]  x_2
                     \end{array}
                   \right] & \mbox{\rm if } x \not \in \mbox{\rm span}\{e_1\} \, .
                   \end{array}
  \right.
$$
Note that $h_f^{-1} = h_{f^{-1}}$, and from (\ref{eq3_rem32A}) it is
readily deduced that $h_f \in \cH_1(\R^2)$. Moreover, for every $x\not
\in \mbox{\rm span}\{e_1\}$,
$$
h_f (\Phi_t x) = h_f \left(
  \left[
\begin{array}{c} e^{-2t} x_1 \\[-1mm] e^{-t} x_2 \end{array}
    \right]
  \right) = \left[
\begin{array}{c} e^{-2t} x_2^2 f(x_1/x_2^2) \\ e^{-t} x_2 \end{array}
    \right] = \Phi_t h_f (x) \qquad \forall t \in \R \, ,
$$
and the two outer-most expressions agree also if $x\in \mbox{\rm span}\{e_1\}$. Thus $\Phi
\stackrel{h_f}{\cong}\Phi$. Nonetheless, notice that
$$
h_f\bigl( \widetilde{L^{\Phi} }(-1) \bigr)  = \bigl\{ x\in \R^2 : x_1 = f(0)
x_2^2\bigr\} \, .
$$
Whenever $f(0)\ne 0$ the set $h_f\bigl( \widetilde{L^{\Phi}
}(-1) \bigr) $ is a parabola and hence {\em not\/} a subspace, let
alone equal to $\widetilde{L^{\Phi}} (-1)$. By contrast, $h_f\bigl(L^{\Phi}(s)
\bigr) = L^{\Phi}(s)$ for all $s\in \R$, in accordance with
Proposition \ref{prop2bb} and regardless of the choice of
$f$. This example illustrates how $\widetilde{L^{\Phi}}$ may behave
poorly under equivalence. Since the good behaviour of key objects is
crucial for the purpose of the present article, the spaces $\widetilde{L^{\Phi}}$ will not be
considered further here.  
\end{rem}

The second observation is a (marginally reformulated) version of \cite[Lem.\
5.1]{BW2}. To state it concisely, given any $m\in \N$ and $b\in \R^m$,
define ${\sf d}(b)\in \{m , \ldots, 2m \}$ as
$$
{\sf d}(b) = \sum\nolimits_{j=1}^m {\sf d}(b_j) =  \# \{1\le j\le
m: b_j = 0\} + 2 \# \{1\le j\le
m: b_j \ne  0\} \, ;
$$
for $m=1$, this is consistent with earlier usage of ${\sf d}(b)$.

\begin{prop}\label{prop1zz}
Given $k,\ell \in \N$ and $b\in \R^k$, $c\in \R^{\ell}$ with ${\sf
  d}(b)={\sf d}(c) = d$, let
$\Phi$, $\Psi$ be the flows on $\R^d$ generated by
$$
\mbox{\rm diag}\, [J_1(ib_1), \ldots , J_1 (ib_k)] \, , \quad
\mbox{\rm diag}\, [J_1(ic_1), \ldots , J_1 (ic_{\ell}) ] 
$$
respectively. Then the following statements are equivalent:
\begin{enumerate}
\item there exists an open set $U\subset \R^{d}$ with $0\in U$
  and a continuous one-to-one function $f:U\to \R^{d}$ so that
  $T_x^{\Phi} = T_{f(x)}^{\Psi}$ for every $x\in U$;
\item there exists a bijection $g: \{1, \ldots , k\}\to \{1, \ldots, \ell\}$ so that
  $|b_j| = |c_{g(j)}|$ for every $j$;
\item $\Phi \stackrel{{\sf lin}}{\cong}\Psi$.
\end{enumerate}
\end{prop}

\section{Distortion points for stable flows}\label{sec3}

This section introduces a simple yet important property that a point
may or may not have relative to a stable flow on $X$. The importance
of this property stems from the following basic fact about any bi-Lip\-schitz
homeomorphism of $X$; see \cite[Lem.\ A1]{MM}. 

\begin{lem}\label{prop3}
Let $h\in \cH_1(X)$. For any two sequences $(x_n),
(y_n)$ in $X\setminus \{0\}$ with $x_n \to 0$, $y_n \to 0$ the
following statements are equivalent:
\begin{enumerate}
\item $|x_n - y_n|/|y_n| \to 0$;
\item $|h(x_n) - h(y_n)|/|h(y_n)| \to 0$. 
\end{enumerate}
\end{lem}

\begin{proof}
Denote by $\kappa \ge 1$ a common Lipschitz constant for $h,h^{-1}$ on
some neighborhood $U$ of $0$. Pick $0<\varepsilon < 1$ so small that
$B_{\varepsilon}(0)\subset U  \cap h^{-1} (U)$. Then $|x-y|/\kappa \le
|h(x)-h(y)|\le |x-y|\kappa$ for all $x, y\in B_{\varepsilon}(0)$. Consequently, given any two sequences
$(x_n),(y_n)$ in $X \setminus \{0\}$ with $x_n \to 0$, $y_n \to 0$,
there exists $N\in \N$ so that  
$$
0\le  \frac{|x_n - y_n|}{\kappa^2 |y_n|} \le \frac{\kappa |h(x_n) -
  h(y_n)|}{\kappa^2 |h(y_n)|/\kappa} =  \frac{|h(x_n) - h(y_n)|}{|h(y_n)|} \le
 \frac{\kappa|x_n - y_n|}{|y_n|/\kappa} = \kappa^2 \frac{|x_n - y_n|}{|y_n|}
\qquad \forall n \ge N \, .
$$
From this it is obvious that indeed (i)$\Leftrightarrow$(ii).
\end{proof}

To harness Lemma \ref{prop3} for the study of stable flows, for
convenience let 
$$
\cR  = \bigl\{ \rho : \R \to \R \enspace \mbox{\rm is continuous
  with} \: 
\lim\nolimits_{t\to \infty} \rho (t) = \infty \bigr\} \, ,
$$
and observe that $\cR$ is closed under composition.
Given $\delta>0$ and a stable flow $\Phi$ on $X$, say that $x\in
X$ is $\delta$-{\bf distorting} (for $\Phi$) if for every
$\varepsilon>0$ there exists a $y\in B_{\varepsilon}(x)\setminus \{0\}$ so that 
$$
\limsup\nolimits_{t \to \infty} \frac{|\Phi_t x -
  \Phi_{\rho(t)}y|}{|\Phi_{\rho(t)} y|} \ge \delta \qquad \forall \rho \in
\cR \, .
$$
Say that $x$ is a {\bf distortion point} if $x$ is 
$\delta$-distorting for some (and hence every sufficiently small)
$\delta>0$. Informally put, $x$ is a distortion point for $\Phi$ if
there exist, arbitrarily close to $x$, orbits 
that are noticeably non-tangential to $\Phi_{\R} x$ in forward
time. Denote by $D^{\Phi}$ the set of all distortion points for
$\Phi$, i.e., let 
$$
D^{\Phi}= \bigcup_{\delta>0} \Bigl\{
x\in X : x \: \mbox{\rm is $\delta$-distorting for $\Phi$} 
\Bigr\} \, .
$$
Trivially $x=0$ is $1$-distorting, so $0\in D^{\Phi}$ always.
As $x$ belonging to $D^{\Phi}$ indicates a lack of tangency
between orbits near $\Phi_{\R}x$, it seems plausible, not least in
light of Lemma \ref{prop3}, that $D^{\Phi}$ may be preserved under
Lipschitz equivalence. This indeed is the case.

\begin{lem}\label{lem3_1}
Let $\Phi$, $\Psi$ be stable flows on $X$. Assume that $\Phi
\stackrel{h}{\thicksim} \Psi$ for some $h\in \cH_1(X)$. Then $h(D^{\Phi})
= D^{\Psi}$.
\end{lem}

\begin{proof}
Plainly $\tau_x$ is increasing for every $x\in X\setminus \{0\}$, and hence $\tau_x
\in \cR$. (Unlike $\tau_x$, however, a typical element
of $\cR$ may be neither one-to-one nor onto.) To show that
$h(D^{\Phi}) \subset D^{\Psi}$, assume that $x\in X$ is
$\delta$-distorting for $\Phi$
for some $\delta > 0$. Let $\kappa \ge 1$ be a Lipschitz constant for $h$, $h^{-1}$ on
some neighborhood of $0$. Given $\varepsilon >0$, pick $\varepsilon_0 >0$ so
small that $h \bigl( B_{\varepsilon_0}(x)\bigr)\subset
B_{\varepsilon} \bigl( h(x)\bigr)$. Since $x$ is $\delta$-distorting,
there exists a $y\in B_{\varepsilon_0}(x)\setminus \{0\}$ with 
\begin{equation}\label{eq31}
\limsup\nolimits_{t \to \infty} \frac{|\Phi_t x -
  \Phi_{\sigma (t)}y|}{|\Phi_{\sigma(t)} y|} \ge \delta \qquad \forall \sigma \in
\cR \, .
\end{equation}
Clearly $h(y) \in B_{\varepsilon} \bigl(
h(x)\bigr)\setminus \{0\}$ because $h(0)=0$. Given any $\rho\in \cR$, 
note that $ \sigma:= \tau_y^{-1} \circ \rho \circ \tau_x \in \cR$. With this,
\begin{align*}
& \limsup\nolimits_{t\to \infty}  \frac {|\Psi_t h(x) -
  \Psi_{\rho(t)}h(y)|}{|\Psi_{\rho(t)} h(y)|} = \limsup\nolimits_{t\to
                                      \infty}
                                     \frac{|h(\Phi_{\tau_x^{-1}(t)}x)
                                     - h(\Phi_{\tau_y^{-1}\circ \rho
                                     (t)}y)|}{|h(\Phi_{\tau_y^{-1}\circ
                 \rho(t)}y)|}
  \ge \\
&   \qquad \ge  \limsup\nolimits_{t \to   \infty}
                                     \frac{|\Phi_{\tau_x^{-1}(t)}x -
           \Phi_{\sigma \circ \tau_x^{-1} (t)}y|/\kappa
           }{| \Phi_{\sigma \circ \tau_x^{-1} (t)}y |\kappa }
                                     = \frac{1}{\kappa^2}  \limsup\nolimits_{t\to  \infty}
                                     \frac{|\Phi_t x -
                                     \Phi_{\sigma (t)}y|}{|\Phi_{\sigma(t)}y|}  \ge \frac{\delta}{\kappa^2}\, ;                                  
\end{align*}
here the first $\ge$ is due to $h,h^{-1}$ being Lipschitz
near $0$, whereas the last $\ge$ is due to (\ref{eq31}). 
Since $\varepsilon >0$ and $\rho \in \cR$ have been arbitrary, $h(x)$
is $\delta/\kappa^2$-distorting for $\Psi$. Thus $h(D^{\Phi})\subset D^{\Psi}$, and interchanging the
roles of $\Phi$, $\Psi$ yields $h(D^{\Phi})= D^{\Psi}$.
\end{proof}

\begin{rem}\label{rem3a}
For the flows $\Phi$, $\Psi$ on $\R^2$ generated by
$J_2(-1)$, $-I_2$ respectively, it turns out that $D^{\Phi} = \mbox{\rm span}\,
\{e_1\}\ne \{0\} = D^{\Psi}$ even though $\Phi
\stackrel{1^-}{\thicksim} \Psi$; see Theorem \ref{lem5bb} below. This
shows that the conclusion of Lemma \ref{lem3_1} 
may fail if merely $h\in \cH_{1^-}(X)$.
\end{rem}

The remainder of this section develops an explicit description of
$D^{\Phi}$ for every stable flow $\Phi$ on $X$. As it turns out,
$D^{\Phi}$ always is a $\Phi$-invariant proper subspace of $X$. To explicitly describe
this subspace, recall from Section \ref{sec2a} the Lyapunov spaces
$L^{\Phi}$ associated with any linear flow $\Phi$ on $X$, for which
$$
L^{\Phi}(s^-) = \left\{
x\in X : \lim\nolimits_{t\to \infty} \frac{\Phi_t
  x}{e^{st}} = 0   \right\} \qquad \forall s\in \R \, . 
$$
As a convenient refinement of $L^{\Phi}$, for every $m\in \N_0$ let
$$
L_m^{\Phi}(s) = \left\{
x\in X : \lim\nolimits_{t\to \infty} \frac{\Phi_t
  x}{e^{s t} t^m } = 0 
  \right\} \qquad \forall s\in \R \, . 
$$
The sets $L_m^{\Phi}(s)$, henceforth referred to
informally as the {\bf refined Lyapunov spaces} of $\Phi$ at $s$, are readily
seen to have the following basic properties. 

\begin{prop}\label{prop3aoo1}
Let $\Phi$ be a linear flow on $X$ and $m,n\in \N_0$, $s,t\in \R$. Then:
\begin{enumerate}
\item $L_m^{\Phi}(s)$ is a $\Phi$-invariant subspace of $X$;
  \item if $m\le n$ and $s\le t$, then $L_m^{\Phi} (s) \subset
    L_{n}^{\Phi} (t)$;
  \item $L^{\Phi}(s^-)= L_0^{\Phi}(s) \subset \ldots \subset
    L_m^{\Phi}(s) \subset L_{d}^{\Phi} (s) = L^{\Phi}(s)$;
    \item if $m\le n$ and $L_m^{\Phi} (s) = L_{m+1}^{\Phi}(s)$, then $L_m^{\Phi}(s)
      =L_{n}^{\Phi}(s)$.
    \end{enumerate}
    Moreover, $\Phi$ is diagonal if and only if $L^{\Phi}_1(s) =
    L^{\Phi}(s)$ for all $s\in \R$.
\end{prop}

\noindent
By Proposition \ref{prop3aoo1} the family $L_m^{\Phi}$ does indeed
refine $L^{\Phi}$: With
$$
m^{\Phi}(s) := \min\bigl\{m\in \N_0 : L_m^{\Phi}(s) = L^{\Phi}(s) \bigr\}
\qquad \forall s\in \R \, , 
$$
notice that $m^{\Phi}(s)\ne 0$ precisely if $s$ is a Lyapunov exponent of
$\Phi$; if it is, then the $m^{\Phi}(s) + 1$ subspaces   
$L^{\Phi} (s^-)=L_0^{\Phi}(s)\subset \ldots \subset
L_{m^{\Phi}(s)}^{\Phi}(s)= L^{\Phi}(s)$ all are different, and the multiplicity of $s$ is at least $m^{\Phi}(s)$.
For a simple concrete example,
consider the flow $\Phi$ on $\R^d$ generated by $J_d(-1)$. Then
$m^{\Phi}(s)$ equals $d$ or $0$ depending on whether $s=-1$ or $s\ne
-1$, and for
every $m\in \N_0$, $s\in \R$,
$$
L^{\Phi}(s) = \left\{
  \begin{array}{cl}
    \{0\} & \mbox{\rm if } s<-1 \, , \\
    \R^d & \mbox{\rm if } s\ge -1 \, ,
    \end{array}
  \right.
  \quad
  L_m^{\Phi}(s) = \left\{
  \begin{array}{ll}
    \{0\} & \mbox{\rm if } s<-1 \, , \\
    \mbox{\rm span}\, \{e_1, \ldots , e_{\min\{m, d\}} \} &
                                                              \mbox{\rm
                                                              if } s =
                                                              -1 \, ,
    \\
    \R^d & \mbox{\rm if } s> -1 \, ;
    \end{array}
  \right.
  $$
  here and throughout, $\mbox{\rm span}\, \varnothing := \{0\}$ as usual.
In general, given any linear flow $\Phi$ on $X$, denote by
$\lambda^{\Phi}$ the largest Lyapunov exponent of $\Phi$, i.e.,
\begin{equation}\label{eq3yy1}
\lambda^{\Phi} = \max\nolimits_{x\in X \setminus \{0\}}
\lambda^{\Phi}_+ (x) = \max\nolimits_{\lambda \in \sigma (\Phi)} {\sf
  Re}\,  \lambda = \min \{s\in \R : L^{\Phi}(s) =  X\} \, ,
\end{equation}
and for convenience let
\begin{equation}\label{eq3yy2}
  m^{\Phi} = m^{\Phi} (\lambda^{\Phi}) = \min  \{m\in \N_0 :
  L_{m}^{\Phi}(\lambda^{\Phi}) = X\} \, .
\end{equation}
For the flow $\Phi$ on $\R^d$ generated by $J_d(-1)$ in particular,
simply $\lambda^{\Phi}=-1$ and $m^{\Phi} = d$. In general, 
$m^{\Phi}$ is the largest size of any Jordan block associated with an
eigenvalue of maximal real part. Thus $m^{\Phi}$ equals the largest dimension of any
irreducible component with eigenvalue $\lambda^{\Phi}$, or one-half the largest dimension
of any irreducible component with non-real eigenvalues on
$\lambda^{\Phi}+i \R$, or both.

Given any $x\in X$, as $t\to \infty$ notice that $ |\Phi_t x|$
does not decay more slowly (if $\lambda^{\Phi}<0$) or does not grow
faster (if $\lambda^{\Phi}>0$) than $ e^{\lambda^{\Phi}t} t^{m^{\Phi} - 1}$.
In fact, the latter expression yields {\em
  precisely\/} the slowest possible rate of decay or fastest possible
rate of growth, as the subspace
$L_{m^{\Phi}-1}^{\Phi} (\lambda^{\Phi})$ always is proper whereas by
(\ref{eq3yy2}) clearly $L_{m^{\Phi}}^{\Phi} (\lambda^{\Phi}) = X$.
While (refined) Lyapunov spaces are useful in the study of 
linear flows in general \cite{BaPe, CK}, they
are especially relevant for {\em stable\/} flows, not least because
$L_{m^{\Phi}-1}^{\Phi} (\lambda^{\Phi}) $, the largest non-trivial
space among all refined Lyapunov spaces, holds particular
significance for such flows.

\begin{theorem}\label{lem5bb}
  Let $\Phi$ be a stable flow on $X$. Then $D^{\Phi} = L_{m^{\Phi}-1}^{\Phi}
  (\lambda^{\Phi})$ with $m^{\Phi}$, $\lambda^{\Phi}$ given
  by {\rm (\ref{eq3yy2})}, {\rm 
  (\ref{eq3yy1})} respectively.
\end{theorem}

\begin{proof}
Let $X_1, \ldots , X_{\ell_0}$ with $\ell_0\in \N$ be $\Phi$-invariant
subspaces of $X=\R^d$ so that $X = \bigoplus_{\ell =1}^{\ell_0}
X_{\ell}$ and each flow 
$\Phi_{[\ell ]}:= \Phi|_{\R \times X_{\ell}}$ is irreducible.  
Assume w.l.o.g.\ that $\lambda^{\Phi}=-1$, and write $m^{\Phi}$
simply as $m\in \N$. Using an appropriate linear change of
coordinates, it can be assumed that $\Phi$ is generated by
\begin{equation}\label{eqlem51}
A = \mbox{\rm diag}\, [J_{m_1}(-1+ib_1), \ldots , J_{m_k}(-1 +ib_k), A_0] \, ,
\end{equation}
with the appropriate $k \in \{1, \ldots , \ell_0\}$, $m_1, \ldots , m_k \in \N$ with $m=m_1 \ge
\ldots \ge m_k$, and $b\in \R^{k}$; additionally, $k_0\in \N_0$, and $A_0\in \R^{k_0 \times k_0}$
is present in (\ref{eqlem51}) only if $\lambda_+^{\Phi} (x) < - 1$ for some
$x\ne 0$, and in this case $\lim_{t\to \infty}e^t|e^{tA_0}|= 0$. Thus
$d=\sum_{j=1}^k m_j\,  {\sf d}(b_j) + k_0$ and, for instance, $X_1= \mbox{\rm span}\,
\{e_1, \ldots , e_{m_1\,  {\sf d}(b_1)}\}$ with
$\Phi_{[1]}$ generated by $J_{m_1}(-1+ib_1)$. For convenience, for every
$\ell \in \{1, \ldots , k \}$ denote
by $P_{\ell}$ the orthogonal projection of $X$ onto $X_{\ell}$, and by $P:=
\sum_{\ell=1}^k P_{\ell}$ the orthogonal projection of $X$ onto $\bigoplus_{\ell=1}^k
X_{\ell}$, the latter space being the sum of all generalized eigenspaces of
$A$ associated with any eigenvalue on $-1 + i\R$. Deduce from (\ref{eq2z1})
that 
\begin{equation}\label{eqlem52}
\Phi_t = e^{tA} = e^{-t} R(t)  \left( P + 
\sum\nolimits_{j=1}^{m -1} \frac{t^j}{j!} K^j + e^t \, \mbox{\rm diag}\,
\bigl[O_{d-k_0}, e^{tA_0}  \bigr]   \right) \qquad
\forall t \in \R \, ,
\end{equation}
where $R(t), K\in \R^{d\times d}$ are given by
$$
R(t) =  \left\{ \begin{array}{ll} I_d & \mbox{\rm if $t=0$}, \\
  \mbox{\rm diag}\, \bigl[R_{m_1} (b_1t), \ldots ,
                  R_{m_k} (b_{k}t), I_{k_0} \bigr] & \mbox{\rm if
                                                     $t\ne 0$},
                \end{array} \right. \quad
              K = \mbox{\rm diag}\, \bigl[K_{m_1}(b_1), \ldots , K_{m_k}(b_k), O_{k_0}\bigr] 
$$
respectively, with an ``empty'' sum understood to equal
$0$ as usual. Note that $R(t)$ is an isometry, and $|K^j|=1$ for all $j\in \{0, \ldots , m-1\}$, whereas $K^m =
O_{d}$. Clearly $R(t) P_{\ell}= P_{\ell}R(t)$ and $K P_{\ell}= P_{\ell} K$ for every
$\ell \in \{1, \ldots k \}$. For convenience, let
$$
E = L_{m-1}^{\Phi} (-1) = \left\{
  \begin{array}{ll}
    \ker P & \mbox{\rm if } m = 1 \, , \\
    \ker K^{m-1} & \mbox{\rm if } m\ge 2 \, ,
    \end{array}
  \right.
$$
so the theorem simply asserts that $E=D^{\Phi}$. This will now be
proved in two separate steps.

\medskip

\noindent
\underline{Step I:} Proof of $E\supset D^{\Phi}$.

\smallskip

\noindent
Pick any $x\not \in E$ and $\delta>0$.
On the one hand, if $m=1$
then $x\not \in E$ means that $Px\ne 0$. In this case, choose
$0<\varepsilon < \frac12 \delta |Px|$ so small that $|Py|>\frac12
|Px|$ for all $y\in B_{\varepsilon}(x)$. By (\ref{eqlem52}),
$$
|\Phi_t z| = e^{-t} \big|Pz + e^t \, \mbox{\rm diag}\, \bigl[
O_{d-k_0}, e^{tA_0}\bigr] z \big|   \qquad \forall t \in \R , z \in X  \, .
$$
Since $\lim_{t\to \infty} e^t |e^{tA_0}|=0$, it follows
that
\begin{equation}\label{eq4_N1}
\lim \nolimits_{t\to \infty} \frac{|\Phi_t x - \Phi_t y|}{|\Phi_t
  y|} = \frac{|Px - Py|}{|Py|} < \frac{2|x-y|}{|Px|}< \delta \qquad \forall y\in
B_{\varepsilon} (x) \, .
\end{equation}
On the other hand, if $m\ge 2$ then $x \not \in E$ means that $K^{m-1}x\ne
0$. In this case, choose $0<\varepsilon < \frac12
\delta |K^{m-1}x|$ so small that $|K^{m-1}y|>\frac12 |K^{m-1}x|$ for
all $y\in B_{\varepsilon}(x)$. Also, choose $T_1>0$ so large that
$$
\sum\nolimits_{j=0}^{m-2} \frac{t^{j-m+1}}{j!}  +\frac{e^t}{t^{m-1}}
|e^{tA_0}| \le
\frac{m}{(m-1)! t} \qquad \forall t \ge T_1 \, .
$$
Deduce from (\ref{eqlem52}) that for every $t \ge T_1$ and $z\in X$,
\begin{align*}
  |\Phi_t z|  & \ge e^{-t} t^{m-1} \left(
\frac{|K^{m-1} z|}{(m-1)!} - \left(  \sum\nolimits_{j=0}^{m-2}
               \frac{t^{j-m+1}}{j!} + \frac{e^t}{t^{m-1}}|e^{tA_0}|
                         \right)  |z|
                            \right) \\
  & \ge e^{-t} \frac{t^{m-1}}{(m-1)!} \left(
|K^{m-1} z| - \frac{m|z|}{t} \right) ,
\end{align*}
but also
$$
|\Phi_t z |  \le e^{-t} t^{m-1}  \left(\frac1{(m-1)!} +  \sum\nolimits_{j=0}^{m-2}
               \frac{t^{j-m+1}}{j!}    + \frac{e^t}{t^{m-1}}|e^{tA_0}|\right) |z| 
   \le  e^{-t} \frac{t^{m-1}}{(m-1)!} \left(
1 + \frac{m}{t} \right)|z| \, .
$$
With this, and in analogy to (\ref{eq4_N1}),
$$
\limsup\nolimits_{t\to \infty} \frac{|\Phi_t x - \Phi_t y|}{|\Phi_t
  y|} \le \lim\nolimits_{t\to \infty} \frac{|x-y|
  (1+m/t)}{|K^{m-1}y| - m|y|/t} =
\frac{|x-y|}{|K^{m-1}y|} < \frac{2 \varepsilon }{|K^{m-1} x|} < \delta \qquad \forall y \in B_{\varepsilon} (x) \, .
$$
Regardless of whether $m=1$ or $m\ge 2$, therefore, $x$ is {\em not\/} $\delta$-distorting for
$\Phi$. Since $\delta>0$ has been arbitrary, $x\not \in D^{\Phi}$, and
since $x\not \in E$ has been arbitrary as well, $E\supset D^{\Phi} $. 

\medskip

\noindent
\underline{Step II:} Proof of $E\subset D^{\Phi}$.

\smallskip

\noindent
Pick any $x\in E$, $0<\varepsilon < 1$, and $\rho \in \cR$.

Assume first that $P_1 x = 0$. In this case, let $y = x + \frac12 \varepsilon e_m \in B_{\varepsilon}
(x)\setminus \{0 \}$. Observe that $P_1\Phi_t x = 0$ for all $t\in \R$, whereas
$P_1y=\frac12 \varepsilon e_m \ne  0$ and
$$
\left|
|P_1\Phi_s y| - \frac{e^{-s} s^{m - 1} \varepsilon}{2(m-1)!}
  \right| \le \frac{e^{-s} \varepsilon}{2} \sum\nolimits_{j=0}^{m -2}
  \frac{s^j}{j!} \qquad \forall s\in \R^+ \, . 
$$
Choose $T_2 >0$ so large that $\rho(t)>0$ for
all $t\ge T_2$, but also
$$
\frac{e^{-\rho(t)} \rho(t)^{m - 1}\varepsilon}{3(m - 1)!} \le 
|P_1\Phi_{\rho(t)} y| \le |\Phi_{\rho(t)} y | \le 
\frac{ 2 e^{-\rho(t)}\rho(t)^{m- 1}\varepsilon}{3(m - 1)!} \qquad \forall t
\ge  T_2 \, ;
$$
notice that $T_2$ may depend on $x$, $\varepsilon$, and $\rho$. It follows that
$$
\frac{|\Phi_t x - \Phi_{\rho(t)}y|}{|\Phi_{\rho(t)}y|} \ge
\frac{|P_1 \Phi_t x - P_1 \Phi_{\rho(t)}y|}{|\Phi_{\rho(t)}y|} =
\frac{|P_1\Phi_{\rho(t)}y|}{|\Phi_{\rho(t)}y|} \ge
\frac{\varepsilon/3}{2\varepsilon/3 } = \frac12  \qquad \forall t \ge  T_2 \, ,
$$
and hence clearly
\begin{equation}\label{eq4_6A}
\limsup\nolimits_{t\to \infty} \frac{|\Phi_t x -
  \Phi_{\rho(t)}y|}{|\Phi_{\rho(t)}y|} \ge \frac12 \, .
\end{equation}

It remains to consider the case of $P_1x\ne 0$. In this
case (which can occur only if $m\ge 2$) let $k = \max \{0\le j \le m-2: P_1K^j x
\ne 0\}$. Then $P_1K^k x \in X_1 \cap \ker K$, so $P_1K^k x = \nu
|P_1K^k x| R(\vartheta) e_1$ with an appropriate $\nu \in \{-1,1\}$
and $\vartheta \in \R^+$. With this, let
$$
y = x - {\textstyle \frac12} \varepsilon \nu R(\vartheta) e_m \in
B_{\varepsilon} (x) \setminus \{0\} \, .
$$
Note that $K^j e_m = e_{m-j}$ for every $j\in \{0, \ldots , m-1\}$,
and hence $K^{m-1} y = P_1K^{m-1} y = -\frac12 \varepsilon \nu
R(\vartheta) e_1$. Deduce from (\ref{eqlem52}) that
\begin{equation}\label{eq3_4_new}
\left|
|P_1 \Phi_t x| - \frac{e^{-t} t^k}{k!} |P_1K^k x|
\right| \le \left|
P_1 \Phi_t x - \frac{e^{-t}t^k}{k!} R(t) P_1K^k x
\right| \le e^{-t} |x| \sum\nolimits_{j=0}^{k-1} \frac{t^j}{j!} \qquad
\forall t \in \R^+ \, .
\end{equation}
Choose $T_3>0$ so large that
$\rho(t)>0$ for all $t\ge T_3$, but also
\begin{equation}\label{eq3_5_new}
|x| \sum\nolimits_{j=0}^{k-1} \frac{t^j}{j!} \le
\frac{t^k}{3k!}|P_1K^k x| \, , \quad
(|x|+1) \left(
\sum\nolimits_{j=0}^{m-2} \frac{\rho(t)^{j}}{j!} +
e^{\rho(t)}\big|e^{\rho(t) A_0}\big|
  \right)\le \frac{\rho(t)^{m-1} \varepsilon}{6(m-1)!} \qquad \forall
  t \ge T_3 \, ;
\end{equation}
similarly to $T_2$ above, $T_3$ may depend on $x$, $\varepsilon$, and
$\rho$. Deduce from (\ref{eq3_4_new}), (\ref{eq3_5_new}) that
$$
\left|
|P_1 \Phi_t x| - \frac{e^{-t}t^k}{k!} |P_1K^k x|
\right| \le \left|
P_1 \Phi_t x - \frac{e^{-t}t^k}{k!} R(t) P_1K^k x
  \right| \le \frac{ e^{-t}t^k}{3k!} |P_1 K^k x| \qquad \forall t \ge
  T_3 \, ,
$$
as well as 
$$
\left|
|P_1 \Phi_{\rho(t)}y | - \frac{e^{-\rho(t)}\rho(t)^{m-1}\varepsilon }{2(m-1)!}
\right| \le \left|
\Phi_{\rho(t)} y - \frac{e^{-\rho(t)}\rho(t)^{m-1}}{(m-1)!} R\bigl(
\rho(t)\bigr)  K^{m-1} y
  \right|\le \frac{ e^{-\rho(t)}\rho(t)^{m-1}\varepsilon}{6(m-1)!}
  \quad \: \: \forall t \ge T_3 \, ,
$$
and consequently
$$
\frac{e^{-\rho(t)}\rho(t)^{m-1}\varepsilon }{3(m-1)!} \le
|\Phi_{\rho(t)}y| \le \frac{2 e^{-\rho(t)}\rho(t)^{m-1}\varepsilon
}{3(m-1)!}  \qquad \forall t \ge T_3 \, .
$$
It follows that for every $t\ge T_3$,
\begin{align*}
  |P_1 \Phi_t x & - P_1 \Phi_{\rho(t)}y| \ge \\
                 & \ge \left|
\frac{e^{-t}t^k}{k!} R(t)P_1K^k x -
                                         \frac{e^{-\rho(t)}\rho(t)^{m-1}}{(m-1)!}
                                         R\bigl( \rho(t) \bigr)
                                         K^{m-1} y
                                         \right| - \frac{ e^{-t}
                                         t^k}{3k!} |P_1K^k x| -
                                         \frac{e^{-\rho(t)}\rho(t)^{m-1}\varepsilon}{6(m-1)!}
  \\
  & =  \left| 
\frac{e^{-t}t^k}{k!} |P_1K^k x| R(t) e_1 +
                                         \frac{e^{-\rho(t)}\rho(t)^{m-1}\varepsilon
    }{2 (m-1)!}
                                         R\bigl( \rho(t) \bigr)  e_1
                                         \right| - \frac{ e^{-t}
                                         t^k}{3k!} |P_1K^k x| -
                                         \frac{e^{-\rho(t)}\rho(t)^{m-1}\varepsilon}{6(m-1)!}
    \, .
\end{align*}
In order to show that (\ref{eq4_6A}) holds in this case as well, it is helpful to distinguish two
mutually exclusive possibilities: On the one hand, if $\sup_{t\ge T_3} |\rho(t) -
t|=\infty$ then
there exists an increasing sequence $(t_n)$ with $t_1\ge T_3$ and $t_n \to \infty$ so that
$$
b_1 (\rho(t_n) - t_n) \in 2\pi
\Z \qquad \forall n \in \N \, .
$$
(This condition trivially is satisfied for {\em every\/} sequence
$(t_n)$ if $b_1=0$.) With this, $R(t_n) e_1 = R\bigl( 
\rho(t_n) \bigr)e_1$ for every $n\in \N$, so
$$
|P_1 \Phi_{t_n} x - P_1 \Phi_{\rho(t_n)} y| \ge \left(
\frac{e^{-t_n}t_n^k}{k!} |P_1K^k x| + \frac{e^{-\rho(t_n)}\rho(t_n)^{m-1}\varepsilon}{2(m-1)!}
  \right) \left( 1 - \frac{1}{3}\right) \ge \frac{  |\Phi_{\rho(t_n)} y|}{2} \, ,
$$
and consequently
\begin{equation}\label{eqastx}
\limsup\nolimits_{t\to \infty} \frac{|\Phi_t x -
  \Phi_{\rho(t)}y|}{|\Phi_{\rho(t)}y|} \ge 
\limsup\nolimits_{t\to \infty} \frac{|P_1 \Phi_t x -
  P_1 \Phi_{\rho(t)}y|}{|\Phi_{\rho(t)}y|} \ge
\frac12 \, .
\end{equation}
On the other hand, if $\gamma := \sup_{t\ge T_3}|\rho(t) -
t|<\infty$ then for all $t\ge T_3$,
$$
|P_1 \Phi_t x - P_1 \Phi_{\rho(t)}y|  \ge
                                       \frac{e^{-\rho(t)}\rho(t)^{m-1}\varepsilon}{3(m-1)!}
                                       -
                                       \frac{4 e^{-t} t^k}{3
                                         k!}|P_1K^k x| \ge \frac{
                                         |\Phi_{\rho(t)}y| }{2} -
                                       \frac{4e^{-t}t^k}{3 k!} |P_1K^k
                                       x| \, ,
$$
and consequently
\begin{align*}
\limsup\nolimits_{t\to \infty} \frac{|\Phi_t x -
  \Phi_{\rho(t)}y|}{|\Phi_{\rho(t)}y|} & \ge \frac12 -
                                                                                  \frac{4|P_1K^k x|
                                         (m-1)!}{\varepsilon k!}
                                         \limsup\nolimits_{t\to \infty}
                                         \frac{e^{\rho(t)-t}t^k}{\rho(t)^{m-1}}
  \\
  & \ge \frac12  - \frac{4 |P_1K^k x|(m-1)!e^{\gamma}}{\varepsilon
    k!} \lim\nolimits_{t\to \infty} \frac{t^k}{(t-\gamma)^{m-1}} =
    \frac12 \, ,
\end{align*}
where the last equality is due to the fact that $k\le m-2$. 

In summary, (\ref{eq4_6A}) holds for every $x\in E$, $0<\varepsilon <
1$, $\rho \in \cR$, and with an appropriate $y\in
B_{\varepsilon}(x)\setminus \{0\}$. This means that every $x\in E$ is
$\frac12$-distorting for $\Phi$. Thus $E\subset D^{\Phi}$, and the proof is complete.
\end{proof}

The Lyapunov spaces $L^{\Phi}$ for stable flows, with their good
behaviour recorded in Proposition \ref{prop2bb}, are crucial
tools for establishing a H\"{o}lder classification of linear flows
\cite{BW2}. For the purpose of
the present article the refined Lyapunov spaces $L_m^{\Phi}$ are just as
crucial, not least due to the following result which significantly
strengthens Proposition \ref{prop2bb} in the Lipschitz context.

\begin{theorem}\label{lem3aol1}
Let $\Phi$, $\Psi$ be stable flows on $X$. Assume that $\Phi
\stackrel{h}{\thicksim}\Psi$ for some $h\in \cH_1(X)$. Then there
exists a unique $\alpha \in \R^+$ so that
\begin{equation}\label{eq4_6BA}
h\bigl( L_m^{\Phi} (\alpha s)\bigr) = L_m^{\Psi} (s) \qquad \forall m
\in \N_0 , \, s \in \R \, ,
\end{equation}
as well as
\begin{equation}\label{eq4_6BB}
\sup\nolimits_{t\ge 0} |\tau_x (t) - \alpha t| < \infty \qquad \forall
x \in X \setminus \{0\} \, .
\end{equation}
\end{theorem}

As presented below, the proof of (\ref{eq4_6BB}) makes use
of the following elementary calculus fact.

\begin{prop}\label{lem2a}
For every $m\in \N_0$ the function $f_m: \R\to\R^+$ given by $f_m(t)=
e^{-t} \max\{|t|, m\}^{m}$ is a continuous, decreasing bijection with
$$
\lim\nolimits_{|t|\to \infty} \Bigl( f_m^{-1} \bigl( r f_m(t) \bigr)  -
t \Bigr) = -\log r \qquad \forall r \in \R^+ \, .
$$
\end{prop}

\begin{proof}[Proof of Theorem \ref{lem3aol1}]
Since $\Phi \stackrel{h}{\thicksim} \Psi$ with $h\in \cH_{1}(X)$, by
Proposition \ref{prop2bb} there exists a unique $\alpha \in \R^+$ so that $h
\bigl( L^{\Phi} (\alpha s) \bigr)= L^{\Psi} (s)$ for all $s\in
\R$. No generality is lost by assuming that $\alpha =1$. Then also $h
\bigl( L^{\Phi} (s^-) \bigr)= L^{\Psi} (s^-)$ for all $s\in
\R$. Given any $s\in \R$, assume first that $L^{\Phi}(s^-) =
L^{\Phi}(s)$. In this case, $L^{\Psi} (s^-) = L^{\Psi}(s)$ as well,
and from
$$
L^{\Psi}(s^-) = h
\bigl( L^{\Phi} ( s^-) \bigr) \subset h \bigl(
L_m^{\Phi}(s)\bigr) \subset  h \bigl(
L^{\Phi}(s)\bigr) = L^{\Psi} (s) = L^{\Psi} (s^-) \, ,
$$
it is clear that
$$
h\bigl( L_m^{\Phi} ( s)\bigr) = L^{\Psi}(s) =  L_m^{\Psi} (s) \qquad
\forall m \in \N_0 \, .
$$
By contrast, assume that $L^{\Phi} (s^-) \ne L^{\Phi} (s)$, and hence
also $L^{\Psi}(s^-)\ne L^{\Psi}(s)$. In this case, $m^{\Phi}(s) \ge 1$, and for every $j\in
\{0, \ldots , m^{\Phi}(s) \}$ denote by $\Phi_{[j]}$ the
restriction of $\Phi$ to $\R \times L_{m^{\Phi}(s)  - j}^{\Phi}(s)$. Similarly,
$m^{\Psi}(s) \ge 1$, and
$\Psi_{[j]}$ denotes the restriction of $\Psi$ to $\R \times L_{m^{\Psi}(s)
  - j}^{\Psi}(s)$. Thus $\Phi_{[0]}$ simply is $\Phi$ restricted to
$\R\times L^{\Phi}(s)$,
and $m^{\Phi_{[0]}} = m^{\Phi}(s)$. By Theorem \ref{lem5bb},
$D^{\Phi_{[0]}} = L_{m^{\Phi}(s)-1}^{\Phi}(s)$, and hence $\Phi_{[1]}$
is $\Phi$ restricted to $\R \times L_{m^{\Phi}(s)-1}^{\Phi}(s)$. Similarly,
$D^{\Psi_{[0]}} = L_{m^{\Psi}(s)-1}^{\Psi}(s)$, and $\Psi_{[1]}$ is
$\Psi$ restricted to $L_{m^{\Psi}(s)-1}^{\Psi}(s)$. Moreover, since
$\Phi_{[0]}\stackrel{h_0}{\thicksim} \Psi_{[0]}$ with $h_0 =
h|_{L^{\Phi}(s)}$, Lemma \ref{lem3_1} provides the crucial middle equality in
$$
h\bigl(L_{m^{\Phi}(s)-1}^{\Phi}(s) \bigr) = h (D^{\Phi_{[0]}}) =
D^{\Psi_{[0]}} = L_{m^{\Psi}(s)-1}^{\Psi} (s) \, .
$$
Repeating this argument shows that
\begin{equation}\label{eq3qq1}
h\bigl( L_{m^{\Phi}(s) - j}^{\Phi}(s)\bigr) = h(D^{\Phi_{[j]}}) =
D^{\Psi_{[j]}} = L_{m^{\Psi}(s)-j}^{\Psi} (s) \qquad \forall j \in
\bigl\{0 , \ldots , \min \{m^{\Phi}(s), m^{\Psi}(s)\}  \bigr\} \, .
\end{equation}
Now if $m^{\Phi}(s) < m^{\Psi}(s)$, then (\ref{eq3qq1}) with
$j=m^{\Phi}(s)$ would imply
$$
h \bigl( L^{\Phi}(s^-) \bigr) = h \bigl( L_0^{\Phi}(s)\bigr) =
L_{m^{\Psi}(s) - m^{\Phi}(s) }^{\Psi} (s) \ne L^{\Psi}(s^-) =h
\bigl( L^{\Phi}(s^-) \bigr)  \, ,
$$
an obvious contradiction. Similarly, $m^{\Phi}(s)> m^{\Psi}(s)$ is
impossible, so it follows that $m^{\Phi}(s)= m^{\Psi}(s)$. Hence (\ref{eq3qq1})
really reads
$$
h \bigl( L_m^{\Phi}(s) \bigr) = L_m^{\Psi}(s) \qquad \forall m \in
\{0,\ldots, m^{\Phi}(s) \} \, ; 
$$
since this equality automatically is correct also whenever $m> m^{\Phi}(s)$, the
proof of (\ref{eq4_6BA}) is complete.

It remains to prove (\ref{eq4_6BB}) with $\alpha = 1$. To this
end, let $s_1 < \ldots < s_k$ be the $k\le d$
{\em distinct\/} Lyapunov exponents of $\Phi$, that is, $k = \# \{{\sf
  Re}\, \lambda : \lambda \in \sigma (\Phi)\}$. Then $X\setminus \{0\}$
equals the disjoint union $\bigcup_{j=1}^k L^{\Phi}(s_j)\setminus
L^{\Phi}(s_j^-)$, where each set itself is a disjoint union, namely
$$
L^{\Phi}(s_j)\setminus L^{\Phi}(s_j^-) =
\bigcup\nolimits_{m=0}^{m^{\Phi}(s_j)-1} L_{m+1}^{\Phi}(s_j)\setminus
L_{m}^{\Phi} (s_j)\qquad \forall j\in \{1, \ldots , k\} \, .
$$
Thus, given any $x\in X\setminus \{0\}$, there exist unique integers
$j\in \{1, \ldots , k\}$ and $m\in \{0, \ldots, m^{\Phi}(s_j)-1\}$
so that $x\in L_{m+1}^{\Phi}(s_j)\setminus L_{m}^{\Phi} (s_j)$, and by
(\ref{eq4_6BA}) also $h(x) \in  L_{m+1}^{\Psi}(s_j)\setminus L_{m}^{\Psi}(s_j)$.
For convenience, given any two functions $f, g:\R^+ \to \R^+$, write
$f(t)\prec g(t)$ whenever
$\limsup_{t\to \infty} f(t)/g(t)<\infty$. Thus $f(t)\prec g(t)$ if and
only if there exist $r,T\in \R^+$ so that $f(t)\le r g(t)$ for all
$t\ge T$. Furthermore, understand $f(t)\asymp g(t)$ to mean that both
$f(t)\prec g(t)$ and $g(t)\prec f(t)$. With this, 
$$
|\Phi_t x| \asymp e^{s_j t} t^{m} \, , \quad |\Psi_{\tau_x(t)} h(x)|
\asymp e^{s_j \tau_x(t)} \tau_x(t)^{m} \, ,
$$
and consequently, since $s_j<0$ and $h$ is Lipschitz near $0$,
$$
f_m \bigl(|s_j|\tau_x(t) \bigr) \asymp e^{-|s_j| \tau_x(t)}
\tau_x(t)^{m} \asymp |h(\Phi_t x)| \prec |\Phi_t x| \asymp
e^{-|s_j|t} t^{m} \asymp f_m (|s_j| t ) \, ,
$$
where $f_m:\R\to \R^+$ is the bijection of Proposition
\ref{lem2a}. Similarly,
$$
f_m (|s_j| t) \asymp e^{-|s_j|t} t^{m} \asymp \big|
h^{-1} \bigl(  \Psi_{\tau_x(t)} h(x)\bigr)
\big| \prec |\Psi_{\tau_x(t)} h(x)| \asymp e^{-|s_j| \tau_x(t)}
\tau_x(t)^{m} \asymp f_m\bigl( |s_j| \tau_x(t) \bigr) \, .
$$
Thus, with $r>1$ and $T\in \R^+$ sufficiently large,
$$
\frac{f_m (|s_j|t)}{r} \le f_m \bigl( |s_j| \tau_x(t)\bigr) \le r f_m
(|s_j| t) \qquad \forall t \ge T \, .
$$
Since $f_m$ is decreasing, it follows that
\begin{equation}\label{eqqqz}
f_m^{-1} \bigl( rf_m (|s_j| t )\bigr) - |s_j| t \le |s_j|( \tau_x (t)
- t ) \le 
f_m^{-1} \left( \frac{f_m(|s_j| t)}{r}\right) - |s_j| t
\qquad \forall t \ge T \, .
\end{equation}
By Proposition \ref{lem2a}, the left- and right-most expressions in
(\ref{eqqqz}) both have finite limits as $t\to \infty$. (In fact,
these limits are $-\log r$ and $\log r$
respectively, with the latter number being the larger one because
$r>1$.) From this it is clear that $\sup_{t\ge
  0}|\tau_x(t) -  t|<\infty$ as claimed.
\end{proof}

\begin{rem}\label{rem37}
Simple examples show that (\ref{eq4_6BA}) and (\ref{eq4_6BB}) do not
even carry over to hyperbolic flows in general, 
and they may also fail if merely $h\in \cH_{1^-}(X)$.
\end{rem}


\section{Characterizing Lipschitz equivalence}\label{sec4}

This section provides a proof of Theorem \ref{thm1x}, the main result previewed in the
Introduction. This is first done for a special case,
namely for stable flows with a single Lyapunov exponent and
equal-sized irreducible components. Although this case is very special
indeed, it elucidates the preservation of most but not all
characteristic quasi-frequencies (i.e., imaginary parts of eigenvalues
not on $i\R$) under Lipschitz equivalence. As it turns out, this partial preservation of quasi-frequencies is
precisely what distinguishes Lipschitz equivalence from its
differentiable  and H\"{o}lder counterparts. Thus, the two
observations regarding those very special flows (Lemmas \ref{lem4_1}
and \ref{lem4_2}) constitute important steps towards Theorem
\ref{thm1x}. Both observations build naturally on an auxiliary
result (Lemma \ref{lem4_0}) which is going to be established first. To set
the scene for this, fix $k,\ell\in \N$ along with $m_1, \ldots, m_k, n_1,
\ldots, n_{\ell}\in \N$, $b\in \R^k$, $c\in \R^{\ell}$, and
$k_0,\ell_0\in \N_0$ with 
\begin{equation}\label{eq4_8}
\sum\nolimits_{j=1}^k m_j \, {\sf d}(b_j) + k_0 =
\sum\nolimits_{j=1}^{\ell} n_j \, {\sf d}(c_j) + \ell_0 = d \, ;
\end{equation}
furthermore, let $A_0\in \R^{k_0 \times k_0}$, $B_0\in \R^{\ell_0 \times
  \ell_0}$ be such that
\begin{equation}\label{eq4_9}
\lim\nolimits_{t\to \infty} e^t (k_0 |e^{tA_0}| + \ell_0 |e^{tB_0}|) = 0 \, .
\end{equation}
With these ingredients, consider the two $d\times d$-matrices
\begin{equation}\label{eq4_10}
A  = \mbox{\rm diag}\, \bigl[ J_{m_1}(-1+ib_1), \ldots ,
J_{m_k}(-1+ib_k), A_0 \bigr] \, ,  \quad 
 B  =  \mbox{\rm diag}\, \bigl[ J_{n_{1}} (-1+ic_1), \ldots ,
 J_{n_{\ell}}(-1+ic_{\ell}) , B_0\bigr] \, ,
\end{equation}
where $A_0$, $B_0$ are understood to be present in (\ref{eq4_10}) only if $k_0\ge 1$,
$\ell_0\ge 1$ respectively. With $\Phi$, $\Psi$ generated by $A$,
$B$ respectively, notice that both flows are stable with $\lambda^{\Phi} = 
-1 =\lambda^{\Psi}$ as well as $m^{\Phi} = \max_{j=1}^k m_j$ and
$m^{\Psi}=\max_{j=1}^{\ell} n_j$. To neatly formulate the following
auxiliary result, for every $p\in \N$ 
denote $\{1\le j \le k: m_j=  p\}$ by $U_p$ and
$\{1\le j \le \ell: n_j=  p\}$ by $V_p$.

\begin{lem}\label{lem4_0}
Given $k,\ell\in \N$, $m_1 , \ldots, m_k, n_1, \ldots, n_{\ell}\in
\N$, $b\in \R^k$, $c\in \R^{\ell}$, and $k_0,\ell_0\in \N_0$ with {\rm
(\ref{eq4_8})}, as well as $A_0\in \R^{k_0\times k_0}$, $B_0 \in
\R^{\ell_0\times \ell_0}$ with {\rm (\ref{eq4_9})} as appropriate, let $\Phi$, $\Psi$ be the flows on $\R^d$ generated by
$A$, $B$ in {\rm (\ref{eq4_10})} respectively. Assume that
$\Phi\stackrel{1}{\thicksim}\Psi$. Then:
\begin{enumerate}
\item $\sum_{j\in U_p } {\sf d}(b_j) = \sum_{j\in
    V_p} {\sf d}(c_j)$ for every $p \in \N$;
\item for every $p\in \N\setminus \{1\}$ either $U_p=V_p=\varnothing$
  or else there exists
  a bijection $g: U_p \to V_p$ so that $|b_j|=|c_{g(j)}|$ for every $j\in U_p$. 
\end{enumerate}
\end{lem}

In the proof of Lemma \ref{lem4_0} given below, the following elementary
observation regarding bijections in general is used; see also Figure \ref{fig51}.

\begin{prop}\label{prop2zz}
Given non-empty sets $U_0, U, V_0, V, W$ with $U_0\subset U$, $V_0
\subset V$ as well as functions $f:U\to W$, $g:V\to W$, assume 
there exist bijections $h_0 :
U_0 \to V_0$, $h:U\to V$ with $f|_{U_0} =
g \circ h_0$, $f=g\circ h$. If $U_0$, $V_0$ are finite then
there exists a bijection $\widetilde{h}:U\setminus U_0 \to V \setminus
V_0$ with $f|_{U\setminus U_0} = g \circ \widetilde{h}$.
\end{prop}

\begin{figure}[ht] 
  \psfrag{tl1a}[]{$U$}
  \psfrag{tl1b}[]{$V$}
  \psfrag{tl1a0}[]{$U_0$}
  \psfrag{tl1b0}[]{$V_0$}
  \psfrag{tr1}[]{$W$}
\psfrag{ttf}[]{$f$}
\psfrag{ttg}[]{$g$}
\psfrag{tth}[]{$h$}
\psfrag{tth0}[]{$h_0$}
\psfrag{tth1}[]{ $\widetilde{h}$}
\psfrag{tt1}[l]{ $f = g \circ h$}
\psfrag{tt2}[l]{ $f|_{U_0}= g\circ h_0$}
\psfrag{tt3}[l]{ $h$, $h_0$ bijections}
\psfrag{tt4}[l]{ $\widetilde{h}$ bijection}
%
%
\begin{center}
\includegraphics{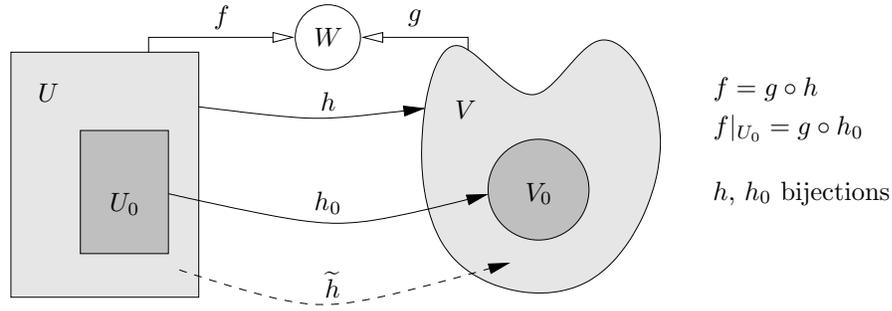}
\end{center}
\vspace*{-4mm}
\caption{If $U_0$, $V_0$ are
  finite then $f|_{U\setminus U_0} = g\circ \widetilde{h}$ for some bijection
  $\widetilde{h}:U\setminus U_0 \to V\setminus V_0$; see  Proposition \ref{prop2zz}.}\label{fig51}
\end{figure}

Finiteness of $U_0$, $V_0$ is essential in Proposition
\ref{prop2zz}. To see this, take for instance $U=V=\N$, $W=\{0,1\}$ and $f(n)=g(n)
= |\sin ( n \pi/2)|$. Then $U_0 = \N\setminus \{1,3\}$, $V_0 = 
\N\setminus\{1\}$ meet all requirements except finiteness, yet
clearly no bijection $\widetilde{h}:U\setminus U_0 \to V\setminus V_0$ whatsoever
exists. Similarly, $V_0 = \N\setminus \{1,2\}$ meets the same
requirements, but since $f(U\setminus U_0) \ne g( V\setminus V_0)$ no
bijection $\widetilde{h}:U\setminus U_0 \to V\setminus V_0$ can
possibly have the property that $f|_{U\setminus U_0} = g \circ \widetilde{h}$. 

\begin{proof}[Proof of Lemma \ref{lem4_0}]
Utilizing virtually every tool developed so far, the argument is
elementary but somewhat intricate; for the reader's convenience, it
is divided into five separate steps. Throughout, assume that $\Phi
\stackrel{h}{\thicksim} \Psi$ with $h\in \cH_1(\R^d)$.

\medskip

\noindent
\underline{Step I:} Proof of (i).

\smallskip

\noindent
Note that $L^{\Phi} (-1) = L^{\Psi}(-1) =\R^d$, whereas
$$
L^{\Phi}(-1^-) = L_0^{\Phi} (-1) = \mbox{\rm span}\, \{e_{d-k_0+1},
\ldots , e_d\} \, , \quad
L^{\Psi}(-1^-) = L_0^{\Psi} (-1) = \mbox{\rm span}\, \{e_{d-\ell_0+1},
\ldots , e_d\} \, ,
$$
which shows that $\alpha = 1$ in Theorem \ref{lem3aol1}, and $k_0
=\ell_0$. For convenience, let $L_0 = L_0^{\Phi}(-1) =
L_0^{\Psi}(-1)$, with $\dim L_0 = k_0=\ell_0$, and denote by $P$ the
orthogonal projection of $\R^d$ onto $\mbox{\rm span}\{e_1, \ldots ,
e_{d-k_0}\}$, thus $L_0 = \ker P$. Also, for every $p\in \N$ let
$U^*_p =\bigcup_{q\ge p}U_q = \{1\le j \le k: m_j \ge p\}$, and
similarly $V^*_p = \bigcup_{q\ge p} V_q$. With this, observe that 
$$
\dim L_{p}^{\Phi}(-1) = \dim L_{p-1}^{\Phi} (-1) + \sum\nolimits_{j\in
  U^*_p} {\sf d}(b_j) \, , \quad
\dim L_{p}^{\Psi}(-1) = \dim L_{p-1}^{\Psi} (-1) + \sum\nolimits_{j\in
  V^*_p} {\sf d}(c_j) \, ,
$$
and since $\dim L_p^{\Phi}(-1) = \dim L_p^{\Psi}(-1)$ for every $p \in
\N$ by Theorem \ref{lem3aol1},
\begin{equation}\label{eqpx1}
\sum\nolimits_{j \in U_p} {\sf d}(b_j) = \sum\nolimits_{j \in U^*_p \setminus U^*_{p+1}} {\sf d}(b_j) =
\sum\nolimits_{j \in V^*_p \setminus V^*_{p+1}} {\sf d}(c_j) =
\sum\nolimits_{j \in V_p } {\sf d}(c_j)  \qquad
\forall p \in \N \, .
\end{equation}
Thus (i) holds. Notice in particular that $L_1^{\Phi}(-1) =
L_1^{\Psi}(-1)$ yields ${\sf d}(b) = {\sf d}(c)$. Also,
$m^{\Phi} = m^{\Psi}=:m$ because $m^{\Phi} = \max \{p \in
\N : U_p \ne \varnothing\}$ and  $m^{\Psi} = \max \{p \in
\N : V_p \ne \varnothing\}$.

\medskip

\noindent
\underline{Step II:} Formulae for $\Phi$, $\Psi$.

\smallskip

\noindent
To prepare for the proof of (ii), notice that by permuting the Jordan
blocks $J_{m_j} (-1 + ib_j)$ if necessary, one may assume that $m=m_1
\ge \ldots \ge m_k$. Recall from (\ref{eqlem52}) that
\begin{equation}\label{eqpx2}
  \Phi_t = e^{tA} = e^{-t} R^{\Phi}(t) \left(
P + \sum\nolimits_{j=1}^{m-1}\frac{t^j}{j!} K^j + e^t \, \mbox{\rm
  diag}\, \bigl[O_{d-k_0}, e^{tA_0}\bigr] 
    \right) \qquad \forall t \in \R \, ,
\end{equation}
where $R^{\Phi}(t), K\in \R^{d\times d}$ are given by
$$
R^{\Phi}(t) = \left\{\begin{array}{ll} I_d & \mbox{\rm if $t=0$}, \\
  \mbox{\rm diag}\,\bigl[R_{m_1}(b_1t), \ldots ,
                       R_{m_k}(b_kt), I_{k_0} \bigr] & \mbox{\rm if
                                                       $t\ne 0 $},
                     \end{array}
                     \right.
\quad  K  =  \mbox{\rm diag}\, \bigl[K_{m_1}(b_1), \ldots ,
K_{m_k}(b_k), O_{k_0}\bigr] \, .
$$
Notice that the number of $p\times p$-Jordan blocks appearing in $K$
is $\sum_{j\in U_p} {\sf d}(b_j)$, and hence the total number of
Jordan blocks (of any size) is precisely $\sum_{j\in U_1^*} {\sf
  d}(b_j) = {\sf d}(b)$; thus
$$
K = \mbox{\rm diag}\, \bigl[ \, \underbrace{J_{m_1},
                                                  \ldots ,
                                                  J_{m_k}}_{{\sf
                                                  d}(b)\, 
                                                  {\rm blocks}},
O_{k_0}\bigr] \, . 
$$
Similarly, assume w.l.o.g.\ that $m=n_1 \ge \ldots \ge n_{\ell}$; then
\begin{equation}\label{eqpx3}
  \Psi_t = e^{tB} = e^{-t} R^{\Psi}(t) \left(
P + \sum\nolimits_{j=1}^{m-1}\frac{t^j}{j!} K^j + e^t \, \mbox{\rm
  diag}\, \bigl[O_{d-\ell_0}, e^{tB_0} \bigr] 
    \right) \qquad \forall t \in \R \, ,
\end{equation}
with $R^{\Psi}(t) \in \R^{d\times d}$ given by
$$
R^{\Psi}(t) = \left\{\begin{array}{ll} I_d & \mbox{\rm if $t=0$}, \\
   \mbox{\rm diag}\,\bigl[R_{n_1}(c_1t), \ldots ,
 R_{n_{\ell}}(c_{\ell}t), I_{{\ell}_0}\bigr]  & \mbox{\rm if
                                                       $t\ne 0 $}.
                     \end{array}
                     \right.
$$
By virtute of (\ref{eqpx1}), the same matrix $K$ can be assumed to appear in
 (\ref{eqpx2}) and (\ref{eqpx3}). Also, $K$ 
 commutes with $A$, $B$, $R^{\Phi}(t)$, $R^{\Psi}(t)$, and
 $P=\mbox{\rm diag}\, [I_{d-k_0}, O_{k_0}]$. In fact, $KP=PK=K$, and
 hence $K$ commutes with $\Phi_t$, $\Psi_t$ as well; moreover, $|K^j|=1$ for every
 $j\in \{0, \ldots , m-1\}$, whereas $K^{m}=O_d$. Similarly,
 $K^{\top}$ commutes with $P$, $\Phi_t$, $\Psi_t$. Step V below will
 show that
 \begin{align}\label{eqpx4a}
& \mbox{\rm for every $p\in \N \setminus \{1\}$ either
  $U_p^*=V_p^*=\varnothing$ or else there exists a bijection} \nonumber \\[-4mm]
   & \\[-3mm]
   & g_p: U^*_p \to V^*_p\: \mbox{\rm so that $|b_j| = |c_{g_p(j)}|$ for every $j\in U^*_p$.} \nonumber
\end{align}
Notice that validity of (\ref{eqpx4a}) immediately implies (ii): Indeed, since
$U_p = U_p^*\setminus U^*_{p+1} $ and $V_p = V^*_{p}\setminus
V^*_{p+1}$ for every $p\in \N$, an application of Proposition
\ref{prop2zz} with $U_0 = U^*_{p+1}$, $U=U^*_p$, $V_0 =
V^*_{p+1}$, $V=V^*_p$, $W=\R$ as well as $f(j)=|b_{j}|$, $g(j)=|c_{j}|$ and
$h_0 = g_{p+1}$, $h=g_p$ yields the desired conclusion.

 Since there is nothing to prove in (\ref{eqpx4a}) whenever $m=1$ or $p>m$, henceforth assume
 that $m\ge 2$ and $p\le m$. If $p<m$ then applying Theorem
 \ref{lem5bb} (to determine $D^{\Phi}$) and Lemma
 \ref{lem3_1} (to replace $\Phi$ by $\Phi|_{\R \times
 D^{\Phi}}$) iteratively $m-p$ times yields a flow that formally is
 generated by $A$ as in (\ref{eq4_10}), but with each $m_j$
 replaced by $\min \{m_j, p\}$. Since $\min \{m_1, p\}=p$, no
 generality is lost by assuming $p=m$ from the outset, in which case
 $U_p\ne \varnothing$, $V_p\ne \varnothing$ while also $m_j = m$ for
 every $j\in U_p$ and $n_j = m$ for every $j\in V_p$. In other words,
 it suffices to prove (\ref{eqpx4a}) for $p=m\ge 2$, in which case
 $U_p^* = U_p = \{1, \ldots, \# U_m\}$ and $V_p^* = V_p = \{1, \ldots,
 \# V_m\}$. This will be done in Step V, after two further preparatory steps.
 
 \medskip

\noindent
\underline{Step III:} Construction of auxiliary bounded flows
$\widetilde{\Phi}$, $\widetilde{\Psi}$.

\smallskip

\noindent
Notice first that
$$
L_j:= L_j^{\Phi}(-1) = L_j^{\Psi} (-1) = \ker K^{j} \qquad \forall
j\in \{1, \ldots , m\} \, ,
$$
so in particular $L_1 = \ker K$ and $L_{m-1} = \ker
K^{m-1}$. Moreover,
$$
Q:= (K^{\top})^{m-1} K^{m-1} \in \R^{d\times d}
$$
is the orthogonal projection of $\R^d$ onto $\mbox{\rm range}\,
(K^{\top})^{m-1}$, with $K^{m-1}Q = K^{m-1}$, while $I_d - Q$ is the orthogonal projection onto
$\ker K^{m-1} = L_{m-1}$. Since $K$ commutes with $R^{\Phi}(t)$ and
$P$, and since
$R^{\Phi}(t) = R^{\Phi}(-t)^{\top}$ for all $t\in \R$
as well as $P=P^{\top}$ and $KP=K$,
also $QR^{\Phi}(t) = R^{\Phi}(t) Q$ and $QP=PQ=Q$. Deduce from
(\ref{eqpx2}) that $Qe^{tA} = e^{-t} R^{\Phi}(t) Q$ for all $t\in \R$,
and hence
$$
QA = \frac{{\rm d}}{{\rm d}t} Qe^{tA}\Big|_{t=0} = - Q + \frac{{\rm
    d}}{{\rm d}t} R^{\Phi}(t) \Big|_{t=0} Q\, .
$$
Since $Q^2 = Q$, it follows that $QAQ = QA$, which in turn yields $\bigl(Q(A+I_d)
\bigr)^n=Q^{\min \{n,1\}} (A+I_d)^n$ for every $n\in \N_0$.
Now, consider the flow $\widetilde{\Phi}$ on $\R^d$ generated by
$Q(A+I_d)$. Observe that $\widetilde{\Phi}_t$ is explicitly given by 
\begin{equation}\label{eqqx1}
\widetilde{\Phi}_t = e^{tQ(A+I_d)}  =
                                          \sum\nolimits_{n=0}^{\infty}
                                          \frac{t^n}{n!} 
                                          Q^{\min \{n,1\}}(A+I_d)^n  = 
I_d + Q(e^t e^{tA} - I_d)
                                          = I_d  - Q + Q R^{\Phi}(t) \qquad \forall t \in \R 
    \, ;
\end{equation}
in particular, $\widetilde{\Phi}$ is bounded. Moreover, deduce from
(\ref{eqqx1}) that 
\begin{equation}\label{eqqx2}
K^{m-1} \widetilde{\Phi}_t = R^{\Phi}(t) K^{m-1} \, , \quad
Q\widetilde{\Phi}_t = \widetilde{\Phi}_t Q \qquad \forall t \in
\R \, .
\end{equation}
Also, since
$\widetilde{\Phi}_t - \widetilde{\Phi}_t Q  = I_d - Q$ for all $t\in
\R$, it is clear that
\begin{equation}\label{eqqx2a}
T_x^{\widetilde{\Phi}} = T_{Qx}^{\widetilde{\Phi}} \qquad \forall x
\in \R^d \, .
\end{equation}
Finally, observe that the matrix $Q(A+I_d)\in \R^{d\times d}$ equals
$$
\mbox{\rm diag}\, \bigl[
(K_{m_1} (b_1)^{\top})^{m-1} K_{m_1}(b_1)^{m-1} J_{m_1}(ib_1), \ldots
, (K_{m_k} (b_k)^{\top})^{m-1} K_{m_k}(b_k)^{m-1} J_{m_k}(ib_k), O_{k_0}
\bigr] \, .
$$
Since $K_{m_j} (b_{j})^{m-1} J_{m_j}(ib_j) =O_{m_j \, {\sf
    d}(b_j)}$ whenever $m_j < m$ or $b_j = 0$, whereas
$$
(K_m (b)^{\top})^{m-1} K_m(b)^{m-1} J_m(ib) = \left[
  \begin{array}{lr|lr}
     0 &  \cdots & \cdots & 0 \\
    \vdots & \ddots & & -b \\ \hline
    \vdots & & \ddots & \vdots  \\
    0 & b & \cdots  &  0    \end{array}
  \right]\qquad \forall b \in \R\setminus \{0\} \, ,
  $$
there exists a permutation matrix $Q_A \in \R^{d\times d}$ with
$Q_A^{-1} = Q_A^{\top} = Q_A$ so that
\begin{equation}\label{eqrx1}
Q_A Q(A+I_d) Q_A = \mbox{\rm diag}\, \bigl[
J_1(ib_1), \ldots , J_1(ib_{\# U_m}) , O_{d_0}
\bigr] \, ,
\end{equation}
where $d_0= d - \sum_{j\in U_m} {\sf d}(b_j)$.

In complete analogy, let the flow $\widetilde{\Psi}$ on $\R^d$ be generated by
$Q(B+I_d)$. Then (\ref{eqqx1})--(\ref{eqqx2a}) remain valid with $\Phi$, $A$
replaced by $\Psi$, $B$ respectively, and since also $d_0 = d - \sum_{j\in
  V_m}{\sf d}(c_j)$ by (i), there exists a permutation matrix $Q_B\in
\R^{d\times d}$ with $Q_B^{-1} = Q_B^{\top} = Q_B$ so that
$$
Q_B Q(B+I_d) Q_B = \mbox{\rm diag}\, \bigl[
J_1(ic_1), \ldots , J_1(ic_{\# V_m}) , O_{d_0}
\bigr]  \, .
$$

 \medskip

\noindent
\underline{Step IV:} Distortion ratios for
$\Phi$, $\Psi$.

\smallskip

\noindent
Pick any $x\in \R^d \setminus L_{m-1}$ and let
$y = K^{m-1}x \in L_1 \setminus L_0$ for convenience. Note that $Py =
y$, and by (\ref{eqpx2}),
$$
|\Phi_t x| = e^{-t} \left|
Px + \sum\nolimits_{j=1}^{m-1} \frac{t^j}{j!} K^j x + e^t \, \mbox{\rm
  diag}\,\bigl[O_{d-k_0}, e^{tA_0}\bigr] x
\right| \, , \quad
|\Phi_t y| = e^{-t} |y | \qquad \forall t \in \R \, .
$$
Thus, with the function $\rho_x :\R\to \R$ given by
\begin{equation}\label{eqqx5}
  \rho_x (t) = t + \log |y| - \log \left|
Px + \sum\nolimits_{j=1}^{m-1} \frac{t^j}{j!} K^j x + e^t \, \mbox{\rm
  diag}\,\bigl[O_{d-k_0}, e^{tA_0}\bigr] x
\right| \qquad \forall t \in \R \, ,
\end{equation}
clearly $|\Phi_t x| = |\Phi_{\rho_x(t)}y|$ for every $t\in
\R$, and $\rho_x$ is smooth (though not necessarily monotone) with 
\begin{equation}\label{eqqx5a}
  \rho_x(t) = t - (m-1) \log t  + \log (m-1)! + o(1) \qquad \mbox{\rm
    as } t \to \infty \, ;
\end{equation}
in particular, $\rho_x \in \cR$. Moreover, deduce from (\ref{eqpx2}) that for every sufficiently
large $t$,
\begin{align*}
  \left|
\frac{\Phi_t x - \Phi_{\rho_x(t)}y}{|\Phi_{\rho_x(t)}y|} - \Bigl(
  R^{\Phi}(t) - R^{\Phi} \bigl( \rho_x(t) \bigr) \Bigr) \frac{y}{|y|}
  \right|  & = \left|
\frac{Px + \sum_{j=1}^{m-1} t^j K^jx/j! + e^t \, \mbox{\rm diag}\,
             \bigl[O_{d-k_0}, e^{tA_0}\bigr]x }{\big|Px + \sum_{j=1}^{m-1}
             t^j K^jx/j! + e^t \, \mbox{\rm diag}\, \bigl[O_{d-k_0},
             e^{tA_0}\bigr]x \big|} - \frac{y}{|y|}
             \right|  \\
           & \le \frac{2|x| \bigl( \sum_{j=0}^{m-2} |t|^j/j! + e^t
             |e^{tA_0}|\bigr)}{\big| Px +\sum_{j=1}^{m-1} t^j
    K^j x/j! + e^t \, \mbox{\rm diag}\,
             \bigl[O_{d-k_0}, e^{tA_0}\bigr]x \big|} \, ,
\end{align*}
from which it is clear that
\begin{equation}\label{eqqx5b}
  \lim\nolimits_{t\to \infty}
  \left(
\frac{\Phi_t x - \Phi_{\rho_x(t)}y}{|\Phi_{\rho_x(t)}y|} - \Bigl(
  R^{\Phi}(t) - R^{\Phi} \bigl( \rho_x(t) \bigr) \Bigr) \frac{y}{|y|}
  \right) = 0 \, .  
\end{equation}
By Theorem \ref{lem3aol1}, $h(x)\in \R^d \setminus L_{m-1}$ and $h(y)\in
L_1 \setminus L_0$, so $K^{m-1}h(x)\ne 0$ and $Ph(y)\ne 0$. Clearly, $\tau_x$, $\tau_y$ are
increasing. For convenience, write $\tau_y \circ \rho_x$ simply as
$\sigma_x$, and notice that $\sigma_x \in \cR$. It is readily seen
that, in analogy to
(\ref{eqqx5b}),  
\begin{equation}\label{eqqx6}
  \lim\nolimits_{t\to \infty}
  \left(
\frac{\Psi_{\tau_x(t)}h( x) -
  \Psi_{\sigma_x(t)}h(y)}{|\Psi_{\sigma_x(t)}h(y)|} - \left( e^{-\gamma_x(t)}
  R^{\Psi}\bigl( \tau_x(t)\bigr)  \frac{K^{m-1}h(x)}{|K^{m-1}h(x)|} -
  R^{\Psi} \bigl( \sigma_x(t) \bigr) \frac{Ph(y)}{|Ph(y)|}\right)
\right) = 0 \, , 
\end{equation}
with the continuous function $\gamma_x : \R \to \R$ given by
\begin{equation}\label{eqqx15A}
\gamma_x(t) = \tau_x(t) - t - \bigl( \sigma_x(t) - \rho_x(t) \bigr) -
\log |K^{m-1} h(x)| + \log |Ph(y)| \qquad \forall t \in \R \, .
\end{equation}
Notice that $\sup_{t\ge 0}|\gamma_x(t)|<\infty$, also by Theorem
\ref{lem3aol1}.

 \medskip

\noindent
\underline{Step V:} Proof of (ii).

\smallskip

\noindent
To conclude the argument, pick any $x\in \R^d$; as in Step IV, let
$y=K^{m-1}x$ for convenience.

First assume that $x\in L_{m-1}$. Then $Qx=0$
 so $x\in \mbox{\rm Fix}\,
\widetilde{\Phi}$, but also $h(x)\in L_{m-1}$ and thus likewise $h(x) \in \mbox{\rm Fix}\,
\widetilde{\Psi}$.

Next assume that $x\not \in L_{m-1}$ but $x\in \mbox{\rm Fix}\,
\widetilde{\Phi}$. (This is possible only if $b_j=0$ for some $j$.)
Then by (\ref{eqqx2}),
$$
0 = K^{m-1}\bigl(\widetilde{\Phi}_t - \widetilde{\Phi}_{\rho_x (t)}\bigr)x =
\bigl( R^{\Phi} (t) - R^{\Phi} \bigl( \rho_x (t) \bigr) \bigr) \qquad
\forall t \in \R \, ,
$$
and hence, by (\ref{eqqx5b}) and Lemma \ref{prop3},
$$
\lim\nolimits_{t\to \infty} \frac{\Phi_t x -
  \Phi_{\rho_x(t)}y}{|\Phi_{\rho_x(t)}y|} = 0 = \lim\nolimits_{t\to
  \infty} \frac{\Psi_{\tau_x(t)}h( x) -
  \Psi_{\sigma_x(t)}h(y)}{|\Psi_{\sigma_x(t)}h(y)|} \, .
$$
Since $\sup_{t\ge 0}|\gamma_x(t)|<\infty$, by
(\ref{eqqx6}) also
\begin{equation}\label{eqqx6a}
\lim\nolimits_{t\to \infty} \left(
  R^{\Psi}\bigl( \tau_x(t)\bigr)  \frac{K^{m-1}h(x)}{|K^{m-1}h(x)|} - e^{\gamma_x(t)}
  R^{\Psi} \bigl( \sigma_x(t) \bigr) \frac{Ph(y)}{|Ph(y)|}\right) = 0 \, ,
\end{equation}
and since $R^{\Psi}(t)$ is an isometry necessarily $\lim_{t\to \infty}
\gamma_x (t) = 0$, which in turn implies
$$
\lim\nolimits_{t\to \infty} R^{\Psi} \bigl( \tau_x(t) - \sigma_x(t)
\bigr) \frac{K^{m-1} h(x)}{|K^{m-1}h(x)|} = \frac{Ph(y)}{|Ph(y)|} \, .
$$
Applying $(K^{\top})^{m-1}$ and recalling the
$\widetilde{\Psi}$-analogue of (\ref{eqqx2}) yields
\begin{equation}\label{eqqx15B}
\lim\nolimits_{t\to \infty} \widetilde{\Psi}_{\tau_x(t) - \sigma_x(t)}
Qh(x) = \frac{|K^{m-1}h(x)|}{|Ph(y)|}
(K^{\top})^{m-1} Ph(y)  \, .
\end{equation}
By (\ref{eqqx5a}) and (\ref{eqqx15A}), and with 
$c:= \log(|K^{m-1}h(x)|/|Ph(y)|(m-1)!)$ for convenience, observe that
\begin{align}\label{eqqx7}
\tau_x(t) - \sigma_x (t) & = -\rho_x(t) + t + \gamma_x(t) +
                          \log|K^{m-1}h(x)| - \log |Ph(y)|\nonumber \\
  & = (m-1) \log t + c + o(1) \qquad
\mbox{\rm as } t \to \infty \, .
\end{align}
From this and (\ref{eqqx15B}), it is clear that $\lim_{t\to \infty}
\widetilde{\Psi}_t Qh(x)$ exists, so $Qh(x)\in \mbox{\rm Fix}\,
\widetilde{\Psi}$ by Proposition \ref{propxy}, and hence also $h(x)\in
\mbox{\rm Fix}\, \widetilde{\Psi}$ by the $\widetilde{\Psi}$-analogue
of (\ref{eqqx2a}). In other words, $h(x)\in \mbox{\rm Fix}\,
\widetilde{\Psi}$ whenever $x\in \mbox{\rm Fix}\,
\widetilde{\Phi}\setminus L_{m-1}$. As seen earlier, $h(\mbox{\rm
  Fix}\, \widetilde{\Phi} \cap L_{m-1}) = h(L_{m-1}) = L_{m-1} \subset
\mbox{\rm Fix}\, \widetilde{\Psi}$. In summary, $h(\mbox{\rm Fix}\,
\widetilde{\Phi}) \subset \mbox{\rm Fix}\,
\widetilde{\Psi}$, and interchanging the roles of $\widetilde{\Phi}$,
$\widetilde{\Psi}$ yields $h(\mbox{\rm Fix}\,
\widetilde{\Phi}) = \mbox{\rm Fix}\,
\widetilde{\Psi}$.

Next assume that $x\in \mbox{\rm Per}\, \widetilde{\Phi}$ with
$T:=T_x^{\widetilde{\Phi}} \in \R^+$. Recall that $\rho_x$ is smooth and $\rho_x(t) - t \to -\infty$ as
$t\to \infty$, by (\ref{eqqx5a}) and since $m\ge 2$. It is possible, therefore, to choose
$j_0\in \Z$ and an increasing sequence $(t_n)$ with $t_n \to \infty$
so that
$$
\rho_x(t_n) - t_n = -(n+j_0) T \qquad \forall n \in \N \, .
$$
With this, 
$$
0 = K^{m-1} \bigl(\widetilde{\Phi}_{t_n} - \widetilde{\Phi}_{\rho_x(t_n)}\bigr) x
= \bigl(
R^{\Phi} (t_n) - R^{\Phi} \bigl( \rho_x(t_n)\bigr)
\bigr) y  \qquad \forall n \in \N \, ,
$$
and identical reasoning as before yields a discrete-time
analogue of (\ref{eqqx15B}), namely
$$
 \lim\nolimits_{n\to \infty} \widetilde{\Psi}_{\tau_x(t_n) - \sigma_x(t_n)}
Qh(x)
= \frac{|K^{m-1}h(x)|}{|Ph(y)|}  (K^{\top})^{m-1} Ph(y)  \, .
$$
Correspondingly, the discrete-time analogue of (\ref{eqqx7}) reads 
$$
\tau_x(t_n) - \sigma_x(t_n) = nT  +  c + \log (m-1)! + j_0 T + o(1)
\qquad \mbox{\rm as } n\to \infty \, .
$$ 
Thus $Qh(x) \in  \mbox{\rm Per}_T \widetilde{\Psi}$ by Proposition \ref{propxy}, and the $\widetilde{\Psi}$-analogue of
(\ref{eqqx2a}) yields $h(x)\in  \mbox{\rm Per}_T \widetilde{\Psi}$. By
what has been shown previously, $h(x)\not \in  \mbox{\rm Fix}\,
\widetilde{\Psi}$, and so $T/T_{h(x)}^{\widetilde{\Psi}}\in
\N$. In summary, $h(x)\in  \mbox{\rm Per}\, \widetilde{\Psi}\setminus
\mbox{\rm Fix}\, \widetilde{\Psi}$ whenever $x\in  \mbox{\rm Per}\, \widetilde{\Phi}\setminus
\mbox{\rm Fix}\, \widetilde{\Phi}$, and in this case
$T_x^{\widetilde{\Phi}}/ T_{h(x)}^{\widetilde{\Psi}} \in
\N$. Interchanging the roles of $\widetilde{\Phi}$, $\widetilde{\Psi}$
yields
$$
h( \mbox{\rm Per}\, \widetilde{\Phi} ) = \mbox{\rm Per}\,
\widetilde{\Psi}  \quad \mbox{\rm and} \quad T_x^{\widetilde{\Phi}} =
T_{h(x)}^{\widetilde{\Psi}} \enspace \forall x \in \mbox{\rm Per}\,
\widetilde{\Phi} \, .
$$
Finally, observe that $T_x^{\widetilde{\Phi}} = \infty =
T_{h(x)}^{\widetilde{\Psi}} $ whenever $x\not \in  \mbox{\rm Per}\,
\widetilde{\Phi}$ or equivalently $h(x)\not \in  \mbox{\rm Per}\,
\widetilde{\Psi}$. Altogether, therefore,
\begin{equation}\label{eqrx5}
 T_x^{\widetilde{\Phi}} =
T_{h(x)}^{\widetilde{\Psi}}  \qquad \forall x\in \R^d \, .
\end{equation}

Recall now that the flows $Q_A \widetilde{\Phi} Q_A$, $Q_B
\widetilde{\Psi} Q_B$ are generated by
$$
\mbox{\rm diag}\, \bigl[
J_1(ib_1), \ldots , J_1(ib_{\# U_m}) , O_{d_0}
\bigr] \, , \quad
\mbox{\rm diag}\, \bigl[
J_1(ic_1), \ldots , J_1(ic_{\# V_m}) , O_{d_0}
\bigr] 
$$ 
respectively. In particular, with $\widetilde{h}:\R^d \to \R^d$ given
by $\widetilde{h} (z) = Q_B h(Q_A z)$ for all $z\in \R^d$, it is clear
that $\widetilde{h}\in \cH_{1}(\R^d)$, and (\ref{eqrx5}) provides the
crucial middle equality in
$$
T_z^{Q_A \widetilde{\Phi} Q_A} = T_{Q_A z}^{\widetilde{\Phi}}
= T_{h( Q_A z)}^{\widetilde{\Psi}} = T^{Q_B \widetilde{\Psi}
  Q_B}_{\widetilde{h} (z)} \qquad \forall z \in \R^d \, .
$$
An application of Proposition \ref{prop1zz} yields a
bijection $g: U_m \to V_m$ so
that $|b_j| = |c_{g(j)}|$ for every $j\in U_m$. This proves
(\ref{eqpx4a}) for $p=m$, which in turn implies (ii) as previously
discussed.
\end{proof}

To appreciate the immediate relevance of Lemma \ref{lem4_0} for the
Lipschitz classification problem, it is worthwhile explicitly stating two
corollaries for the simplest special case, namely when all Lyapunov
exponents are $-1$ and all irreducible components have the same size
$m$. As such, consider (\ref{eq4_10}) with $m_1 = \ldots =
n_{\ell}=m\in \N$ and $k_0 =
\ell_0 = 0$, resulting in
\begin{equation}\label{eq4_1}
A  = \mbox{\rm diag}\, \bigl[ J_m (-1+ib_1), \ldots ,
J_m(-1+ib_k)\bigr] \, ,  \quad 
 B  =  \mbox{\rm diag}\, \bigl[ J_m (-1+ic_1), \ldots ,
 J_m(-1+ic_{\ell}) \bigr] \, ;
\end{equation}
note that (\ref{eq4_8}) now simply reads $m\, {\sf d}(b) = m\, {\sf d}(c) = d$.
The Lipschitz equivalence of $\Phi$, $\Psi$ generated by $A$,
$B$ respectively plays out very differently
 depending on whether $m=1$ or $m\ge 2$. For $m=1$, it turns out that
 $\Phi$, $\Psi$ automatically are Lipschitz equivalent.

 \begin{lem}\label{lem4_1}
Given $k,\ell\in \N$ and $b\in \R^k$, $c\in
\R^{\ell}$ with ${\sf d}(b) = {\sf d}(c)=d$, let $\Phi$, $\Psi$ be the flows on $\R^d$ generated by
$A$, $B$ in {\rm (\ref{eq4_1})} respectively, with $m=1$. Then $A$,
$B$ are Lipschitz similar, and $\Phi \stackrel{1}{\cong}\Psi$.  
\end{lem}

\begin{proof}
Letting $\widetilde{\Phi}$ be the flow on $\R^d$ generated by $-I_d$,
it will be shown that $A$ is Lipschitz similar to $-I_d$ and that $\Phi \stackrel{1}{\cong}
\widetilde{\Phi}$. Once established, clearly this proves both
assertions of the lemma.

First, that $A$, $-I_d$ are Lipschitz similar follows directly from the
definition of $\cL$ since
$$
\cL A = \mbox{\rm diag}\, [\cL J_1(-1+ib_1), \ldots , \cL J_1(-1 +
ib_k)] = \mbox{\rm diag}\, [- I_{{\sf d}(b_1)}, \ldots , -I_{{\sf
    d}(b_k)} ] = - I_{{\sf d}(b)} \, .
$$

Next, let $h_0 = I_{\R}$, and for $a\in \R\setminus \{0\}$ let the map $h_a : \R^2 \to \R^2$
be given by
\begin{equation}\label{eq5_add20}
h_a(y) = \left\{
  \begin{array}{ll}
    y & \mbox{\rm if $y=0$ or $|y|=1$}\, , \\
    R_1(a\log|y|) \, y & \mbox{\rm otherwise} \, .
  \end{array}
    \right.
\end{equation}
It is readily confirmed that $h_a \in \cH_1 (\R^2)$, with $h_a^{-1} =
h_{-a}$ and
$$
|h_a(y) - h_a(z)| \le (1+|a|) |y-z|\qquad \forall y,z \in \R^2 \, .
$$
Moreover, for every $t\in \R\setminus \{0\}$, $y\in \R^2\setminus
\{0\}$,
$$
h_a \bigl(e^{tJ_1(-1+ia)} y\bigr) = h_a (e^{-t} R_1(at)\, y) = e^{-t} R_1(a
\log|e^{-t}y|) R_1(at) \, y = e^{-t} h_a(y) \, ,
$$
and the two outer-most expressions agree also if $t=0$ or $y=0$. In
summary, for every $a\in \R$,
$$
h_a\bigl(e^{tJ_1(-1+ia)}y\bigr) = e^{-t} h_a(y) \qquad \forall t \in \R, y \in
\R^{{\sf d}(a)} \, .
$$

Now, let $\R^d=\bigoplus_{j=1}^k E_j$ be the decomposition of $X=\R^d$ into
(mutually orthogonal) irreducible components, with $\dim E_j = {\sf
  d}(b_j)$ for every $j\in \{1, \ldots , k\}$, and denote by $P_j$ the
orthogonal projection of $\R^d$ onto $E_j$. Furthermore let $H_j : E_j \to
\R^{{\sf d}(b_j)}$ be an isometric isomorphism with $H_j \Phi_t x =
e^{t J_1(-1+ib_j)}H_j x$ for all $t\in \R$, $x\in E_j$. With this,
$$
\Phi_t x = \sum\nolimits_{j=1}^k \Phi_t P_j x = \sum\nolimits_{j=1}^k
H_j^{-1} e^{tJ_1(-1+ib_j)} H_j P_j x \qquad \forall t\in \R , x  \in  \R^d \, .
$$
Consider $h:\R^d \to  \R^d  $ given by $h = \bigtimes_{j=1}^k h_{b_j}$, that is,
$$
h(x) = \sum\nolimits_{j=1}^k H_j^{-1} h_{b_j} (H_j P_j x) 
\qquad \forall x \in \R^d  \, .
$$
Clearly, $h\in \cH_1(\R^d)$ and
\begin{align*}
h(\Phi_t x) & =\sum\nolimits_{j=1}^k H_j^{-1} h_{b_j} (H_j P_j \Phi_t
              x) =  \sum\nolimits_{j=1}^k H_j^{-1} h_{b_j} (e^{tJ_1(-1 + i
              b_j)} H_j P_j x)  \\
  & = \sum\nolimits_{j=1}^k H_j^{-1} \left( e^{-t} h_{b_j} (H_j P_j x)
    \right) = e^{-t} h(x) = \widetilde{\Phi}_t h(x) \qquad \forall
    t\in \R ,x \in \R^d \, . 
\end{align*}
In other words, $\Phi \stackrel{h}{\cong} \widetilde{\Phi}$, and hence $\Phi
\stackrel{1}{\cong}\widetilde{\Phi}$ as claimed.    
\end{proof}

In stark contrast to Lemma \ref{lem4_1}, for $m\ge 2$ the flows
$\Phi$, $\Psi$ can be Lipschitz equivalent only in the obvious way,
that is, only if their generators are (Lipschitz) similar.

\begin{lem}\label{lem4_2}
Given $m\in \N \setminus \{1\}$, $k,\ell\in \N$, and $b\in \R^k$, $c\in
\R^{\ell}$ with $m\, {\sf d}(b) = m\, {\sf d}(c) = d$, let $\Phi$, $\Psi$ be the flows on $\R^d$ generated by
$A$, $B$ in {\rm (\ref{eq4_1})} respectively. Then the following
statements are equivalent:
\begin{enumerate}
\item $\Phi \stackrel{{\sf lin}}{\cong}\Psi$;
\item $\Phi \stackrel{1}{\thicksim} \Psi$;
  \item $k=\ell$, and there exists a bijection $g:\{1, \ldots , k\}\to \{1, \ldots
    , \ell\}$ so that $|b_j| = |c_{g(j)}|$ for every $j$;
\item $A$, $B$ are similar;
\item $A$, $B$ are Lipschitz similar.    
\end{enumerate}
\end{lem}

\begin{proof}
Clearly, (i)$\Rightarrow$(ii) and
(iii)$\Rightarrow$(iv)$\Rightarrow$(v), as well as
(iv)$\Rightarrow$(i). By defintion, $\cL J_m(z) = J_m(z)$ for every $m\in
\N \setminus \{1\}$ and $z\in \C$, and hence also
(v)$\Rightarrow$(iv). Finally, Lemma \ref{lem4_0} shows that
(ii)$\Rightarrow$(iii) since $U_p = \{1, \ldots, k\}$, $V_p = \{1,
\ldots , \ell\}$ if $p=m \in \N\setminus \{1\}$, and $U_p = V_p =
\varnothing$ otherwise.
\end{proof}

At last, all the necessary tools have been assembled for a proof of the
main result of this article.

\begin{proof}[Proof of Theorem \ref{thm1x}]
Notice first that (iii)$\Leftrightarrow$(iv) by the definitions of
Lyapunov and Lipschitz similarity, and clearly
(ii)$\Rightarrow$(i). Thus it suffices to show that
(i)$\Rightarrow$(iii)$\Rightarrow$(ii), and this will now be done in
two separate steps. A third step then adresses the additional assertion regarding Lipschitz
conjugacy.

\medskip

\noindent
\underline{Step I:} Proof of (i)$\Rightarrow$(iii).

\smallskip

\noindent
Assume that $\Phi
\stackrel{1}{\thicksim}\Psi$. Then $\Phi
\stackrel{1^-}{\thicksim}\Psi$, and by Proposition \ref{prop1za} there
exists $\beta \in \R \setminus \{0\}$ so that $A^{\Phi}$, $\beta
A^{\Psi}$ are Lyapunov similar while $A^{\Phi_{\sf C}}$, $\beta
A^{\Psi_{\sf C}}$ are similar. Replacing $\Psi$ by $\Psi_{*\beta}$
otherwise, it can be assumed that $\beta = 1$; in other words,
$\Lambda^{\Phi} = \Lambda^{\Psi}$, while $A^{\Phi_{\sf C}}$,
$A^{\Psi_{\sf C}}$ are similar, and hence $d^{\Phi}_{\sf C} =
d^{\Psi}_{\sf C}$. 

Now, if $\Phi$ is central then
so is $\Psi$, and $A^{\Phi} = A^{\Phi_{\sf C}} $, $A^{\Psi}=
A^{\Psi_{\sf C}} $ are similar, so clearly (iii) holds. It remains to consider the case
where $\Phi$, $\Psi$ are {\em not\/} central. In this case,
in order to establish (iii), all that needs to be shown is that
$A^{\Phi_{\sf AD}}$, $A^{\Psi_{\sf AD}}$ are similar as well.
No assertion in the theorem is
affected at all by a linear change of coordinates. Without loss of
generality, therefore, assume that
\begin{align}\label{eq5_p1}
  A^{\Phi} & = \mbox{\rm diag} \, \bigl[
J_{m_1} (z_1), \ldots , J_{m_k} (z_k), J_1(z_{k+1}) , \ldots ,
             J_1(z_{k+k_0} )   , A^{\Phi_{\sf C}}
             \bigr] \, , \nonumber \\[-3mm]
           & \: \\[-3mm]
  A^{\Psi} & = \mbox{\rm diag} \, \bigl[
J_{n_1} (w_1), \ldots , J_{n_{\ell}} (w_{\ell} ), J_1(w_{\ell+1}) ,
             \ldots ,  J_1(w_{\ell+\ell_0} )  , A^{\Phi_{\sf C}}
             \bigr] \, , \nonumber 
  \end{align}
with the appropriate
\begin{align*}
  & k,k_0, \ell, \ell_0  \in \N_0\, , m_1, \ldots , m_k , n_1, \ldots , n_{\ell} \in
    \N \setminus \{1\}\, , \\
  & z_1, \ldots , z_{k+k_0}, w_1, \ldots ,
    w_{\ell+\ell_0}\in \{ z\in \C : {\sf Re}\, z\ne 0, {\sf Im}\, z\ge 0 \}
    \, ,
\end{align*}
so that $k+k_0 \ge 1$, $ \ell + \ell_0 \ge 1$ and
$$
\sum\nolimits_{j=1}^k m_j \, {\sf d}({\sf Im}\, z_j) + \sum\nolimits_{j=1}^{k_0}
{\sf d}({\sf Im}\, z_{k+j}) =
\sum\nolimits_{j=1}^{\ell} n_j\,  {\sf d}({\sf Im}\, w_j) + \sum\nolimits_{j=1}^{\ell_0}
{\sf d}({\sf Im}\, w_{\ell+j})  = d - d^{\Phi}_{\sf C}\, ;
$$
here $J_{m_j}(\cdot)$, $J_1 (\cdot)$, and $A^{\Phi_{\sf C}}$ are
understood to be present in $A^{\Phi}$ in (\ref{eq5_p1}) only if
$k\ge 1$, $k_0\ge 1$, and $d^{\Phi}_{\sf C}\ge 1$ respectively, and analogously for
$A^{\Psi}$.

Assume first that $\Phi$, $\Psi$ are stable. Then $d_{\sf
  C}^{\Phi}=0$, so $A^{\Phi_{\sf C}}$ is not present in
(\ref{eq5_p1}), and $z_j,w_j\in \C^-$ for all $j$. If $\Phi$ is
diagonal, then so is $\Psi$ by Proposition \ref{prop3aoo1} and Theorem
\ref{lem3aol1}, and there is nothing to prove because $\Phi_{\sf
  AD}$, $\Psi_{\sf AD}$ are trivial. Otherwise $k,\ell\ge1$ in (\ref{eq5_p1}), and 
$$
A^{\Phi_{\sf AD}} = \mbox{\rm diag} \, \bigl[
J_{m_1} (z_1), \ldots , J_{m_k} (z_k)             \bigr] \, , \quad
A^{\Psi_{\sf AD}} = \mbox{\rm diag} \, \bigl[
J_{n_1} (w_1), \ldots , J_{n_{\ell}} (w_{\ell})             \bigr] \, .
$$
Now, fix any (necessarily negative) Lyapunov exponent $s$. Observe
that
\begin{equation}\label{eq5_p2}
h \bigl( L^{\Phi_{*1/|s|}} (-1)\bigr) = h \bigl( L^{\Phi} (s)\bigr) =
L^{\Psi} (s) = L^{\Psi_{*1/|s|}} (-1) \, ,
\end{equation}
where the middle equality is due to Proposition \ref{prop2bb}. For
convenience, denote by $\widetilde{\Phi}$ the restriction of
$\Phi_{*1/|s|}$ to $\R \times L^{\Phi}(s)$ and similarly by $\widetilde{\Psi}$ the restriction of
$\Psi_{*1/|s|}$ to $\R \times L^{\Psi}(s)$. Clearly
$\widetilde{\Phi}\stackrel{1}{\thicksim} \widetilde{\Psi}$ by
(\ref{eq5_p2}). Permuting the Jordan blocks in (\ref{eq5_p1}) if
necessary, it can be assumed that
\begin{align}\label{eq5_p3}
  A^{\widetilde{\Phi}} & = \mbox{\rm diag}  \left[
J_{m_1} \left( \frac{z_1}{|s|}\right), \ldots , J_{m_{\widetilde{k}}}
                         \left(\frac{z_{\widetilde{k}}}{|s|}\right),
                         J_1\left(\frac{z_{k+1}}{|s|}\right) ,  \ldots ,
             J_1\left( \frac{z_{k+\widetilde{k}_0}}{|s|} \right)   , A_0            \right] \, , \nonumber \\[-2.5mm]
                       & \: \\[-2.5mm]
A^{\widetilde{\Psi}} & = \mbox{\rm diag}  \left[
J_{n_1} \left( \frac{w_1}{|s|}\right), \ldots , J_{n_{\widetilde{\ell}}}
                         \left(\frac{w_{\widetilde{\ell}}}{|s|}\right),
                         J_1\left(\frac{w_{\ell+1}}{|s|}\right) ,  \ldots ,
             J_1\left( \frac{w_{\ell+\widetilde{\ell}_0}}{|s|} \right)   , B_0            \right] \, , \nonumber 
\end{align}
with the appropriate
$0\le \widetilde{k} \le k, \, 0\le \widetilde{k}_0\le k_0 , \, 0\le
\widetilde{\ell}\le \ell $ and $ \,  0\le
\widetilde{\ell}_0\le \ell_0 $
so that
$$
{\sf Re}\, z_j = s \quad \forall j \in \bigl\{1, \ldots ,
\widetilde{k}\bigr\}\cup \bigl\{k+1, \ldots , k+ \widetilde{k}_0 \bigr\} \, , \qquad
{\sf Re}\, w_j = s \quad \forall j \in \bigl\{1, \ldots ,
\widetilde{\ell}\bigr\}\cup \bigl\{\ell+1, \ldots , \ell+ \widetilde{\ell}_0 \bigr\} \, ,
$$
and
\begin{align*}
\sum\nolimits_{j=1}^{\widetilde{k}} m_j \, {\sf d}({\sf Im}\, z_j) + \sum\nolimits_{j=1}^{\widetilde{k}_0}
{\sf d}({\sf Im}\, z_{k+j}) & =
\sum\nolimits_{j=1}^{\widetilde{\ell}} n_j \, {\sf d}({\sf Im}\, w_j) + \sum\nolimits_{j=1}^{\widetilde{\ell}_0}
                              {\sf d}({\sf Im}\, w_{\ell+j})  \\
  & =  \mbox{\rm dim}\, L^{\Phi}(s) -
\mbox{\rm dim}\, L^{\Phi} (s^-)\, ,
\end{align*}
where $J_{m_{\widetilde{k}}}(z_{\widetilde{k}}/|s|)$,
$J_1(z_{k+\widetilde{k}_0}/|s|)$, $J_{n_{\widetilde{\ell}}}(w_{\widetilde{\ell}}/|s|)$,
$J_1(w_{\ell+\widetilde{\ell}_0}/|s|)$ only appear in (\ref{eq5_p3})
if $\widetilde{k}, \widetilde{k}_0, \widetilde{\ell},
\widetilde{\ell}_0\ge 1$ respectively, and
$A_0,B_0\in \R^{\mbox{\footnotesize {\rm dim}}\, L^{\Phi} (s^-)
  \times \mbox{\footnotesize {\rm
    dim}}\, L^{\Phi} (s^-)}$ only appear if
$\mbox{\rm dim}\, L^{\Phi} (s^-)\ge 1$, in which case
$$
\lim\nolimits_{t\to\infty} e^t (|e^{tA_0}| + |e^{tB_0}|) = 0 \, .
$$
Note that
$$
\frac{z_j}{|s|} = - 1 +i \frac{{\sf Im}\, z_j}{|s|} \qquad \forall j \in \bigl\{1, \ldots ,
\widetilde{k}\bigr\}\cup \bigl\{k+1, \ldots , k+ \widetilde{k}_0
\bigr\} \, ,
$$
and similarly
$$
\frac{w_j}{|s|} = - 1 +i \frac{{\sf Im}\, w_j}{|s|} \qquad \forall j \in \bigl\{1, \ldots ,
\widetilde{\ell}\bigr\}\cup \bigl\{\ell+1, \ldots , \ell+ \widetilde{\ell}_0
\bigr\} \, .
$$
It follows from Lemma \ref{lem4_0} that $\widetilde{k} = \widetilde{\ell}$,
and if $\widetilde{k} = \widetilde{\ell}\ge 1$ then there exists a
bijection $g:\bigl\{1, \ldots , \widetilde{k}\bigr\}\to \bigl\{1,
\ldots , \widetilde{\ell} \bigr\}$ so that $J_{m_j} (z_j) =
J_{n_{g(j)}}(w_{g(j)})$ for every $j\in \bigl\{ 1, \ldots ,
\widetilde{k}\bigr\}$. In other words, the two matrices
$$
A^{\Phi_{\sf AD}}_{[s]}:= \mbox{\rm diag}\, [J_{m_j} (z_j): {\sf
  Re}\, z_j = s ] \, , \quad
A^{\Psi_{\sf AD}}_{[s]}:= \mbox{\rm diag}\, [J_{n_j} (w_j): {\sf
  Re}\, w_j = s ] 
$$
are similar, provided that $\widetilde{k} = \widetilde{\ell}\ge
1$. Repeating this argument for every Lyapunov exponent $s$ shows that
$A^{\Phi_{\sf AD}}$, $A^{\Psi_{\sf AD}}$ are similar as well. As
discussed earlier, this proves that (i)$\Rightarrow$(iii) for stable
flows.

In general, i.e., if $\Phi$, $\Psi$ are not stable, applying the same
argument to $\Phi_{\sf S}$, $\Psi_{\sf S}$ shows that $A^{(\Phi_{\sf
    AD})_{\sf S}}$, $A^{(\Psi_{\sf
    AD})_{\sf S}}$ are similar whenever non-trivial, and applying
it to $(\Phi_{\sf U})^*$, $(\Psi_{\sf U})^*$ shows that $A^{(\Phi_{\sf
    AD})_{\sf U}}$, $A^{(\Psi_{\sf
    AD})_{\sf U}}$ are similar whenever non-trivial. Since
$A^{\Phi_{\sf C}}$, $A^{\Psi_{\sf C}}$ are similar, obviously $A^{(\Phi_{\sf
    AD})_{\sf C}}$, $A^{(\Psi_{\sf
    AD})_{\sf C}}$ are similar as well. In summary, therefore,
$A^{\Phi_{\sf AD}}$, $A^{\Psi_{\sf AD}}$ are similar, and (iii) holds.

\medskip

\noindent
\underline{Step II:} Proof of (iii)$\Rightarrow$(ii).

\smallskip

\noindent
Assuming (iii), one may again take $\beta = 1$, so $A^{\Phi}$,
$A^{\Psi}$ are Lipschitz similar while $A^{\Phi_{\sf C}}$, $A^{\Psi_{\sf C}}$ are similar. As before, if
$\Phi$ is central then so is $\Psi$, in which case $A^{\Phi}$,
$A^{\Psi}$ are similar, and consequently $\Phi \stackrel{{\sf
    lin}}{\cong} \Psi$, so clearly (ii) holds. Thus it remains to
establish (ii) assuming that $\Phi$, $\Psi$ are {\em not\/} central. In this
case, (\ref{eq5_p1}) remains valid, and due to the similarity of
$A^{\Phi_{\sf AD}}$, $A^{\Psi_{\sf AD}}$ it
can furthermore be assumed that $k=\ell$ as well as $m_j = n_j$ and
$z_j = w_j$ for all $j\in \{1, \ldots , k\}$. As seen in the proof of
Lemma \ref{lem4_1}, $\Phi \stackrel{1}{\cong}\widetilde{\Phi}$ with
$$
A^{\widetilde{\Phi}} = \mbox{\rm diag} \, \bigl[
J_{m_1} (z_1), \ldots , J_{m_k} (z_k), {\sf Re}\,  z_{k+1} I_{{\sf d}
  ({\sf Im}\,  z_{k+1})} , \ldots ,   {\sf Re}\,  z_{k+k_0} I_{{\sf d}
  ({\sf Im}\,  z_{k+k_0})} , A^{\Phi_{\sf C}}
\bigr] \, ,
$$
and similarly $\Psi \stackrel{1}{\cong}\widetilde{\Psi}$ with
$$
A^{\widetilde{\Psi}} = \mbox{\rm diag} \, \bigl[
J_{m_1} (z_1), \ldots , J_{m_k} (z_k), {\sf Re}\,  w_{k+1} I_{{\sf d}
  ({\sf Im}\,  w_{k+1})} , \ldots ,   {\sf Re}\,  w_{k+{\ell}_0} I_{{\sf d}
  ({\sf Im}\,  w_{k+{\ell}_0})} , A^{\Phi_{\sf C}}
\bigr] \, ,
$$
Necessarily, therefore,
$$
\sum\nolimits_{j=1}^{k_0} {\sf d} ({\sf Im}\,  z_{k+j}) =
\sum\nolimits_{j=1}^{\ell_0} {\sf d} ({\sf Im}\,  w_{k+j}) =: d_0 \, .
$$
On the one hand, if $d_0=0$ then $A^{\widetilde{\Phi}} =
A^{\widetilde{\Psi}} $, so clearly $\widetilde{\Phi} \stackrel{{\sf
    lin}}{\cong} \widetilde{\Psi}$, and hence also $\Phi
\stackrel{1}{\cong} \Psi$. On the other hand, if $d_0\ge 1$ then
$\Lambda^{\widetilde{\Phi}} = \Lambda^{\Phi} = \Lambda^{\Psi} =
\Lambda^{\widetilde{\Psi}}$. As such, the two $d_0\times
d_0$-matrices
$$
\mbox{\rm diag} \, \bigl[
 {\sf Re}\,  z_{k+1} I_{{\sf d}
  ({\sf Im}\,  z_{k+1})} , \ldots ,   {\sf Re}\,  z_{k+k_0} I_{{\sf d}
  ({\sf Im}\,  z_{k+k_0})} \bigr] \, , \quad
\mbox{\rm diag} \, \bigl[
 {\sf Re}\,  w_{k+1} I_{{\sf d}
  ({\sf Im}\,  w_{k+1})} , \ldots ,   {\sf Re}\,  w_{k+\ell_0} I_{{\sf d}
  ({\sf Im}\,  w_{k+\ell_0})} \bigr]
$$
are Lyapunov similar, and hence in fact the same up to a permutation of the diagonal entries. Thus
$A^{\widetilde{\Phi}}$, $A^{\widetilde{\Psi}} $ are similar, so again $\widetilde{\Phi} \stackrel{{\sf
    lin}}{\cong} \widetilde{\Psi}$ and $\Phi
\stackrel{1}{\cong} \Psi$. In summary, the validity of
(iii) with $\beta = 1$ yields $\Phi
\stackrel{1}{\cong} \Psi$. As seen earlier, this completes the proof
of (iii)$\Rightarrow$(ii).

\medskip

\noindent
\underline{Step III:} Characterizing Lipschitz conjugacy.

\smallskip

\noindent
The ``if'' part of the ``Moreover \dots'' statement has just been
established in the previous step. To prove the
``only if'' part, assume that $\Phi \stackrel{1}{\cong}\Psi$. By
Proposition \ref{prop1za}, $A^{\Phi_{\sf C}}$, $A^{\Psi_{\sf C}}$ are
similar. On the one hand, if $\Phi$ is central then so is $\Psi$, in
which case $A^{\Phi}=A^{\Phi_{\sf C}}$, $A^{\Psi} = A^{\Psi_{\sf C}}$
are similar hence Lipschitz similar. On the other hand, if $\Phi$,
$\Psi$ are not central, note that $\Phi_{\bullet} \stackrel{1}{\cong}
\Psi_{\bullet}$ for $\bullet \in \{{\sf S}, {\sf U}\}$ via
$h|_{X_{\bullet}^{\Phi}}$, provided that $X_{\bullet}^{\Phi},
X_{\bullet}^{\Psi} \ne \{0\}$. By what has been proved previously,
there exists $\alpha_{\bullet}\in \R \setminus \{0\}$ so that
$A^{\Phi_{\bullet}}$, $\alpha_{\bullet} A^{\Psi_{\bullet}}$ are
Lipschitz similar, and Theorem \ref{lem3aol1} yields
$\alpha_{\bullet}=1$. Thus $A^{\Phi_{\bullet}}$, $A^{\Psi_{\bullet}}$
are Lipschitz similar whenever $X_{\bullet}^{\Phi},
X_{\bullet}^{\Psi} \ne \{0\}$. Since $A^{\Phi_{\sf C}}$, $A^{\Psi_{\sf
    C}}$ are Lipschitz similar as well, so are $A^{\Phi}$, $A^{\Psi}$.
\end{proof}

The remainder of this section briefly discusses the relation of
Theorem \ref{thm1x} to its differentiable and H\"{o}lder counterparts,
as well as its straightforward extension to complex spaces.

For every $A\in \R^{d\times d}$ and each $\bigstar\in \{0,1^{-}, 1, {\sf diff}, {\sf
  lin}\}$ let
$$
[A]_{\bigstar} = \left\{ B\in \R^{d\times d} : \Phi
\stackrel{\bigstar}{\thicksim} \Psi \enspace \mbox{\rm with $\Phi$,
  $\Psi$ generated by $A$, $B$ respectively} \right\} \, .
$$
Thus $[A]_{\bigstar}$ simply is the (equivalence) class of $A$
induced by $\stackrel{\bigstar}{\thicksim}$, and
$\bigl\{ [A]_{\bigstar}: A \in \R^{d\times d} \bigr\}$ is the
associated partition of $\R^{d\times d}$. Taking $d=2$ for instance,
there are precisely {\em six\/} different topological classes, namely
$$
[O_2]_0, \: [J_2]_0, \: [J_1(i)]_0,\:  \bigl[ \, \mbox{\rm diag}\, [-1,1]
\bigr]_0,  \: \bigl[ \, \mbox{\rm diag}\, [0,1]
\bigr]_0 \quad \mbox{\rm and} \quad  [I_2]_0 \, ;
$$
see also \cite{BW, BW2, Kuiper, Ladis, MM}. By
contrast, the different Lipschitz classes for $d=2$ are
$$
[O_2]_1, \: [J_2]_1,\: [J_1(i)]_1,\:  [J_2(1)]_1 \quad \mbox{\rm and}
\quad \bigl[ \, \mbox{\rm diag}\, [a,1]
\bigr]_1 \: \mbox{\rm with} \:  |a|\le 1 \, .
$$
In general, note that $[A]_{\sf lin} = [A]_{\sf
  diff}$ for every $A\in \R^{d\times d}$ by Proposition
\ref{prop1zb}, and clearly
\begin{equation}\label{eq5r1}
[A]_{\sf diff} \subset [A]_1\subset [A]_{1^-} \subset [A]_0 \qquad
\forall A \in \R^{d\times d} \, .
\end{equation}
While trivially all four classes coincide for $d=1$, every inclusion
in (\ref{eq5r1}) may be strict for $d\ge 2$,
as the example of $A=I_d$ illustrates, for which
\begin{align*}
& [I_d]_{\sf diff}  = \bigl\{ \alpha I_d : \alpha \in \R \setminus \{0\}\bigr\} \,
        , \\
&  [I_d]_{1}  = \bigcup\nolimits_{\alpha \in \R \setminus \{0\}}
              \bigl\{
B \in \R^{d\times d} : B \: \mbox{\rm diagonalizable (over $\C$)}, \sigma(B) \subset \alpha + i \R
              \bigr\} \, , \\
 &   [I_d]_{1^-}  = \bigcup\nolimits_{\alpha \in \R \setminus \{0\}}
              \bigl\{
B \in \R^{d\times d} : \sigma(B) \subset \alpha + i \R
                  \bigr\} \, , \\
  & [I_d]_0  = \bigl\{ B \in \R^{d\times d} : \sigma(B) \subset
            \C^-\bigr\} \cup \bigl\{ B \in \R^{d\times d} : \sigma(B) \subset
            \C^+\bigr\} \, .
\end{align*}

It has been suggested in the literature that the Lipschitz equivalence (or
conjugacy) of two linear flows $\Phi$, $\Psi$ is, in some sense,
``close'' to their differentiable equivalence (or conjugacy). This
sentiment is explicitly expressed in \cite{KS} where it motivates
the usage of Rademacher's theorem, but may also be detected, e.g.,
in \cite{ACK1, Willems}. Via an analysis of the left and middle
inclusions in (\ref{eq5r1}), the sentiment can be put in perspective
and made more precise. To this end, denote by $\cG_d$ the set of all
$A\in \R^{d\times d}$ with the property that
$$
\sigma(A)\cap i \R = \varnothing \, , \quad \# \sigma (A) = d \, , \quad
\mbox{\rm and} \quad
\# \bigl( \sigma(A) \cap (\alpha + i\R) \bigr) \le 2 \quad \forall
\alpha \in \R \setminus \{0\} \, .
$$
Note that every $A\in \cG_d$ is diagonalizable (over $\C$). Also, $\cG_d$ is open
and dense in $\R^{d\times d}$, whereas $\R^{d\times d}\setminus \cG_d$ is a
nullset. Informally, therefore, $A\in \cG_d$ for every ``typical''
real $d\times d$-matrix $A$, both in a topological and a measure-theoretical
sense. Equality in (\ref{eq5r1}) of the differentiable, Lipschitz, and
H\"{o}lder classes is easily characterized whenever $A\in
\cG_d$.

\begin{prop}\label{prop5PO}
  For every $A\in \cG_d$ the following statements are equivalent:
  \begin{enumerate}
  \item $[A]_{\sf diff} = [A]_1$;
  \item $[A]_1 = [A]_{1^-}$;
  \item $\sigma(A)\subset \R$.
\end{enumerate}
\end{prop}

Letting $\cG_d^{\dagger} = \{A\in \cG_d : \sigma(A)\subset \R\}$ for
convenience, note that $\cG_d^{\dagger}$ is non-empty and open. Clearly
$\cG_1^{\dagger} = \cG_1$, whereas $\cG_d \setminus \cG_d^{\dagger}$
is non-empty and open as well whenever $d\ge 2$. By Proposition \ref{prop5PO},
$$
[A]_{\sf diff} = [A]_1 = [A]_{1^-} \quad \forall A \in
\cG_d^{\dagger} \qquad \mbox{\rm and} \qquad 
[A]_{\sf diff} \ne  [A]_1 \ne  [A]_{1^-} \quad \forall A \in \cG_d \setminus
\cG_d^{\dagger} \, .
$$
Informally, therefore, the Lipschitz class of a ``typical'' generator,
both in a topological and a measure-theoretical sense, coincides with
its differentiable counterpart precisely as often as with its
H\"{o}lder counterpart. Moreover, notice that
$$
[A]_1 \cap \cG_d = [A]_{1^-} \cap \cG_d \qquad \forall A \in \cG_d \, .
$$
Thus, the classes $[A]_1$, $[A]_{1^-}$ are {\em
  essentially the same\/} even when not truly identical. Compare this
to
$$
[A]_{\sf diff} \cap \cG_d \ne [A]_1\cap \cG_d \qquad \forall A \in
\cG_d \setminus \cG_d^{\dagger} \, ,
$$
where in fact the set on the left is a nowhere dense nullset within the
set on the right. In other words, the classes $[A]_{\sf diff}$, $[A]_1$
are {\em drastically\/} different whenever $A\in \cG_d\setminus
\cG_d^{\dagger}$. This state of affairs prevails even for the example
$A=I_d$ above, despite the fact $I_d \not \in \cG_d$. All this suggests that while
Lipschitz equivalence is, unsurprisingly, situated ``between''
differentiable and H\"{o}lder equivalence, it is, if anything,
``closer'' to the latter than to the former. Figure \ref{fig52}
illustrates this discrepancy for $d=2$. 

\begin{figure}[ht] 
  \psfrag{tdet}[l]{$\det A$}
  \psfrag{ttr}[]{$\mbox{\rm trace} \, A$}
  \psfrag{tc}[]{ $\frac14 (\mbox{\rm trace} \, A)^2$}
\psfrag{tt1}[]{$[A]_1= [A]_{1^-}$}
\psfrag{tt2}[]{$[A]_{\sf diff} = [A]_1$}
\psfrag{tt2a}[]{$[A]_{1^-} \cap \cG_2$}
\psfrag{tt2b}[]{$=$}
\psfrag{tt3}[]{$[A]_{\sf diff} \ne [A]_1$}
\psfrag{tt3a}[]{$[A]_{\sf diff} \cap \cG_2$}
\psfrag{tt3b}[]{$\ne $}
\psfrag{tt3c}[]{$[A]_1\cap \cG_2$}
\psfrag{tg2p}[]{$\cG_2 \setminus \cG_2^{\dagger}$}
\psfrag{tg2pa}[]{$ \cG_2^{\dagger}$}
\psfrag{to2}[]{$O_2, J_2$}
\psfrag{tj1i}[]{$J_1(i)$}
\psfrag{tdm11}[l]{$\mbox{\rm diag}\, [-1,1]$}
\psfrag{td01}[l]{$\mbox{\rm diag}\, [0,1]$}
\psfrag{ti2}[]{$I_2$}
%
%
\begin{center}
\includegraphics{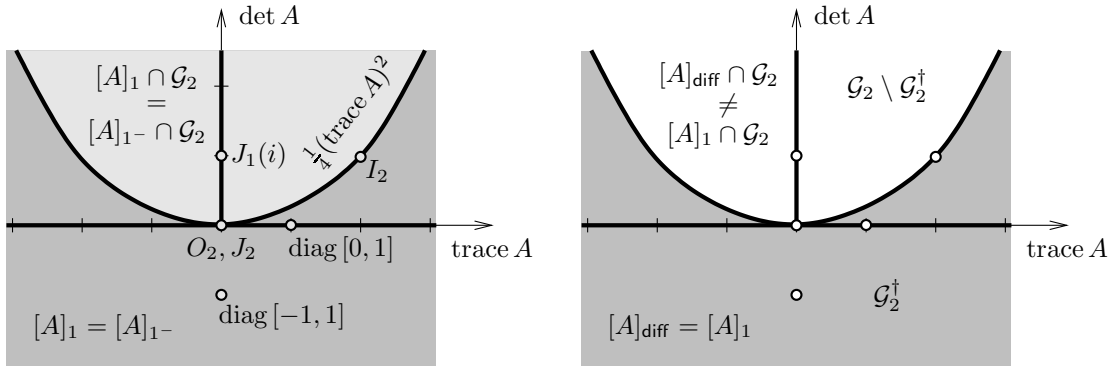}
\end{center}
\vspace*{-4mm}
\caption{The differentiable, Lipschitz,
  and H\"{o}lder classes of $A\in \R^{2\times 2}$ coincide for $A\in
  \cG_2^{\dagger}$ yet differ for $A\in \cG_2 \setminus
  \cG_2^{\dagger}$. Whenever $A\in \cG_2 \setminus
  \cG_2^{\dagger}$, the class $[A]_1$ resembles $[A]_{1^-}$ much more
  (left) than it resembles $[A]_{\sf diff}$. Notice that all matrices in $\R^{2\times
    2}\setminus \cG_2$ correspond to points on the solid
  black curves. Circles indicate the {\em six\/} different
  classes $[A]_0$, with $O_2$, $J_2$ corresponding to the same
  point.}\label{fig52}
\end{figure}

\subsection*{Lipschitz equivalence on complex spaces}

The analysis of $\Phi \stackrel{1}{\thicksim} \Psi$ thus far has
focussed entirely on {\em real\/} flows. As outlined here, it can
easily be extended to linear flows on arbitrary
finite-dimensional normed spaces; see also \cite{BW, BW2, Ladis}. To
this end, let $(X, \|\cdot\|)$ be a finite-dimensional normed space
with $X \ne \{0\}$
over $\K = \R$
or $\K = \C$. Denote by $X^{\R}$ the {\bf realification} of $X$, i.e.,
$X^{\R}$ equals $X$ as a set but is a linear space with the field of
scalars restricted to $\R$, and define $\iota_X : X \to X^{\R}$ as $\iota_X(x)
= x$. Thus, if $\K = \C$ then $\iota_X$ is an $\R$-linear bijection
and $\dim X^{\R} = 2\dim X$; moreover, $\|\cdot\|_{X^{\R}} :=
\|\cdot\|\circ \iota_X^{-1}$ is a norm on $X^{\R}$, and $\iota_X$ is
an isometry. (Trivially, if $\K = \R$ then $X^{\R} = X$ as linear spaces and $\iota_X = I_X$.)
Every map $h:X\to X$ induces a map $h^{\R} = \iota_X \circ h \circ
\iota_X^{-1}: X^{\R} \to X^{\R}$, and clearly
\begin{equation}\label{eq71}
h \in \cH_{1} (X) \quad \Longleftrightarrow \quad h^{\R} \in
\cH_{1} (X^{\R}) \qquad \forall h\in \cH_0 (X) \, .
\end{equation}
By contrast, note that with $J_X:= (iI_X)^{\R}$, 
\begin{equation}\label{eq72}
h \in \cH_{{\sf diff}} (X) \quad \Longleftrightarrow \quad h^{\R} \in
\cH_{{\sf diff}} (X^{\R}) \enspace \mbox{\rm and} \enspace J_X D_0
h^{\R} = ( D_0 h^{\R})  J_X \qquad \forall h\in \cH_0 (X) \, .
\end{equation}
Given any (not necessarily linear) flow $\varphi$ on $X$, its
realification $\varphi^{\R}$ is the flow on $X^{\R}$ with
$(\varphi^{\R})_t = (\varphi_t)^{\R}$ for all $t\in \R$. By (\ref{eq71}),
given two flows $\varphi$, $\psi$ on $X$, observe that $\varphi
\stackrel{1}{\thicksim} \psi $ precisely if $\varphi^{\R}
\stackrel{1}{\thicksim} \psi^{\R} $, 
and similarly for $\varphi \stackrel{1}{\cong} \psi$. For a $\K$-linear flow
$\Phi$ on $X=\K^d$, it is readily seen that all the dynamical objects
associated with $\Phi$ that have been studied in earlier sections behave
naturally under realification: For instance, $A^{\Phi^{\R}} =
(A^{\Phi})^{\R}$, and $X_{\bullet}^{\Phi^{\R}} =
(X_{\bullet}^{\Phi})^{\R} = \iota_X (X_{\bullet}^{\Phi})$ for every $\bullet
\in \{ {\sf D}, {\sf AD}, {\sf S}, {\sf C}, {\sf U}, {\sf H}\}$, as well as
$(\Phi_{\bullet})^{\R} = (\Phi^{\R})_{\bullet}$. Also, if $\K = \C$
then $\lambda_{2j-1}^{\Phi^{\R}} = \lambda_{2j}^{\Phi^{\R}} =
\lambda_j^{\Phi}$ for every $j\in \{1, \ldots, d\}$. With this, the
Lipschitz classification of $\K$-linear flows follows 
immediately from Theorem \ref{thm1x}; it turns out to be a {\em
  real\/} result, in the sense that whether or not $\Phi 
\stackrel{1}{\thicksim} \Psi$ is determined completely by the associated
realifications $\Phi^{\R}$, $\Psi^{\R}$. 

\begin{theorem}\label{thm71}
Let $\Phi$, $\Psi$ be $\K$-linear flows on $X$. Then the following
statements are equivalent:
\begin{enumerate}
\item $\Phi \stackrel{1}{\thicksim}\Psi$;
  \item there exists $\alpha\in \R\setminus \{0\}$ so that $\Phi \stackrel{1}{\cong}\Psi_{*\alpha}$;
  \item $\Phi^{\R} \stackrel{1}{\thicksim} \Psi^{\R}$;
      \item there exists $\beta\in \R\setminus \{0\}$ so that $\Phi^{\R} \stackrel{1}{\cong}\Psi^{\R}_{*\beta}$;
\item there exists $\gamma \in \R \setminus \{0\}$ so that
$A^{\Phi^{\R}}$, $\gamma A^{\Psi^{\R}}$ are Lipschitz similar while
$A^{\Phi_{\sf C}^{\R}}$, $\gamma A^{\Psi_{\sf C}^{\R}}$ are similar.
\end{enumerate}
Moreover, $\Phi \stackrel{1}{\cong}\Psi$ if and only if $\Phi^{\R} \stackrel{1}{\cong}\Psi^{\R}$
 if and only if $A^{\Phi^{\R}}$, $A^{\Psi^{\R}}$ are Lipschitz similar while
$A^{\Phi_{\sf C}^{\R}}$, $A^{\Psi_{\sf C}^{\R}}$ are similar.
\end{theorem}

\begin{proof}
Obviously (ii)$\Rightarrow$(i) and (iv)$\Rightarrow$(iii) by
definition, but also (i)$\Leftrightarrow$(iii) and
(ii)$\Leftrightarrow$(iv) by (\ref{eq71}). Furthermore, Theorem
\ref{thm1x} shows that
(iii)$\Leftrightarrow$(iv)$\Leftrightarrow$(v), and so all five
statements are equivalent. The claim regarding $\Phi
\stackrel{1}{\cong}\Psi$ follows similarly.
\end{proof}

\begin{rem}\label{rem73}
As detailed in \cite[Sec.\ 7]{BW2}, the H\"{o}lder counterpart of Theorem \ref{thm71},
i.e., the extension of Proposition \ref{prop1za} to any
finite-dimensional normed space,
also is a real theorem in the above sense. By contrast, the corresponding extension of Proposition
\ref{prop1zb} is {\em not\/}
a real theorem: Using (\ref{eq72}), it is readily seen that $\Phi^{\R} \stackrel{{\sf
    diff}}{\thicksim}\Psi^{\R}$ is necessary but not in general
sufficient for $\Phi\stackrel{{\sf
    diff}}{\thicksim}\Psi$; see \cite[Sec.\ 6]{BW}.
\end{rem}


\section{Pointwise Lipschitz equivalence at $0$}\label{sec6}

This final section introduces and briefly discusses the novel concept
of pointwise Lipschitz equivalence. Though arguably also of
independent interest, this concept is introduced here solely to highlight
several subtle aspects of the Lipschitz classification of linear flows
that may otherwise go unnoticed.

Given the pivotal role of the point $0$ in the analysis in
previous sections, it is natural to ask how the main results of
this article would be affected if, for instance, the homeomorphism
$h$ in (\ref{eq1_1}) or (\ref{eq1_2}), and its inverse, were assumed
to satisfy only a pointwise Lipschitz condition at $0$, instead of the
familiar (uniform) Lipschitz condition in an entire
neighbourhood. Formally, in analogy to (\ref{eq2_0}) let
$$
\cH_{{\sf pw}1} = \left\{ h \in \cH : \exists r \in \R^+ \: \mbox{\rm
    s.t.} \: \sup\nolimits_{x\in B_r(0)\setminus \{0\}}  \frac{ |h(x)|
    +  |h^{-1}(x) | }{|x| } < \infty \right\}  \, ,
$$
and notice the strict inclusion $(\cH_1 \cup \cH_{\sf diff} ) \subset
\cH_{{\sf pw}1}$. Say that two flows $\varphi$, $\psi$ on $X$ are {\bf
  pointwise Lipschitz equivalent} (or {\bf conjugate}) at $0$, in symbols $\varphi
\stackrel{{\sf pw}1}{\thicksim} \psi$ (or $\varphi
\stackrel{{\sf pw}1}{\cong} \psi$), if $\varphi
\stackrel{h}{\thicksim}\psi$ (or $\varphi \stackrel{h}{\cong}\psi$)
for some $h\in \cH_{{\sf pw}1}$. As a refinement of (\ref{eq13}),
\begin{equation}\label{eq6_1}
\begin{tikzcd}
\varphi \stackrel{{\sf lin}}{\thicksim} \psi \arrow[d, Rightarrow]
\arrow[r, Rightarrow] & \varphi \stackrel{1}{\thicksim} \psi \arrow[d,
Rightarrow]
\arrow[r, Rightarrow] & \varphi \stackrel{1^- }{\thicksim} \psi
\arrow[d, Rightarrow] \\
\varphi \stackrel{{\sf diff}}{\thicksim} \psi \arrow[r, Rightarrow] &
\varphi \stackrel{{\sf pw}1}{\thicksim} \psi \arrow[r, Rightarrow] &
\varphi \stackrel{0}{\thicksim} \psi 
\end{tikzcd}
\end{equation}
and similarly with $\thicksim$ replaced by $\cong$. No implication in
(\ref{eq6_1}) can be reversed in general. Moreover, with $\varphi
\stackrel{{\sf pw}1}{\thickapprox} \psi$ defined accordingly, Proposition \ref{lemH1} holds
for $\bigstar = {\sf pw}1$ as well.

\begin{rem}\label{rem6_2}
Though inspired by similar terminology in the literature \cite{DJ,
  gut, Hei0, mess},
the definition of {\em pointwise Lipschitz\/} in this article is motivated
solely by its subsequent use in the classification of (linear) flows. Letting
$$
{\rm Lip}_x h = \limsup\nolimits_{y\to x, y\ne x} \frac{|h(x) - h(y)|
  + |h^{-1}(x) - h^{-1}(y)|}{|x-y|} \in [0,\infty] \qquad \forall x
\in X  , h\in \cH_0 \, , 
$$
note that $h\in \cH_{{\sf pw}1}$ simply means that ${\rm Lip}_0
h<\infty$. Many other definitions in a similar vein are conceivable. For
instance, one might instead require that ${\rm
  Lip}_x h<\infty$ for some $r\in \R^+$ and all $x\in B_r(0)$. This
would result in a slightly more restrictive definition but, as will
become clear shortly, would not affect at all any of the 
conclusions regarding the pointwise Lipschitz equivalence of linear
flows. Alternatively, one might require that
$\sup_{x\in B_r(0)}{\rm Lip}_x h<\infty$. By \cite[Cor.\ 2.4]{DJ} this
even more restrictive definition would simply lead back to the familiar
(uniform) Lipschitz equivalence. 
\end{rem}

Informally put, the first main result of this section says that pointwise Lip\-schitz {\em equivalence\/}
is rather uninteresting for most linear flows, as it simply coincides
with topological equivalence. This means, however, that
$\stackrel{{\sf pw}1}{\thicksim}$ differs {\em drastically\/} from
$\stackrel{1}{\thicksim}$, even though the decrease in
regularity from $h\in \cH_{1}$ to $h\in \cH_{{\sf pw}1}$ may appear to
be minuscule.

\begin{theorem}\label{thm6_4}
Let $\Phi$, $\Psi$ be linear flows on $X$. If $\Phi$, $\Psi$ are
hyperbolic then the following statements are equivalent:
\begin{enumerate}
\item $\Phi \stackrel{{\sf pw}1}{\thicksim} \Psi$;
\item $\Phi \stackrel{0}{\thicksim} \Psi$;
\item $\{d_{\sf S}^{\Phi}, d_{\sf U}^{\Phi}\} = \{d_{\sf S}^{\Psi}, d_{\sf U}^{\Psi}\}$.
\end{enumerate}
\end{theorem}

Denoting by $\langle
\cdot , \cdot \rangle$ the standard inner product on $X=\R^d$, recall
that a matrix $G\in \R^{d\times d}$ is {\em positive definite\/} if
$G=G^{\top}$ and $\langle Gx,x\rangle >0$ for all $x\in
X\setminus \{0\}$. As presented below, the proof of Theorem \ref{thm6_4} makes use of
the following elementary linear algebra fact; see, e.g., \cite[Sec.\ 7.2]{HJ}. 

\begin{prop}\label{lem6_3}
Given $A\in \R^{d\times d}$ with $\sigma (A)\subset
\C^+$, there exists a positive definite $G\in \R^{d\times d}$ so that 
$G A + A^{\top }G$ and $GA^2 + 2 A^{\top} G A + (A^{\top})^2 G$ both are
positive definite. 
\end{prop}

\begin{proof}[Proof of Theorem \ref{thm6_4}]
Plainly (i)$\Rightarrow$(ii), and it is well known that
(ii)$\Leftrightarrow$(iii); see, e.g., \cite[Thm.\ 1.1]{BW2}. Thus, it
only has to be shown that (ii)$\Rightarrow$(i). To this end, assume
that $\Phi \stackrel{0}{\thicksim} \Psi$. By Proposition \ref{lemH1},
no generality is lost by assuming that $(d_{\sf S}^{\Phi}, d_{\sf
  U}^{\Phi})= ( d_{\sf S}^{\Psi}, d_{\sf U}^{\Psi})$. With this, all
that needs to be shown is that
\begin{equation}\label{eq6_p1}
\Phi \stackrel{{\sf pw}1}{\thicksim} \widetilde{\Phi} \quad \mbox{\rm
  where} \quad A^{\widetilde{\Phi}} = \mbox{\rm diag}\, [-I_{d_{\sf
    S}^{\Phi}}, I_{d_{\sf U}^{\Phi}}] \, ,
\end{equation}
and this will now be done in two separate steps for the reader's
convenience.

\medskip

\noindent
\underline{Step I:} Proof of (\ref{eq6_p1}) for (un)stable $\Phi$.

\smallskip

\noindent
Assume that $\Phi$ is stable. As $\sigma (-A^{\Phi})\subset \C^+$, use
Proposition \ref{lem6_3} to pick a positive definite $G$ so that
$$
B^{\Phi} := - (G A^{\Phi} + (A^{\Phi})^{\top} G) \, , \quad
C^{\Phi} := G (A^{\Phi})^2 + 2  (A^{\Phi})^{\top} G A^{\Phi} +
\bigl( (A^{\Phi})^{\top} \bigr)^2 G
$$
both are positive definite. Letting $\|\cdot\|= |G^{1/2} \cdot|$ for
convenience, note that $t\mapsto \|\Phi_t x\|^2$ is strictly convex
for every $x\in X \setminus \{0\}$ because
$$
\frac{{\rm d}^2}{{\rm d}t^2} \|\Phi_t x\|^2 = \langle C^{\Phi} \Phi_t
x , \Phi_t x \rangle >0 \qquad \forall t \in \R \, ,
$$
and hence $t\mapsto \|\Phi_t x\|$ is strictly decreasing. For every
$x\in X\setminus \{0\}$ denote by $T(x)$ the unique $t\in \R$ with
$\|\Phi_t x\|=1$; clearly, $x\mapsto T(x)$ is smooth on $X\setminus
\{0\}$. With this, let $h(0)=0$ and
$$
h(x) = \|x\| \Phi_{T(x)} x \qquad \forall x \in X \setminus \{0\} \, .
$$
Note that the map $h$ is continuous on $X$, smooth on $X\setminus \{0\}$, and
$\|h(x)\|=\|x\|$ for all $x\in X$. Moreover, it is readily seen that $h$ is
one-to-one and onto, with
$$
h^{-1} (x) = \frac1{\|x\|} \Phi_{T(x/\|x\|^2)} x \qquad \forall x \in
X \setminus \{0\} \, ,
$$
and hence $h\in \cH_{{\sf pw}1}(X)$. Furthermore, $T(\Phi_t x) = T(x)
- t$ for all $t\in \R$, $x\in X\setminus \{0\}$, and so
$$
h(\Phi_t x) = \|\Phi_t x\| \Phi_{T(\Phi_t x)} \Phi_t x =
\frac{\|\Phi_t x\|}{\|x\|} h(x) = \widetilde{\Phi}_{\tau_x(t)} h(x)
\qquad \forall t \in \R , x\in X\setminus \{0\} \, ,
$$
where, for every $x\in X\setminus \{0\}$,
$$
\tau_x(t) = \log \frac{\|x\|}{\|\Phi_t x\|} \qquad \forall t \in \R \, .
$$
Thus $\Phi \stackrel{{\sf pw}1}{\thicksim} \widetilde{\Phi}$.
The argument for unstable $\Phi$ is identical, except that $B^{\Phi}$
and $\tau_x$ both have to be multiplied by $-1$, as $t\mapsto \|\Phi_t
x\|$ now is strictly {\em increasing}. In summary,
(\ref{eq6_p1}) holds for every (un)stable flow $\Phi$.

\medskip

\noindent
\underline{Step II:} Proof of (\ref{eq6_p1}) for arbitrary
(hyperbolic) $\Phi$.

\smallskip

\noindent
Assume that $d_{\sf S}^{\Phi}, d_{\sf U}^{\Phi}\ge 1$; no generality
is lost by assuming that $X_{\sf S}^{\Phi} = \mbox{\rm span}\, \{ e_1,
\ldots , e_{d_{\sf S}^{\Phi}}\}=: E_{\sf S}$ and $X_{\sf U}^{\Phi} =
\mbox{\rm span}\, \{ e_{d_{\sf S}^{\Phi}+ 1}, \ldots e_d \}=: E_{\sf
  U}$. Note that $X_{\sf S}^{\widetilde{\Phi}} = E_{\sf S}$, $X_{\sf
  U}^{\widetilde{\Phi}} = E_{\sf U}$ as well. For ease of notation,
let $y= P_{\sf S} x \in E_{\sf S}$, $z= P_{\sf U} x \in E_{\sf U}$,
and write $x$ symbolically as
$\left[\! \begin{array}{c} y\\     z \end{array}\! \right]$. As in
Step I, for $\bullet \in \{{\sf S}, {\sf U}\}$ pick a positive
definite $G_{\bullet} \in \R^{d_{\bullet}^{\Phi}\times
  d_{\bullet}^{\Phi}}$ so that $B_{\bullet}:= B^{\Phi_{\bullet}}$,
$C_{\bullet}:= C^{\Phi_{\bullet}}$ both are positive definite, and let
$$
\|x\| = \sqrt{|G_{\sf S}^{1/2} y|^2 + |G_{\sf U}^{1/2}z|^2} \qquad \forall x \in X
\, .
$$
For every $x\in E_{\bullet}\setminus \{0\}$ there exists a unique
$T_{\bullet}(x)\in \R$ with $\|\Phi_{T_{\bullet}(x)}x\|=1$. As seen in
Step I, defining $h_{\bullet}(x)= \|x\|\Phi_{T_{\bullet}(x)}x$ for all $x\in
E_{\bullet}\setminus \{0\}$ yields a homeomorphism $h_{\bullet}\in
\cH_{{\sf pw}1}(E_{\bullet})$.

Now, fix any $x\in X \setminus (E_{\sf S} \cup E_{\sf U})$ and deduce
from
$$
\frac{{\rm d}^2}{{\rm d}t^2} \|\Phi_t x\|^2 = \langle C_{\sf S }\Phi_t
y , \Phi_t y \rangle +  \langle C_{\sf U }\Phi_t
z , \Phi_t z \rangle >0 \qquad \forall t \in \R \, ,
$$
that $t\mapsto \|\Phi_t x\|^2$ is strictly convex. Since $\|\Phi_t
x\|\to \infty$ as $|t|\to \infty$, there exists a unique $T(x)\in \R$
for which $\|\Phi_t x\|$ is minimal. Note that $x\mapsto T(x)$ is
smooth on $X\setminus (E_{\sf S} \cup E_{\sf U})$. Also, 
$0<\|\Phi_{T(x)}x\|\le \|x\|$, and from
$$
\frac{{\rm d}}{{\rm d}t} \|\Phi_t x\|^2 = -\langle B_{\sf S }\Phi_t
y , \Phi_t y \rangle +  \langle B_{\sf U }\Phi_t
z , \Phi_t z \rangle  \qquad \forall t \in \R \, ,
$$
it is clear that $t$ equals $T(x)$ if and only if $\Phi_t x \in
\cC^{\Phi}$, with the cone $\cC^{\Phi}\subset X$ given by
$$
\cC^{\Phi} = \left\{
\left[\! \begin{array}{c} y\\     z \end{array}\! \right] \in X : \langle
B_{\sf S} y,y\rangle =  \langle
B_{\sf U} z,z\rangle 
  \right\} \, .
$$
Similarly, $t\mapsto \big\|\widetilde{\Phi}_t x\big\|^2 = e^{-2t}
\|y\|^2 + e^{2t} \|z\|^2 $ is strictly convex, and so $\big\|\widetilde{\Phi}_t x\big\|$ is minimal precisely
if $t = \log \sqrt{\|y\| / \|z\| }$ or equivalently
$$
\widetilde{\Phi}_t x \in \cC^{\widetilde{\Phi}} := \left\{
\left[\! \begin{array}{c} y\\     z \end{array}\! \right] \in X : \|y\| =\|z\|
  \right\} \, .
 $$
 Now, observe that there exists a unique $\widetilde{x}\in \cC^{\widetilde{\Phi}}$ with
 $\|\widetilde{x}\| = \|\Phi_{T(x)}x\|$ so that
 $ \widetilde{x} = 
\left[\! \begin{array}{c} \Phi_r y\\    \Phi_s z \end{array}\! \right]
$ for some $r,s\in \R$; 
in fact, $\widetilde{x} = \displaystyle \frac{\|\Phi_{T(x)}x\|}{\sqrt{2}}
\left[\! \begin{array}{c} \Phi_{T_{\sf S}(y)} y\\    \Phi_{T_{\sf U}(z)}
        z \end{array}\! \right] $. Furthermore,
    $\big\|\widetilde{\Phi}_t \widetilde{x}\big\| =\|x\|$ if and only
if 
\begin{equation}\label{eq6_p10}
e^{2t} = \frac{\|x\|^2 \pm \sqrt{\|x\|^4 -
    \|\Phi_{T(x)}x\|^4}}{\|\Phi_{T(x)}x\|^2} \, .
\end{equation}
Note that (\ref{eq6_p10}) yields precisely two values $t = \pm
\widetilde{T}(x)$ with $\widetilde{T}(x)>0$ whenever $T(x)\ne
0$. Defining $h(x)$ as
$\widetilde{\Phi}_{-\widetilde{T}(x)}\widetilde{x}$, $\widetilde{x}$, or
$\widetilde{\Phi}_{\widetilde{T}(x)}\widetilde{x}$ depending on
whether $T(x)> 0$, $T(x)=0$, or
$T(x)<0$ (that is, $\mbox{\rm sign} \, T(x)$ equals $1$, $0$, or $-1$ respectively) yields
\begin{equation}\label{eq6_p11}
  h(x) = \frac1{\sqrt{2}}\left[\begin{array}{c}
\sqrt{\|x\|^2 + \mbox{\rm sign} \, T(x) \sqrt{\|x\|^4 -
                                 \|\Phi_{T(x)}x\|^4}} \, \Phi_{T_{\sf
                                 S}(y)}y \\[3mm]
\sqrt{\|x\|^2 - \mbox{\rm sign} \, T(x) \sqrt{\|x\|^4 -
                                 \|\Phi_{T(x)}x\|^4}} \, \Phi_{T_{\sf
                                 U}(z)}z                                  
                               \end{array}\right] \qquad \forall x \in
                             X \setminus (E_{\sf S} \cup E_{\sf U}) \, .
\end{equation}
Plainly $h$ is continuous on $X \setminus (E_{\sf S} \cup E_{\sf U})
$. Moreover, given $y\in E_{\sf S}\setminus \{0\}$, observe that
letting $z\to 0$ in $x=
\left[\! \begin{array}{c} y\\     z \end{array}\! \right]$ yields
$T(x)\to \infty$, $\|\Phi_{T(x)}x\| \to 0$, and $\|x\|\to \|y\|$.
This shows that
$$
h(x) \to \left[\! \begin{array}{c} \|y\| \Phi_{T_{\sf S}(y)}y\\
                 0 \end{array}\! \right] = \left[\! \begin{array}{c} h_{{\sf
                                                S}}(y)\\
                                                      0 \end{array}\! \right] \quad \mbox{\rm as} \quad z
\to 0 \, .
$$
Similarly, given $z\in E_{\sf U}\setminus \{0\}$,
$$
h(x) \to \left[\! \begin{array}{c} 0 \\ h_{{\sf
   U}}(z) \end{array}\! \right] \quad \mbox{\rm as} \quad y
\to 0 \, .
$$
Thus, with $h(x) := \left[\! \begin{array}{c} h_{\sf S}(y)
                           \\
                           h_{\sf U} (z) \end{array}\! \right]$
 for every $x\in E_{\sf S} \cup E_{\sf U}$, the map $h$ is
 continuous. Moreover, it is readily verified that $h$ is one-to-one
 and onto, as strongly suggested by its construction; see also Figure
 \ref{fig6}. Since $\|h(x)\|=\|x\|$ for all $x\in X$, clearly $h\in
 \cH_{{\sf pw}1}(X)$. Informally put, $h$ is a ``joining'' of its two
 ``marginals'' $h_{\bullet}\in \cH_{{\sf
     pw}1}(E_{\bullet})$. Crucially,
 $T(\Phi_tx)= T(x) -t$ for all $ t \in \R , x \in X
 \setminus (E_{\sf S} \cup E_{\sf U})$, and hence 
\begin{equation}\label{eq6_p12}
h(\Phi_t x) = \frac1{\sqrt{2}}\left[\begin{array}{c}
\sqrt{\|\Phi_t x\|^2 + \mbox{\rm sign} \, ( T(x)-t)  \sqrt{\|\Phi_t x\|^4 -
                                 \|\Phi_{T(x)}x\|^4}} \, \Phi_{T_{\sf
                                 S}(y)}y \\[3mm]
\sqrt{\|\Phi_t x\|^2 - \mbox{\rm sign} \, ( T(x) -t)  \sqrt{\|\Phi_t x\|^4 -
                                 \|\Phi_{T(x)}x\|^4}} \, \Phi_{T_{\sf
                                 U}(z)}z                                  
                                    \end{array}\right] =
\widetilde{\Phi}_{\tau_x(t)} h(x) \qquad \forall t \in \R   \, ,                               
\end{equation}
where $\tau_x$ is given by
$$
\tau_x(t) = \frac12 \log \frac{ \|x\|^2 + \mbox{\rm sign} \, T(x) \sqrt{\|x\|^4 -
                                 \|\Phi_{T(x)}x\|^4}}{ \|\Phi_t
                               x\|^2 + \mbox{\rm sign} \, (T(x)-t)
                               \sqrt{\|\Phi_t x\|^4 -
                                 \|\Phi_{T(x)}x\|^4}}\qquad \forall t
                           \in \R \, .
$$
With the appropriate $\tau_x$ from Step I, the two outer-most
expressions in (\ref{eq6_p12}) agree also for every $x\in ( E_{{\sf S}} \cup
E_{\sf U})\setminus \{0\}$. In summary, $\Phi\stackrel{h}{\thicksim}\widetilde{\Phi}$
with $h\in \cH_{{\sf pw}1}(X)$. This establishes (\ref{eq6_p1}) for
every hyperbolic flow $\Phi$, and hence completes the proof as previously
discussed.
\end{proof}

\begin{figure}[ht] 
  \psfrag{teu}[]{$E_{\sf U}$}
  \psfrag{tes}[]{$E_{\sf S}$}
  \psfrag{tcphi}[]{$\cC^{\Phi}$}
  \psfrag{tphitx}[r]{$\Phi_{T(x)}x$}
\psfrag{tphirx}[]{$\Phi_{\R}x$}
  \psfrag{tphitilxtil}[l]{$\widetilde{\Phi}_{\R}h(x)$}
  \psfrag{ttilx}[]{$\widetilde{x}$}
  \psfrag{tx}[]{$x$}
  \psfrag{tit}[r]{$\Phi \stackrel{h}{\thicksim} \widetilde{\Phi}$}
  \psfrag{thx}[]{$h(x)$}
  \psfrag{tn1}[]{$\|\cdot \| \! = \! \|\Phi_{T(x)}x\|\:$}
   \psfrag{tn2}[]{$\|h(x)\| \! = \! \| x\|$}
   \psfrag{tcphitil}[]{$\cC^{\widetilde{\Phi}}$}
   \psfrag{tconst}[]{$\|\cdot\|= \|x\|$}
%
%
%
\vspace*{2mm}
\begin{center}
\includegraphics{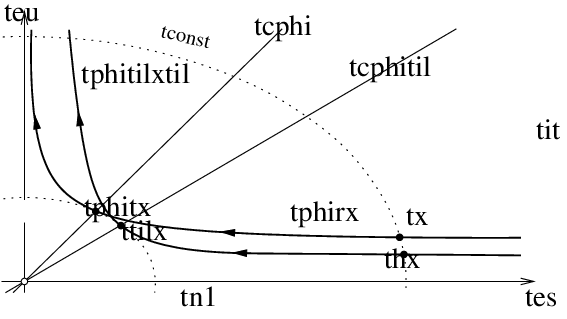}
\end{center}
\vspace*{-4mm}
\caption{Constructing $h\in \cH_{{\sf pw}1}(X)$ to prove
  (\ref{eq6_p1}) for hyperbolic $\Phi$.}\label{fig6}
\end{figure}

\begin{rem}\label{rem6_5}
The non-trivial and perhaps somewhat surprising implication
(ii)$\Rightarrow$(i) in Theorem \ref{thm6_4} may fail if $\Phi$,
$\Psi$ are not hyperbolic. To illustrate this by means of a simple
example, take $d=3$, and for every $a\in \R^+$ let $\Phi$, $\Psi$ be
generated by
$\mbox{\rm diag} \, [-a, J_1(2\pi i)]$, $\mbox{\rm diag} \, [-1,
J_1(2\pi i)]$ respectively. Assume that $\Phi \stackrel{h}{\thicksim} \Psi$ for some
$h\in \cH_{{\sf pw}1}(\R^3)$. Letting $x_n = e^{-n} e_1 +
e^{-(1+a)n}e_2$ for every $n\in \N$, observe that
$$
|x_n| = e^{-n} \sqrt{1 + e^{-2an}} \, , \quad |\Phi_n x_n| =
e^{-(1+a)n} \sqrt{2} \qquad \forall n\in \N \, .
$$
Clearly, $\lim_{j\to \infty} \Phi_j x_n = e^{-(1+a)n}e_2\in \mbox{\rm
  Per}_1 \Phi \setminus \{0\}$ for every $n\in \N$. The continuity of
$h$ implies that $\lim_{j\to \infty} \Psi_{\tau_{x_n}(j)} h(x_n) = h(e^{-(1+a)n}e_2)\in
\mbox{\rm Per}_1 \Psi \setminus \{0\}$. Since $h^{-1}$ is
continuous as well, it is readily seen that $\limsup_{t\to
  \infty}|\tau_{x_n}(t) - t|\le 2$, which in turn implies that $\sup_{t\ge
  0}|\tau_{x_n}(t) - t |\le 4$. Pick any $\kappa > \mbox{\rm Lip}_0h\ge 1$. With this, observe that for all sufficiently large
$n\in \N$,
$$
|h(\Phi_n x_n)| \le \kappa |\Phi_n x_n| = \kappa \sqrt{2} e^{-(1+a)n}
\, ,
$$
but also, since $t\mapsto |\Psi_t h(x_n)|$ is (strictly) decreasing,
$$
|\Psi_{\tau_{x_n}(n)}h(x_n)| \ge  |\Psi_{n+4} h(x_n)|
                  \ge e^{-(n+4)} |h(x_n)| \ge e^{-(n+4)}
                  \frac{|x_n|}{\kappa} = e^{-2n} \frac1{\kappa e^4}
                  \sqrt{1+e^{-2an}} \, .
$$
For all sufficiently large $n\in \N$, therefore,
$$
 e^{-2n} \frac1{\kappa e^4}
                  \sqrt{1+e^{-2an}} \le  |\Psi_{\tau_{x_n}(n)} h(x_n)| =
                  |h(\Phi_{n} x_n ) | \le \kappa \sqrt{2} e^{-(1+a)n} \, ,
$$
so necessarily $a\le 1$. Interchanging the roles of $\Phi$,
$\Psi$ yields $a=1$. In summary, $\Phi \stackrel{{\sf
    pw}1}{\thicksim}\Psi$ (if and) only if $a=1$ whereas clearly
$\Phi \stackrel{0}{\thicksim}\Psi$ for every $a\in \R^+$. 
\end{rem}

The second main result of this section, Theorem \ref{thm6_8} below,
provides a characterization of pointwise 
Lipschitz {\em conjugacy}. Though it arguably is less surprising than
Theorem \ref{thm6_4} and instead aligns more closely with its
uniform counterpart, the result is interesting nonetheless as it
relates pointwise Lipschitz conjugacy to a classical concept of stability theory. Specifically,
recall that two linear flows $\Phi$, $\Psi$ on $X$ are {\bf
  kinematically similar} if there exists an invertible linear operator
$Q$ so that
$$
\sup\nolimits_{t\in \R} (|\Phi_t Q^{-1} \Psi_{-t}| + |\Psi_t Q
\Phi_{-t}|) < \infty \, ;
$$
see, e.g., \cite[Sec.\ 5]{cop} or \cite[Sec.\ 4]{KS}. To express the kinematic similarity of linear flows in terms of their
generators, it is helpful to introduce one further, tailor-made notion of
similarity for matrices, reminiscent of Lipschitz similarity, as follows: In
analogy to (\ref{eq3_1a}), for every $m\in \N$, $a,b\in \R$ define
$$
\cK J_m(a+ib) = a I_{m\, {\sf d}(b)} + K_m(b) =  \left\{
  \begin{array}{ll}
    J_m(a) & \mbox{\rm if } b=0  \, ,\\
   \mbox{\rm diag}\, [J_m(a), J_m(a)] & \mbox{\rm if } b\ne 0  \, .
    \end{array}
\right.
$$
Given any $A\in \R^{d\times d}$, and with $P$, $J_{m_j}(z_j)$ as in
(\ref{eq3_1b}), let
$$
\cK A = P^{-1} \mbox{\rm diag}\, \bigl[
\cK J_{m_1} (z_1) , \ldots , \cK J_{m_k} (z_k)
\bigr]  P \, .
$$
Say that $A^{\Phi}$, $A^{\Psi}$ are {\bf kinematically similar} if
$\cK A^{\Phi}$, $\cK A^{\Psi}$ are similar. Informally put, two
matrices are kinematically similar precisely if their Jordan normal
forms coincide, except possibly for the {\em imaginary\/} parts of
their eigenvalues. Just like its Lipschitz
counterpart, the notion of kinematic similarity for matrices is
well-defined. It also is highly suggestive in that, as the
reader may know or suspect already, $\Phi$, $\Psi$ are kinematically
similar precisely if $A^{\Phi}$,
$A^{\Psi}$ are; see \cite[Prop.\ 4.2]{KS}. Moreover, it is readily seen
that, as a refinement of (\ref{eq2z2}),
$$
A^{\Phi}, A^{\Psi} \enspace \mbox{\rm Lipschitz similar} \quad
\Longrightarrow \quad
A^{\Phi}, A^{\Psi} \enspace \mbox{\rm kinematically similar} \quad
\Longrightarrow \quad
A^{\Phi}, A^{\Psi} \enspace \mbox{\rm Lyapunov similar}  \, ;
$$
here the left and right implication cannot be reversed in general
for $d\ge 4$ and $d\ge 2$ respectively, though both implications are reversible
whenever $\Phi$, $\Psi$ are diagonal.

As alluded to earlier, pointwise Lipschitz {\em conjugacy\/}
behaves quite differently from what one might expect based on Theorem
\ref{thm6_4}. The proof of Theorem \ref{thm6_8} builds on the
following special case.

\begin{lem}\label{lem6_6}
Given $m\in \N$ and $a,b\in \R\setminus \{0\}$, let $\Phi$, $\Psi$ be
the flows on $\R^{2m}$ generated by $J_m(a+ib)$, $\mbox{\rm diag}\,
[J_m(a), J_m(a)]$ respectively. Then $\Phi \stackrel{{\sf pw}1}{\cong}\Psi$.
\end{lem}

\begin{proof}
Mimicking Step I of the proof of Theorem
\ref{thm6_4}, choose $G$ so that $t\mapsto \|\Phi_t
x\|\equiv \|\Psi_t x\|$ is strictly monotone for every $x\in
\R^{2m}\setminus \{0\}$; for instance, one may choose $G = \mbox{\rm
  diag}\, [1,g, \ldots , g^{m-1}, 1,g , \ldots , g^{m-1}]$ with
$g\in \R^+$ large enough (depending on $m$ and $a$). Defining
$T = T(x)$ accordingly, let $h(0)=0$ and
$$
h(x) = \Psi_{-T(x)} \Phi_{T(x)} x \qquad \forall x \in \R^{2m}
\setminus \{0\} \, .
$$
Recall that by (\ref{eq2z1}),
$$
\Phi_t = e^{at} R_m(bt) e^{t K_m(b)} \, , \quad \Psi_t = e^{at}
e^{tK_m (b)} \qquad \forall t \in \R \setminus \{0\} \, ,
$$
and consequently, since $R_m(bt)$, $K_m(b)$ commute,
\begin{equation}\label{eq6_p20}
h(x) = \left\{
  \begin{array}{ll}
    x & \mbox{\rm if $x=0$ or $\|x\|=1$} \, , \\
    R_m\bigl( bT(x) \bigr) x & \mbox{\rm otherwise} \, .
    \end{array}
  \right.
\end{equation}
Note that $h$ is continuous on $\R^{2m}$, smooth on $\R^{2m}\setminus
\{0\}$, and $|h(x)|=|x|$ for all $x\in \R^{2m}$. Moreover, it is readily
seen that $h$ is one-to-one and onto, with
$$
h^{-1}(x) = R_m \bigl( - bT(x) \bigr)x \qquad \forall x\in
\R^{2m}\setminus \{0\} : \|x\|\ne 1 \, ,
$$
and hence $h\in \cH_{{\sf pw}1}(\R^{2m})$. Since $T(\Phi_tx) =
T(x) - t$ for all $t\in \R$, $x\in \R^{2m}\setminus \{0\}$,
$$
h(\Phi_t x) = \Psi_{t - T(x)} \Phi_{T(x) - t} \Phi_t x = \Psi_t h(x)
\qquad \forall t \in \R \, ,
$$
and the two outer-most expressions agree also if $x=0$. In summary,
$\Phi \stackrel{h}{\cong} \Psi$ with $h\in \cH_{{\sf pw}1}(\R^{2m})$.
\end{proof}

\begin{rem}\label{rem6_7}
In the proof of Lemma \ref{lem6_6}, notice that for $m=1$ one may choose
$G=I_2$. Then $T(x) = -
a^{-1}\log |x|$, and (\ref{eq6_p20}) yields {\em exactly\/} the homeomorphism $h_{-b/a} \in
\cH_1 (\R^2)$ encountered in (\ref{eq5_add20}). By contrast, $J_m(a+ib)$,
$\mbox{\rm diag}\, [J_m(a), J_m(a)]$ are not Lipschitz similar
whenever $m\ge 2$.
In this case, by Theorem \ref{thm1x} the conjugacy $h\in \cH_{{\sf
    pw}1}(\R^{2m})$ defined in (\ref{eq6_p20}), though smooth on $\R^{2m}\setminus \{0\}$,
cannot satisfy a (uniform) Lipschitz condition near $0$. To see this
directly also, choose a suitable diagonal $G$ and consider for instance
$$
y_n = ne^{-n} e_1 + a e^{-n} e_2\, , \quad z_n = n e^{-n} e_1 - a
e^{-n} e_2 \qquad \forall n \in \N \, .
$$
A short calculation yields
$$
\limsup\nolimits_{n\to \infty} \frac{|h(y_n) - h(z_n)|}{n|y_n - z_n|}
= \frac{1}{|a|} \, , 
$$
and hence indeed $h\not \in \cH_1(\R^{2m})$. 
\end{rem}

\begin{theorem}\label{thm6_8}
Let $\Phi$, $\Psi$ be linear flows on $X$. Then the following
statements are equivalent:
\begin{enumerate}
\item $\Phi \stackrel{{\sf pw}1}{\cong}\Psi$;
  \item $\Phi$, $\Psi$ are kinematically similar while $\Phi_{\sf C}
    \stackrel{{\sf lin}}{\cong} \Psi_{\sf C}$;
    \item $A^{\Phi}$, $A^{\Psi}$ are kinematically similar while
      $A^{\Phi_{\sf C}}, A^{\Psi_{\sf C}}$ are similar.
\end{enumerate}
\end{theorem}

\begin{proof}
Since none of the statements is affected at all by linear conjugacies,
assume w.l.o.g.\ that $\Phi$, $\Psi$ are generated by
$$
\mbox{\rm diag}\, \bigl[
J_{m_1}(z_1), \ldots, J_{m_k}(z_k), A^{\Phi_{\sf C}}
\bigr] \, , \quad
\mbox{\rm diag}\, \bigl[
J_{n_1}(w_1), \ldots, J_{n_{\ell}}(w_{\ell}), A^{\Psi_{\sf C}}
\bigr] 
$$
respectively, with $k,\ell\in \N_0$, $m_1, \ldots, m_k, n_1, \ldots ,
n_{\ell}\in \N$, and $z_1, \ldots, z_k, w_1, \ldots , w_{\ell}\in
\C^-\cup \C^+$. For every $j\in \{1,\ldots , k\}$ take $h_{j}\in
\cH_{{\sf pw}1}(\R^{m_j \, {\sf d} ({\sf Im}\, z_j)})$ either as in
Lemma \ref{lem6_6} (if $z_j\not \in \R$) or else $h_j = I_{m_j}$ (if
$z_j\in \R$). With this, $\Phi \stackrel{h}{\cong}\widetilde{\Phi}$
where
$$
h = \bigtimes_{j=1}^k h_j \times I_{d_{\sf C}^{\Phi}} \in \cH_{{\sf
    pw}1} (\R^d) \, ,
$$
and $\widetilde{\Phi}$ is generated by
$$
\widetilde{A}:= \mbox{\rm diag}\, \bigl[
\cK J_{m_1}(z_1), \ldots, \cK J_{m_k}(z_k), A^{\Phi_{\sf C}}
\bigr] \, .
$$
Similarly, $\Psi \stackrel{{\sf pw}1}{\cong}\widetilde{\Psi}$, with
$\widetilde{\Psi}$ generated by
$$
\widetilde{B}:= \mbox{\rm diag}\, \bigl[
\cK J_{n_1}(w_1), \ldots, \cK J_{n_{\ell}}(w_{\ell}), A^{\Psi_{\sf C}}
\bigr] \, .
$$
By means of $\widetilde{\Phi}$, $\widetilde{\Psi}$ it will now be shown
that (i)$\Leftrightarrow$(iii). That (ii)$\Leftrightarrow$(iii) is
clear from the fact that, as recalled earlier, $\Phi$, $\Psi$ are
kinematically similar precisely if $A^{\Phi}$, $A^{\Psi}$ are.

To prove (i)$\Rightarrow$(iii) assume that $\Phi \stackrel{{\sf
    pw}1}{\cong}\Psi$. Then $\widetilde{\Phi}
\stackrel{\widetilde{h}}{\cong}\widetilde{\Psi}$ for some
$\widetilde{h}\in \cH_{{\sf pw}1}(\R^d)$ as well. Consider first
the case of stable $\Phi$, $\Psi$, so $\widetilde{\Phi}$,
$\widetilde{\Psi}$ are stable also. For every $m\in \N_0$ the function
$\ell_m^{\widetilde{\Phi}}:\R\to \N_0$ given by
$\ell_m^{\widetilde{\Phi}}(s) = \dim L_m^{\widetilde{\Phi}} (s)$ for
all $s\in \R$ is non-decreasing with $\lim_{s\to -\infty}
\ell_m^{\widetilde{\Phi}}(s)= 0$ and $\ell_m^{\widetilde{\Phi}}(0^-) =
d$, just like $\ell^{\widetilde{\Phi}}$, but unlike the latter it is
not in general right-continuous. In analogy to (\ref{eq4_6BA}),
$$
\widetilde{h} \bigl( L_m^{\widetilde{\Phi}} (s)\bigr) =
L_m^{\widetilde{\Psi}} (s) \qquad \forall m\in \N_0, s \in \R \, ,
$$
due to $\widetilde{\Phi}$,
$\widetilde{\Psi}$ being {\em conjugate\/} rather than merely equivalent.
It follows that $\ell_m^{\widetilde{\Phi}} =
\ell_m^{\widetilde{\Psi}}$, and hence for every $m\in \N$, $s\in \R$,
\begin{align*}
\sum\nolimits_{j: {\sf Re}\, z_j = s, m_j =m} {\sf d} ({\sf Im}\, z_j)
& = - \ell_{m+1}^{\widetilde{\Phi}} (s) + 2 \ell_m^{\widetilde{\Phi}}(s)
- \ell_{m-1}^{\widetilde{\Phi}}(s) \\
& =
- \ell_{m+1}^{\widetilde{\Psi}} (s) + 2 \ell_m^{\widetilde{\Psi}}(s)
                                                                         - \ell_{m-1}^{\widetilde{\Psi}}(s)  =
\sum\nolimits_{j: {\sf Re}\, w_j = s, n_j =m} {\sf d} ({\sf Im}\,
w_j)\, .
\end{align*}
The left- and right-most expressions are precisely the total numbers
of all Jordan blocks $J_m(s)$ appearing in $\widetilde{A}$,
$\widetilde{B}$ respectively. As these numbers agree for every $m\in
\N$, $s\in \R$, clearly $\widetilde{A}$, $\widetilde{B}$ are
similar. In other words, $A^{\Phi}$, $A^{\Psi}$ are kinematically
similar.

It remains to consider the case of arbitrary $\Phi$, $\Psi$. In this
case $\widetilde{\Phi} \stackrel{0}{\cong} \widetilde{\Psi}$, so
$A^{\Phi_{\sf C}}$, $A^{\Psi_{\sf C}}$ are similar; see, e.g.,
\cite[Thm.\ 1.1]{BW2}. Moreover, $\widetilde{\Phi}_{\bullet}
\stackrel{{\sf pw}1}{\cong}\widetilde{\Psi}_{\bullet}$ for $\bullet
\in \{{\sf S}, {\sf U}\}$, and applying the previous argument to
$\widetilde{\Phi}_{\sf S}$, $\widetilde{\Psi}_{\sf S}$ as well as
$\widetilde{\Phi}_{\sf U}^*$, $\widetilde{\Psi}_{\sf U}^*$ shows that
$A^{\Phi_{\sf S}}$, $A^{\Psi_{\sf S}}$ as well as $A^{\Phi_{\sf U}}$,
$A^{\Psi_{\sf U}}$ are kinematically similar. Thus $A^{\Phi}$,
$A^{\Psi}$ are kinematically similar, too. In summary, this proves that
(i)$\Rightarrow$(iii).

To see that (iii)$\Rightarrow$(i) as well, simply note that (iii), by the definition of
kinematic similarity, implies that
$\widetilde{A}$, $\widetilde{B}$ are similar, and so $\Phi \stackrel{{\sf
    pw}1}{\cong}\widetilde{\Phi} \stackrel{{\sf
    lin}}{\cong}\widetilde{\Psi}  \stackrel{{\sf
    pw}1}{\cong} \Psi$, which yields (i).
\end{proof}

Theorems \ref{thm6_4} and \ref{thm6_8} together highlight one
remarkable property common to most forms of equivalence considered in this
article and elsewhere, with the notable exception of pointwise Lip\-schitz
equivalence. The property, alluded to already in the Introduction, is
this: If $\Phi
\stackrel{\bigstar}{\thicksim} \Psi$ for any $\bigstar \in \{0,1^-,1 ,
{\sf diff}, {\sf lin}\}$, then also $\Phi
\stackrel{\bigstar}{\cong}\Psi_{*\alpha}$ for some $\alpha \in
\R\setminus \{0\}$. That ``equivalence means conjugacy up to
rescaling'' appears to have been a largely unquestioned tenet of linear
systems folklore, expressed explicitly, e.g., in \cite[Sec.\ 2.5]{CK}
and \cite[Rem.\ 7.4]{Willems}. To this, the case of pointwise Lipschitz equivalence
offers a welcome antidote: For instance, with $\Phi$, $\Psi$
generated by $I_2$, $J_2(1)$ respectively, $\Phi
\stackrel{{\sf pw}1}{\thicksim}\Psi$ by Theorem \ref{thm6_4}, and yet $\Phi \cancel{ \stackrel{{\sf
    pw}1}{\cong}} \Psi_{*\alpha}$ for every $\alpha \in \R\setminus
\{0\}$ by Theorem \ref{thm6_8}.

\subsubsection*{Acknowledgements}

The first author was partially supported by an {\sc Nserc} Discovery  
Grant.

  



\end{document}